\documentclass{birkjour}
 \newtheorem{thm}{Theorem}[section]
 \newtheorem{cor}[thm]{Corollary}
 \newtheorem{lem}[thm]{Lemma}
 
 \theoremstyle{definition}
 \newtheorem{defn}[thm]{Definition}
 \theoremstyle{remark}
 \newtheorem{rem}[thm]{Remark}
 
 \numberwithin{equation}{section}

 \usepackage{amsmath, amssymb, amsfonts}
 \usepackage{mathrsfs}
 \usepackage{hyperref}
 
 \def\D{\mathcal{D}}
 \def\d{\mathrm{d}}
 \def\rn{ {\mathbb R ^n } }

 \def\M{\mathcal{M}}

\begin{document}

%
%
%
%
%
%
%
%
%

\title[matrix weighted Bourgain-Morrey  Triebel-Lizorkin spaces]
 {On matrix weighted Bourgain-Morrey  Triebel-Lizorkin spaces}

\author[Tengfei Bai]{Tengfei Bai}

\address{%
	Longkunnan 99,
	College of Mathematics and Statistics, Hainan Normal University, Haikou, Hainan 571158,
	China
}

\email{202311070100007@hainnu.edu.cn}

\thanks{The corresponding author	Jingshi Xu is supported by the National Natural Science Foundation of China (Grant No. 12161022) and the Science and Technology Project of Guangxi (Guike AD23023002).
	Pengfei Guo is supported by the National Natural Science Foundation of China (12561002).}
\author{Pengfei Guo}
\address{Longkunnan 99,
	College of Mathematics and Statistics, Hainan Normal University, Haikou, Hainan 571158,
	China}
\email{050116@hainnu.edu.cn}

\author{Jingshi Xu}
\address{Jinji 1,
	School of Mathematics and Computing Science, Guilin University of Electronic Technology, Guilin 541004, China\\
	Center for Applied Mathematics of Guangxi (GUET), Guilin 541004, China\\
	Guangxi Colleges and Universities Key Laboratory of Data Analysis and Computation, Guilin 541004, China
}
\email{jingshixu@126.com}




\begin{abstract}
We introduce the homogeneous (inhomogeneous) matrix weighted Bourgain-Morrey  Triebel-Lizorkin
spaces and obtain their equivalent  norms. We also obtain
their characterizations by Peetre type maximal functions, Lusin-area
function, Littlewood-Paley $g_{\lambda}^{*}$-function, approximation,   wavelet and atom. As an application, we obtain boundedness of pseudo-differential
operators with symbols in the H\"{o}rmander classes and H\"{o}lder-Zygmund
classes on inhomogeneous matrix weighted Bourgain-Morrey Triebel-Lizorkin spaces.
\end{abstract}

\maketitle
\tableofcontents
\section{Introduction}
In the last three decades, the theory of matrix weighted function spaces
has been developed. Indeed, Muckenhoupt $\mathcal A_{2}$ matrix weights were
introduced by Treil and Volberg \cite{TV97} in 1997. For $p\in(1,\infty)$,
Goldberg \cite{Goldberg03} established the $\mathcal  A_{p}$ condition in
terms of matrix weights. He also showed that the matrix $\mathcal A_{p}$ condition
implies $L^{p}$-boundedness of the Hardy-Littlewood maximal operator.
In \cite{CMR16},  Cruz-Uribe et al.  studied degenerate Sobolev spaces where the degeneracy is controlled by a
matrix $\mathcal A_p$ weight. As applications, they obtained the weak solutions of degenerate $p$-Laplace equations and mappings of finite distortion.
In \cite{FrazierR04,FraRou19, R03,Roudenko04}, Frazier and Roudenko
introduced the matrix weight class $\mathcal A_{p}$ for $p\in(0,1]$ and matrix
weighted Besov and Triebel-Lizorkin spaces. For matrix weight $W\in \mathcal A_{p},$
they showed that homogeneous (and inhomogeneous) matrix weighted Triebel-Lizorkin
spaces and Besov spaces are equivalent with their discrete spaces
and the corresponding averaging spaces. For $p\in(1,\infty)$ and $W\in \mathcal  A_{p}$, they showed that
matrix weighted Lebesgue space $L^{p}(W)=\dot{F}_{p}^{0,2}(W)$ and
that the matrix weighted Sobolev space $L_{k}^{p}(W)=F_{p}^{k,2}(W)$. They also obtained the boundedness of Calder\'{o}n-Zygmund
operators.

After this, matrix weighted function spaces  received 
more and more attentions.
In \cite{WYZ22}, Wang, Yang and Zhang studied the characterizations
of homogeneous matrix weighted Triebel-Lizorkin spaces. 
In \cite{BYY23}, Bu, Yang and Yuan  introduced the homogeneous matrix-weighted Besov
spaces  on a  homogeneous space and showed  these spaces are  independent of the choices of both approximations of the identity. They established   characterizations of these spaces in terms of  wavelet, molecule. The boundedness of almost diagonal operators and Calder\'{o}n-Zygmund operators  was also proved.
In \cite{BHYY3, BHYY1,BHYY2},
Bu et al. introduced the matrix-weighted
Besov-type spaces $\dot{B}_{p,q}^{s,\tau}(W)$, Triebel-Lizorkin-type spaces $\dot{F}_{p,q}^{s,\tau}(W)$ and the related sequence spaces. They proposed a new concept of $\mathcal A_{p}$-dimension
of matrix weights. Then they obtained the boundedness of $\varphi$-transform,
pseudo-differential operators, trace operators and Calder\'{o}n--Zygmund
operators.
In \cite{LYY24}, Li, Yang and Yuan introduced the matrix-weighted Besov-Triebel-Lizorkin spaces with logarithmic smoothness. They obtained an equivalent characterization of these spaces in terms of the matrix-weighted Peetre-type maximal functions and the boundedness of some pointwise multipliers on these spaces.

In \cite{V97}, Volberg introduced  the original definition of  matrix-valued weight classes $\mathcal A_{p,\infty}$.
In \cite{BHYY23}, Bu et al.  obtained new characterizations of matrix weights classes $\mathcal A_{p,\infty}$. Next
in \cite{BHYY25}, Bu et al.  characterized inhomogeneous Besov-type and Triebel-Lizorkin-type spaces with matrix weight $W \in \mathcal A_{p,\infty}$ via the $\varphi$-transform, molecules, and wavelets. The boundedness of almost diagonal operators, pseudo-differential operators, trace operators, pointwise multipliers, and Calder\'on-Zygmund operators was also proved.
In \cite{BCYY25}, Bu et al. introduced the  matrix-weighted Hardy spaces and characterized them in terms of various  maximal functions and atoms.
In \cite{YYZ25}, Yang, Yuan and Zhang proved the boundedness of pseudo-differential, trace, and Calder\'on-Zygmund operators on generalized matrix-weighted Besov-Triebel-Lizorkin-type spaces with matrix  $\mathcal A_{p, \infty} $ weights.
In \cite{BYYZ25}, Bu et al. introduced generalized matrix-weighted Besov-Triebel-Lizorkin-type spaces with matrix  $\mathcal A_{p, \infty}$ weights and obtained characterizations in terms of the $\varphi$-transform, the Peetre-type maximal function,  the Littlewood-Paley functions,  molecule and  wavelet. Using  matrix  $\mathcal A_{p, \infty}$ weights, in \cite{YYZ252}, Yang et al. introduced the matrix-weighted variable Besov space and  established the $\varphi$-transform theorem  for these spaces.

In \cite{WGX25}, Wang, Guo and Xu studied the approximation characterization, embedding properties and  duality of matrix weighted modulation spaces which was introduced by  Nielsen in \cite{N25}.  The  criterion for precompactness in the  matrix weighted {Bourgain}-{Morrey} spaces and  matrix-weighted variable Lebesgue spaces can be founded in \cite{BX243,WGX24}.
Many boundednesses of pseudo-differential operators on Triebel-Lizorkin spaces and Besov spaces can be found in \cite{CO22,D97,Mar96, Mar91, md,P16,P19,P20,sa,S09}. We will  recall them in Section \ref{sec:Pseudo-differential-operators}. In \cite{BX242, BX24}, the first author and the third author of this paper obtained the boundedness of pseudo-differential operators  on matrix weighted Besov and Triebel-Lizorkin spaces.

In \cite{HNS23}, Hatano, Nogayama, Sawano and Hakim investigated
the Bourgain-Morrey function spaces $M_{p}^{t,r}$(Definition \ref{def:BMspace}). They considered the properties such as inclusion, dilation, translation, nontriviality, diversity, approximation and density of Bourgain-Morrey  spaces.
They  obtained the boundedness of operators such as the Hardy-Littlewood maximal operator, fractional integral operators, fractional maximal operators, and singular integral operators on Bourgain-Morrey spaces. The dual of Bourgain-Morrey spaces $M_{p}^{t,r}$ with $ 1<p <t<r <\infty$ is given in \cite{HNS23}.
Bourgain-Morrey spaces have many properties while Morrey spaces
have not; see \cite[Chapter 8]{SFH20}. The Bourgain-Morrey space
are useful to study the Strichartz estimate and nonlinear Schr\"{o}dinger
equations; see \cite{C64,MV98,MVV99}. 
The development of  Bourgain-Morrey type function spaces can be seen in \cite{BGX252, BGX253, BX25, HLY23, ZYZ24, ZSTYY23}.
Motivated by above literature, we will introduce matrix weighted Bourgain-Morrey Triebel-Lizorkin spaces and consider their characterizations and the boundedness of pseudo-differential operators on them.

This paper is organized as follows. In Section \ref{sec:preliminaries},
we recall definitions of matrix $\mathcal A_{p}$ weight and Bourgain-Morrey
spaces . In Section \ref{sec:Homogeneous-spaces}, we introduce four kinds of homogeneous
matrix weighted Bourgain-Morrey  Triebel-Lizorkin spaces
and prove they are equivalent. Calder\'{o}n-Zygmund operators
are also contained in Section \ref{sec:Homogeneous-spaces}.
The results about inhomogeneous spaces are given in  Section \ref{sec:Inhomogeneous-spaces}.
In Section \ref{sec:Characterizations},
we give some characterizations of matrix weighted Bourgain-Morrey
type Besov and Triebel-Lizorkin spaces, such as Peetre type maximal functions, Lusin-area
functions, Littlewood-Paley $g_{\lambda}^{*}$-functions, wavelets and atoms.
In Section \ref{sec:Pseudo-differential-operators}, by a approximation characterization of inhomogeneous matrix weighted Bourgain-Morrey type Besov spaces and Triebel-Lizorkin spaces, we obtain the boundedness of pseudo-differential operators with
symbols in  H\"{o}rmander classes and H\"{o}lder-Zygmund
classes  on them, respectively.

\section{Preliminaries}\label{sec:preliminaries}
We  first make some conventions on notation.
For any cube $Q$ of $\rn$, let $c_Q$ be its center and $\ell (Q)$ its edge
length. For any $\lambda >0  $ and any cube $Q $ of $\rn$, let $\lambda Q$ be the cube with the same center
of $Q$ and the edge length $\lambda \ell (Q)$.
For $j\in\mathbb{Z}$, $m\in\mathbb{Z}^{n}$, let $Q_{j,m}:=\prod_{i=1}^{n}[2^{-j}m_{i},2^{-j}(m_{i}+1))$.
We
denote by $\mathcal{D}:=\{Q_{j,m} :j\in\mathbb{Z}$, $m\in\mathbb{Z}^{n}  \} $ the  family of all dyadic cubes in $\mathbb{R}^{n}$,
while $\mathcal{D}_{j}$ is the set of all dyadic cubes with $\ell(Q)=2^{-j},j\in\mathbb{Z}$.
Let $\chi_{E}$ be the characteristic function of the set $E\subset\mathbb{R}^{n}$.
Let $\mathbb N := \{ 1,2,3, \cdots\} $ and $ \mathbb{N}_{0}:=\mathbb{N\cup}\{0\}$. Let $\mathscr{S}(\mathbb{R}^{n})$
denote the Schwartz space on $\rn$, and let $\mathscr{S}'(\mathbb{R}^{n})$
be its dual. Let $\mathcal{P}(\mathbb{R}^{n})$ be the class of the
polynomials on $\mathbb{R}^{n}$. Denote by $\mathscr{S}'/\mathcal{P}(\mathbb{R}^{n})$
the space of tempered distributions modulo polynomials. Let $\mathscr{S}_{0}(\mathbb{R}^{n}):=\{f\in\mathscr{S}(\mathbb{R}^{n}):\partial^{\alpha}\mathcal{F}(g)(0)=0$
for all multi-indices $\alpha$\}. Recall that $\mathscr{S}'/\mathcal{P}(\mathbb{R}^{n})$
is the dual of $\mathscr{S}_{0}(\mathbb{R}^{n})$. We use the symbol
$A\lesssim B$ to denote that there exists a positive constant $c$
such that $A\le cB$. If $A\lesssim B$ and $B\lesssim A$, then we
denote $A\approx B$. The letters $c,C$ will denote various positive
constants and may change in different lines.
For $p\in (0,\infty)$, we define $p'$ by  $ 1/p + 1/p' = 1$.

\subsection{Matrix weights}

Let $m \in\mathbb{N}$. A matrix weight $W$ is a map on $\mathbb{R}^{n}$
such that $W(x)$ is a non-negative definite $m\times m$ matrix for
each $x$ $\in\mathbb{R}^{n}$, where $W$ is almost everywhere invertible
and the entries of $W$ are measurable functions on $\mathbb{R}^{n}$.
The operator norm of a matrix $A$ is defined by
\[
\|A\|:=\sup_{|\vec{z}|=1} |A\vec{z}| ,
\]
where $\vec{z}\in\mathbb{C}^{m}$ and $|\vec{z}|=\big(\sum_{i=1}^{m}|z_{i}|^{2}\big)^{1/2}$.

\begin{defn}
	Let $0< p<\infty $ and $W$ be a matrix weight. Let $\Omega \subset \rn$.  The matrix-weighted 	Lebesgue space $L^p(W, \Omega)$ is defined to be the set of all the measurable vector-valued functions $\vec f :\Omega \to \mathbb C^m$ such that
	\begin{equation*}
		\| \vec f \|_{L^p (W, \Omega)} := \| \vec f \chi_\Omega \|_{L^p (W)} := \left(\int_\rn |W^{1/p} (x) \vec f (x) |^p \chi_\Omega (x) \d x   \right)^{1/p} <\infty.
	\end{equation*}
	
\end{defn}

For $p\in(1,\infty)$, a matrix weight $W\in \mathcal A_{p}(\mathbb{\mathbb{R}}^{n})$
if and only if
\[
\sup_{Q}\frac{1}{|Q|}\int_{Q}\left(\frac{1}{|Q|}\int_{Q}\|W^{1/p}(x)W^{-1/p}(y)\|^{p'}\mathrm{d}y\right)^{p/p'}\mathrm{d}x<\infty,
\]
where $p'=p/(p-1)$ is the conjugate index of $p$, and the supremum
is taken over all cubes $Q\subset\mathbb{R}^{n}$.

For $p\in(0,1]$, a matrix weight $W\in \mathcal A_{p}(\mathbb{\mathbb{R}}^{n})$
if and only if
\[
\underset{Q}{\sup}\;\underset{y\in Q}{\mathrm{ess}\;\sup} \frac{1}{|Q|}\int_{Q}\|W^{1/p}(x)W^{-1/p}(y)\|^{p}\mathrm{d}x<\infty.
\]

We write $\mathcal A_{p}:= \mathcal A_{p}(\mathbb{\mathbb{R}}^{n})$ for brevity.

Given any matrix weight $W$ and $p\in(0,\infty)$, there exists (see
e.g., \cite[Proposition 1.2]{Goldberg03}) for $p>1$ and \cite[p.1237]{FrazierR04}
for $0<p\le1)$ a sequence $\{A_{Q}\}_{Q\in\mathcal{D}}$ of positive
definite $m \times m$ matrices such that
\begin{equation*}
	c_{1}|A_{Q}\vec{y}|\le\Big(\frac{1}{|Q|}\int_{Q}|W^{1/p}(x)\vec{y}|^{p}\mathrm{d}x\Big)^{1/p}\le c_{2}|A_{Q}\vec{y}|,
\end{equation*}
with positive constants $c_{1}$, $c_{2}$ independent of $\vec{y}\in\mathbb{C}^{m}$
and $Q\in\mathcal{D}$. In this case, we call $\{A_{Q}\}_{Q\in\mathcal{D}}$
a sequence of reducing operators of order $p$ for $W$.

\begin{lem}[Lemmas 3.2, 3.3, \cite{FraRou19}]
	\label{lem:W AQ} Suppose that $0<p<\infty$,
	$W\in \mathcal A_{p}$, and $\{A_{Q}\}_{ Q \in \mathcal D}$ is a sequence of reducing operators
	of order $p$ for $W$. Then there exist $\delta_W,C_{v}>0$ such that
	\[
	\sup_{Q}\frac{1}{|Q|}\int_{Q}\|W^{1/p}(x)A_{Q}^{-1}\|^{v}\mathrm{d}x\le C_{v}\;\mathrm{for}\;v<p+\delta_W,
	\]
	and
	\[
	\sup_{Q}\frac{1}{|Q|}\int_{Q}\sup_{P\in\mathcal{D}:x\in P\subset Q}\|W^{1/p}(x)A_{P}^{-1}\|^{v}\mathrm{d}x\le C_{v}\;\mathrm{for}\;v<p+\delta_W.
	\]
	Furthermore, for $1<p<\infty$, $p'=p/(p-1)$, we have
	\begin{equation}
		\sup_{Q}\frac{1}{|Q|}\int_{Q}\|A_{Q}W^{-1/p}(x)\|^{v}\mathrm{d}x\le C_{v}\;\mathrm{for}\;v<p'+\delta_W; \label{eq:p>1 AQW ^-1/p}
	\end{equation}
	for $0<p\le1$, we have
	\begin{equation}
		\sup_{Q}\underset{x\in Q}{\mathrm{ess}\;\sup}\|A_{Q}W^{-1/p}(x)\|<\infty.\label{eq:p le 1 AQW}
	\end{equation}
\end{lem}

\subsection{Bourgain-Morrey spaces}
In \cite{Bou91}, Bourgain considered the special case of Bourgain-Morrey
spaces. We use the notation in \cite{MS18} by Masaki and Segata.
\begin{defn}
	\label{def:BMspace} Let $0<p\le t<\infty$
	and $0<r\le\infty$. Define $M_{p}^{t,r} := M_{p}^{t,r}  (\rn)$ as the space of $f\in L_{\mathrm{loc}}^{p} (\rn)$
	such that
	\[
	\|f\|_{M_{p}^{t,r}} :=\bigg\|\Big\{|Q|^{1/t-1/p} \| f\chi_Q \|_{L^p}  \Big\}_{ Q\in \D}\bigg\|_{\ell^{r}}<\infty.
	\]
	
	For $0 <q \le \infty$, define $M_{p}^{t,r} (\ell^q)$ as the set of all sequences  $\{f _j \} _{j\in \mathbb Z}  \subset L^p_{\operatorname{loc}} (\rn)$ for which
	\begin{equation*}
		\|\{f _j \} _{j\in \mathbb Z} \|_{ M_{p}^{t,r} (\ell^q)  } := \left\|   \| \{f _j \}\|_{\ell^q}       \right\|_{M_{p}^{t,r}} <\infty .
	\end{equation*}
\end{defn}

Theorem 2.10 in \cite{HNS23} says 
$M_{p}^{t,r}\neq\{0\}$ if and only if $0<p<t<r<\infty$ or $0<p\le t<r=\infty$.
Thus we always assume $0<p<t<r<\infty$ or $0<p\le t<r=\infty$ below.

Let $f\in L_{{\rm loc}}^{1}(\mathbb{R}^{n})$. Then the standard Hardy-Littlewood
maximal function of $f$ is defined by
\[
\mathcal{M}f(x):=\sup_{B\ni x}\frac{1}{|B|}\int_{B}|f(y)|dy,\ \forall x\in\mathbb{R}^{n},
\]
where the supremum is taken over all balls containing $x$ in $\mathbb{R}^{n}$.

The following lemma coming from \cite[Lemma 4.1, Theorem 4.3]{HNS23} is the boundedness of Hardy-Littlewood maximal operator on Bourgain-Morrey spaces.

\begin{lem}
	\label{lem:hardy} Let $1<p<t<r<\infty$ or
	$1<p\le t<r=\infty$. Let $1<q\le\infty$. Then the Hardy-Littlewood maximal operator $\mathcal{M}$
	is bounded on $M_{p}^{t,r}$ and $M_{p}^{t,r}(\ell^{q})$.
\end{lem}

Now we turn to  vector functions. Let
$\vec{f}=(f_{1},\ldots,f_{m})^{T}$ with $f_{i}$ is a measurable
function on $\mathbb{R}^{n}$ for $i=1,\ldots,m$ where $(f_{1},\ldots,f_{m})^{T}$
denotes the transpose of $(f_{1},\ldots,f_{m})$.
Throughout the paper, for a function space $X$, $\vec f \in X$ means that $f_i \in X$ with $i =1,2,\ldots, m$.
\begin{defn}
	Let $0<p<t<r<\infty$ or $0<p\le t<r=\infty$. Let $W$ be a matrix
	weight.
	The matrix weighted Bourgain-Morrey space is defined by
	\[
	M_{p}^{t,r}(W) :=\Big\{\vec{f}\in L_{\mathrm{loc}}^{p}(W):\|\vec{f}\|_{M_{p}^{t,r}(W)}<\infty\Big\},
	\]
	where
	\[
	\|\vec{f}\|_{M_{p}^{t,r}(W)}:=\Big\|\Big\{|Q|^{1/t-1/p} \| W^{1/p} \vec f \chi_Q \|_{L^p}  \Big\}_{Q\in\mathcal{D}}\Big\|_{\ell^{r}}.
	\]
	
	Denote by $\mathcal{F}\vec{f} (\xi): = \mathcal{F}(\vec{f})(\xi):=\int_{\mathbb{R}^{n}}\vec{f}(x)e^{-2\pi ix\cdot\xi}\mathrm{d}x$ 	the Fourier transform of $\vec{f}$ and let $\mathcal{F}^{-1} \vec{f} (x):= \mathcal{F}^{-1}(\vec{f})(x):=\int_{\mathbb{R}^{n}}\vec{f}(\xi)e^{2\pi ix\cdot\xi}\mathrm{d} \xi$
	stand for the inverse Fourier transform of $\vec{f}$. 
	The Fourier
	multiplier $\varphi(D)$ with symbol $\varphi$ is defined by
	\[
	\varphi(D)\vec{f}(x) : =\mathcal{F}^{-1}[\varphi\mathcal{F}(\vec{f})](x).
	\]
\end{defn}

\section{Homogeneous matrix weighted Bourgain-Morrey Triebel-Lizorkin spaces}\label{sec:Homogeneous-spaces}
Let $\varphi\in\mathscr{S}(\mathbb{R}^{n}).$ We say that $\varphi$
is admissible, and write $\varphi\in\mathcal{A}$, if
\[
\mathrm{supp}\;\mathcal{F}(\varphi)\subset\{\xi:1/2\le|\xi|\le2\},
\]
\[
|\mathcal{F}(\varphi)(\xi)|\ge c>0\quad\mathrm{if}\quad3/5\le|\xi|\le5/3.
\]
For $j\in\mathbb{Z}$, let $\varphi_{j}(x)=2^{jn}\varphi(2^{j}x)$.
Then $\mathcal{F}(\varphi_{j})(\xi)=\mathcal{F}(\varphi)(2^{-j}\xi)$.

For the following function spaces, we suppose that $s\in\mathbb{R},$
$0<q\le\infty$, $W$ is a matrix weight,  $0<p<t<r<\infty$ or
$0<p\le t<r=\infty$. 

(i) The matrix weighted Bourgain-Morrey Triebel-Lizorkin space $\dot{F}_{p,t,r}^{s,q}(W)$
is the set of all distributions $\vec{f}\in\mathscr{S}'/\mathcal{P}(\mathbb{R}^{n})$
such that
\[
\|\vec{f}\|_{\dot{F}_{p,t,r}^{s,q}(W)}:=
\bigg\|\bigg(\sum_{v=-\infty}^{\infty}2^{vsq}|W^{1/p}\varphi_{v}*\vec{f}|^{q}\bigg)^{1/q}\bigg\|_{M_{p}^{t,r}}<\infty.
\]

(ii) The discrete Bourgain-Morrey Triebel-Lizorkin space $\dot{f}_{p,t,r}^{s,q}(W)$
is the set of all sequences $\vec{s}=\{\vec{s}_{Q}\}_{Q\in\mathcal{D}}$
such that
\[
\|\vec{s}\|_{\dot{f}_{p,t,r}^{s,q}(W)}:=
\bigg\|\bigg(\sum_{Q\in\mathcal{D}}\big[|Q|^{-s/n-1/2}|W^{1/p}\vec{s}_{Q}|\chi_{Q}\big]^{q}\bigg)^{1/q}\bigg\|_{M_{p}^{t,r}}<\infty.
\]

Suppose that for each $Q\in\mathcal{D}$, $A_{Q}$ is a $m\times m$
non-negative definite matrix with $m \in \mathbb N$.

(iii) The \{$A_{Q}$\}-Bourgain-Morrey Triebel-Lizorkin space $\dot{F}_{p,t,r}^{s,q}(A_{Q})$
is the set of all distributions $\vec{f}\in\mathscr{S}'/\mathcal{P}(\mathbb{R}^{n})$
such that
\[
\|\vec{f}\|_{\dot{F}_{p,t,r}^{s,q}(A_{Q})}:=\bigg\|\bigg(\sum_{v=-\infty}^{\infty}\sum_{Q\in\mathcal{D}_{v}}2^{vsq}|A_{Q}\varphi_{v}*\vec{f}|^{q}\chi_{Q}\bigg)^{1/q}\bigg\|_{M_{p}^{t,r}}<\infty.
\]

(iv) The \{$A_{Q}$\}-discrete Bourgain-Morrey Triebel-Lizorkin space
$\dot{f}_{p,t,r}^{s,q}(A_{Q})$ is the set of all sequences $\vec{s}=\{\vec{s}_{Q}\}_{Q\in\mathcal{D}}$
such that
\[
\|\vec{s}\|_{\dot{f}_{p,t,r}^{s,q}(A_{Q})}:=\bigg\|\Big(\sum_{v=-\infty}^{\infty}\sum_{Q\in\mathcal{D}_{v}}\Big[|Q|^{-s/n-1/2}|A_{Q}\vec{s}_{Q}|\chi_{Q}\Big]^{q}\Big)^{1/q}\bigg\|_{M_{p}^{t,r}}<\infty.
\]

For $\varphi\in\mathcal{A},$ let
\[
\mathcal{F}\psi=\frac{\mathcal{F}\varphi}{\sum_{v\in\mathbb{Z}}|\mathcal{F}\varphi_{v}|^{2}}.
\]
Then $\psi\in\mathcal{A}$ and we have
\begin{equation*}
	\sum_{v\in\mathbb{Z}}\overline{\mathcal{F}\varphi_{v}(\xi)}  \mathcal F \psi_{v}(\xi)=1\mathrm{\;for\;all}\;\xi\neq 0 . 
\end{equation*}
For $Q=Q_{v,k}$, $v \in \mathbb Z, k \in \mathbb Z^n$, let
\begin{equation}
	\varphi_{Q}(x):=2^{vn/2}\varphi(2^{v}x-k)=|Q|^{-1/2}\varphi((x-x_{Q})/\ell(Q))\label{eq:varphi  Q}
\end{equation}
where $x_{Q}$ is the lower left corner of $Q$ ($x_Q = 2^{-v}k$). $\psi_{Q}$ is similar.
For a vector valued functions $\vec{f}$ and test function $g$, we
define $\langle\vec{f},g\rangle=(\langle f_{1},g\rangle,\ldots,\langle f_{d},g\rangle)^{T}$.
Then we have
\begin{equation}
	\vec{f}=\sum_{Q\in\mathcal{D}}\langle\vec{f},\varphi_{Q}\rangle\psi_{Q}\label{eq:converge}
\end{equation}
convergence in $\mathscr{S}'/\mathcal{P}(\mathbb{R}^{n})$. This is
the so called $\varphi$-transform; see \cite{FJW91,FrazierR04,FraRou19,R03}.

\subsection{Equivalence of quasi-norms}
First, we consider the relation of spaces $\dot{F}_{p,t,r}^{s,q}(W)$ and $\dot{F}_{p,t,r}^{s,q}(A_{Q})$.

\begin{thm}
	\label{thm:equivalence} Let $0<p<t<r<\infty$ or $0<p\le t<r=\infty$.
	Let $s\in\mathbb{R}$, $W\in \mathcal A_{p}$ and
	$\{A_{Q}\}_{Q\in\mathcal{D}}$ be a sequence of reducing operators of order p for $W$. 
	Let $ q \in (0, p + \delta_W) $ where $\delta_W >0$  is the same as in Lemma \ref{lem:W AQ}. 
	For $\vec{f}\in\mathscr{S}'/\mathcal{P}(\mathbb{R}^{n})$, let $\vec{s}=\{\vec{s}_{Q}\}_{Q\in\mathcal{D}}$,
	where $\vec{s}_{Q}=\langle\vec{f},\varphi_{Q}\rangle.$ Then
	\[
	\|\vec{f}\|_{\dot{F}_{p,t,r}^{s,q}(W)}\approx\|\vec{f}\|_{\dot{F}_{p,t,r}^{s,q}(A_{Q})}\approx\|\vec{s}\|_{\dot{f}_{p,t,r}^{s,q}(A_{Q})}\approx\|\vec{s}\|_{\dot{f}_{p,t,r}^{s,q}(W)}.
	\]
	In addition, function spaces $\dot{F}_{p,t,r}^{s,q}(W)$, $\dot{F}_{p,t,r}^{s,q}(A_{Q})$
	are independent
	of the choice of $\varphi\in\mathcal{A},$ in the sense that different
	choices yield equivalent quasi-norms.

\end{thm}

To prove Theorem \ref{thm:equivalence}, we need 
many properties of matrix weights. 
We first recall  the concepts of the strongly and weakly doubling matrix which are from \cite[Definition 2.1]{FraRou19}
and \cite[Definition 1.3]{Roudenko04}. 
To fit this paper, we make some changes.

\begin{defn}
	Let $\{A_{Q}\}_{Q\in\mathcal{D}}$ be a sequence of positive definite
	matrices and let $p , \Delta \in(0,\infty)$, $d,\tilde d \in [0,\infty) $. We say that $\{A_{Q}\}_{Q\in\mathcal{D}}$
	is strongly doubling of order $(d,\tilde d, \Delta, p)$ if there exists $c>0$
	such that
	\[
	\|A_{Q}A_{R}^{-1}\| \le c\max \left\{   \left( \frac{\ell (R)}{ \ell (Q) } \right)^{d/p},  \left( \frac{\ell (Q)}{ \ell (R) } \right)^{\tilde d /p' }   \right\} \left( 1+ \frac{|c_Q - c_R| }{ \max \{ \ell(Q) , \ell (R) \}   }  \right) ^\Delta
	\]
	for all dyadic cubes $R,Q\in\mathcal{D}$. We say that $\{A_{Q}\}_{Q\in\mathcal{D}}$
	is weakly doubling of order $\Delta > 0$ if there exists $c>0$ such
	that
	\[
	\|A_{Q_{j,k}}A_{Q_{j,l}}^{-1}\|\le c(1+|k-l|)^{\Delta}
	\]
	for all $k,l\in\mathbb{Z}^{n}$ and all $j\in\mathbb{Z}$.
	
	Note that a strongly doubling sequence of order $(d,\tilde d, \Delta,p)$ is weakly
	doubling of order $\Delta$. 
\end{defn}

In \cite{BHYY1}, Bu, Hyt\"onen, Yang and Yuan introduced the concept of the $\mathcal A_p$-dimension of matrix weights and improved   \cite[Lemma 2.2]{FraRou19}.
We first recall the definition of $\mathcal A_p$-dimension $d$ from \cite[Definition 2.22]{BHYY1}.
\begin{defn}
	Let $p \in (0,\infty)$, $d \in \mathbb R$, and $W$  be a matrix weight. Then $W$ is said
	to have the $\mathcal A_p$-dimension $d$, denoted by $W \in\mathbb D_{p,d}:= \mathbb D_{p,d} (\mathbb R^n, \mathbb C ^m)$, if there exists a  positive constant $C$ such that, for any cube $Q \subset \rn$  and any $i \in \mathbb{N}_{0}$, when $p \in (0,1]$,
	\begin{equation*}
		\mathop{\mathrm{ess\,sup}}\limits_{y \in 2^i Q}   \frac{1}{|Q|} \int_Q \left\| W^{1/p} (x) W^{-1/p} (y) \right\|^p \d x \le C 2^{id}
	\end{equation*}
	or, when $p \in (1,\infty),$
	\begin{equation*}
		\frac{1}{|Q|} \int_Q \left( \frac{1}{|Q|} \int_{2^i Q}  \left\| W^{1/p} (x) W^{-1/p} (y) \right\| ^{ p'} \d y  \right)^{p /p'} \d x \le C 2^{id},
	\end{equation*}
	where $1/p +1/p' =1$.
\end{defn}	
The next lemma says that if $W \in  \mathcal A_p$, then $W$  naturally belongs to  $\mathbb D_{p,d}$ for some $d \in [0,n)$.

\begin{lem}
	[Proposition 2.26,  \cite{BHYY1}]
	Let $ p \in (0,\infty)$ and $W \in \mathcal A_p$. Then there exists $d \in [0,n)$  such that $W$ has the $ \mathcal A_p$-dimension $d$.
\end{lem}

The following result says that if  $W \in \mathcal A_p $ for  $p \in (0,\infty)$, then $\{A_Q\}_{\operatorname{cube} Q}$, the  family of reducing operators of order $p$ for $W$, is strongly doubling of order $(d,\tilde d, \Delta, p)$.

\begin{lem}[Lemma 2.28, \cite{BHYY1}] \label{AQ AR improved}
	Let $p \in (0,\infty)$, let $W \in \mathcal A_p $ have the $\mathcal A_p$-dimension $d \in [0,n)$, and let $\{A_Q\}_{\operatorname{cube} Q}$  be a family of reducing operators of order $p$ for $W$. If $p \in (1,\infty)$, let further $\widetilde W := W^{-1/(p-1)}$ (which belongs to $\mathcal A _{p'}$)  have the $ \mathcal A_{p'}$-dimension $\tilde d$, while, if $p \in (0,1]$, let $\tilde d =0$. 
	Let 
	\begin{equation} \label{Delta}
		\Delta := \frac{d}{p} + \frac{ \tilde{d} }{p'} .
	\end{equation} 
	Then there exists a constant $C>0$ such that, for any cubes $Q$ and $R$ of $\rn$,
	\begin{equation*}
		\left\| A_Q A_R^{-1} \right\| \le C \max \left\{   \left( \frac{\ell (R)}{ \ell (Q) } \right)^{d/p},  \left( \frac{\ell (Q)}{ \ell (R) } \right)^{\tilde d /p' }   \right\} \left( 1+ \frac{|c_Q - c_R| }{ \max \{ \ell(Q) , \ell (R) \}   }  \right) ^\Delta.
	\end{equation*} 
\end{lem}

Throughout the  paper, we fix some notations in Lemma \ref{AQ AR improved} since we will use it frequently. That is,
for $W \in \mathcal A_p$, the meanings of $d, \tilde d, \Delta$ is the same as in Lemma \ref{AQ AR improved}. 

\begin{rem}
	Let $W \in \mathcal A_p $ for $p \in (0,\infty)$ and $\{A_Q\}_{\operatorname{cube} Q}$  be a family of reducing operators of order $p$ for $W$.
	From Lemma \ref{AQ AR improved}, we get that $\{A_Q\}_{\operatorname{cube} Q}$ is a  weakly doubling sequence
	of order $\Delta$ of positive definite matrices where $\Delta$ is the same as in (\ref{Delta}).
\end{rem}

\begin{lem}[(2.8), \cite{FraRou19}]
	\label{lem:A inequality} Let $\varphi\in\mathcal{A}$
	and $\varphi_{j}(x)=2^{jn}\varphi(2^{j}x)$, $j\in\mathbb{Z}$. Suppose
	that $\{A_{Q}\}_{Q\in\mathcal{D}}$ is a weakly doubling sequence
	of order $\alpha > 0$ of positive definite matrices. Then, for any
	$A\in(0,1]$ and $R\in(0,\infty)$, there exists a positive constant
	c, depending on $\{A_{Q}\}_{Q\in\mathcal{D}}$, $A$ and $R$, such that for any $j\in\mathbb{Z}$,
	$k\in\mathbb{Z}^{n}$, and $\vec{f}\in\mathscr{S}'(\mathbb{R}^{n})$,
	\[
	\sup_{x\in Q_{j,k}}|A_{Q_{j,k}}\varphi_{j} * \vec{f}(x)|^{A}\le c\sum_{l\in\mathbb{Z}^{n}}(1+|k-l|)^{-A(R-\alpha)}2^{jn}\int_{Q_{j,l}}|A_{Q_{j,l}}\varphi_{j} * \vec{f}(z)|^{A}\mathrm{d}z.
	\]
\end{lem}

\begin{lem}[Lemma 3.7, \cite{WYZ22}]
	\label{lem:Littlewood max} Let $\eta>n$.
	Then there exists a positive constant c such that, for any $j\in\mathbb{Z}$
	and any complex-valued measurable function $g$ on $\mathbb{R}^{n}$,
	\[
	\sum_{k\in\mathbb{Z}^{n}}\sum_{l\in\mathbb{Z}^{n}}(1+|k-l|)^{-\eta}2^{jn}\int_{Q_{j,l}}|g(s)|\mathrm{d}s\chi_{Q_{j,k}}\le c\mathcal{M}(g).
	\]
\end{lem}

The following result is one direction of the equivalence of $\dot{F}_{p,t,r}^{s,q}(A_{Q})$  and $\dot{f}_{p,t,r}^{s,q}(A_{Q}) $.

\begin{thm}
	\label{thm:sup A_Q LE} Let $\varphi\in\mathcal{A}$ and $\tilde{\varphi}(x)=\overline{\varphi(-x)}$.
	Suppose that $\{A_{Q}\}_{Q\in\mathcal{D}}$ is a weakly doubling sequence
	(of any order $\Delta > 0$) of non-negative definite matrices. Let
	$0<p<t<r<\infty$ or $0<p\le t<r=\infty$. Let $s\in\mathbb{R}$,
	$0<q\le\infty$. Then there exists a constant $c>0$ depending only on $s,p,t,r,q,\Delta,\varphi,n$
	such that for all $\vec{f}\in\mathscr{S}'/\mathcal{P}(\mathbb{R}^{n})$,
	\begin{equation}
		\bigg\|\bigg(\sum_{j\in\mathbb{Z}}\sum_{Q\in\mathcal{D}_{j}}\bigg( 2^{js}\sup_{x\in Q}|A_{Q}\varphi_{j}*\vec{f}(x)|\chi_{Q}(x)\bigg)^{q}\bigg)^{1/q}\bigg\|_{M_{p}^{t,r}}\le c\|\vec{f}\|_{\dot{F}_{p,t,r}^{s,q}(A_{Q})},\label{eq:sup Aq le discrete}
	\end{equation}
	and
	\begin{equation}
		\|\{\langle\vec{f},\tilde{\varphi}_{Q}\rangle\}_{Q\in\mathcal{D}}\|_{\dot{f}_{p,t,r}^{s,q}(A_{Q})}\le c\|\vec{f}\|_{\dot{F}_{p,t,r}^{s,q}(A_{Q})}.\label{eq:discrete Aq le F}
	\end{equation}	
\end{thm}

\begin{proof}
	
	Let $0< A < \min\{1,p,q\}$ and $A(R- \Delta )>n$ where $\Delta$ is the same as in (\ref{Delta}).
	By Lemmas \ref{lem:A inequality}
	and \ref{lem:Littlewood max},  we have
	\begin{align*}
		& \sum_{Q\in\mathcal{D}_{j}}\Big( 2^{js}\sup_{x\in Q}|A_{Q}\varphi_{j}*\vec{f}(x)|\chi_{Q}(x)\Big)^{q}\\
		& \lesssim \bigg|\sum_{k\in\mathbb{Z}^{n}}\sum_{\ell\in\mathbb{Z}^{n}}(1+|k-l|)^{-A(R-\Delta)}2^{jn}\int_{Q_{j, \ell}}|2^{js}A_{Q_{j, \ell}}\varphi_{j}*\vec{f}(s)|^{A}\mathrm{d}s\chi_{Q_{j, k}} (x) \bigg|^{q/A}\\
		& \lesssim \mathcal{M}\bigg(\sum_{Q\in\mathcal{D}_{j}}\big(2^{js}|A_{Q}\varphi_{j}*\vec{f}|\chi_{Q}\big)^{A} \bigg)^{q/A} (x).
	\end{align*}
	Using Lemma \ref{lem:hardy}, we obtain (\ref{eq:sup Aq le discrete}).
	
	Since $|Q_{j, k}|^{-1/2}\langle\vec{f},\tilde{\varphi}_{Q_{j, k}}\rangle=\varphi_{j}*\vec{f}(x_{Q_{j, k}})$,
	(\ref{eq:sup Aq le discrete}) implies (\ref{eq:discrete Aq le F}).
\end{proof}
To estimate the converse of (\ref{eq:discrete Aq le F}), we recall
almost diagonal matrices (see, for instance, \cite[page 53]{FJ90}). 
\begin{defn}
	Let $0<p<\infty$, $0<q\le\infty$, $s\in\mathbb{R}$. A matrix $B=\{b_{QP}\}_{Q,P\in\mathcal{D}}$
	is almost diagonal, written $B\in\mathbf{ad}_{p}^{s,q}$, if there
	exist $\epsilon,c>0$ such that $|b_{QP}|\le c\omega_{QP}$ for all
	$Q,P\in\mathcal{D}$, where
	\begin{align}
		\omega_{QP} & : =\Big(\frac{\ell(Q)}{\ell(P)}\Big)^{s}\min\bigg\{\Big(\frac{\ell(P)}{\ell(Q)}\Big)^{(n+\epsilon)/2+n/\min(1,q,p)-n},\Big(\frac{\ell(Q)}{\ell(P)}\Big)^{(n+\epsilon)/2}\bigg\}\nonumber \\
		&
		\quad \times\bigg(1+\frac{|x_{Q}-x_{P}|}{\max\{\ell(Q),\ell(P)\}}\bigg)^{-\frac{n}{\min(1,q,p)}-\epsilon}.\label{eq:omega QP}
	\end{align}
\end{defn}

The following Lemma is  \cite[Remark A.3]{FJ90}.
\begin{lem}\label{lem:discrete HL general} 
	Suppose that $0<a\le r<\infty,$
	$\lambda>nr/a.$  
	For each dyadic cube $Q$ with side length $\ell (Q) = 2^{-\nu}$  and each $x \in Q$,
	\begin{align*}
	&	\bigg(\sum_{\ell(P)=2^{-\mu}}\frac{|s_{P}|^{r}}{(1+  \frac{|x_{P}-x_{Q}| }{ \max\{ \ell(P) , \ell (Q)\}} )^{\lambda}}\bigg)^{1/r} \\
		& \le C 2^{ (\mu - \nu)_+  n/a } \bigg\{\mathcal{M}\Big(\sum_{\ell(P)=2^{-\mu}}|s_{P}|^{a}\chi_{P}\Big)(x)\bigg\}^{1/a},
	\end{align*}
	where $C$ depends only on $n$  and $ \lambda - nr /a$; here $(\mu - \nu)_+ := \max (\mu - \nu, 0  ) $.
\end{lem}

A matrix \textbf{$B=\{b_{QP}\}_{Q,P\in\mathcal{D}}$ }acts on a sequence
$\vec{s}=\{\vec{s}_{Q}\}_{Q\in\mathcal{D}}$ by matrix multiplication
in each component: $B\vec{s}=\vec{t}=\{\vec{t}_{Q}\}_{Q\in\mathcal{D}}$,
where $\vec{t}_{Q}=\sum_{P\in\mathcal{D}}b_{QP}\vec{s}_{P}$.
\begin{thm}
	\label{thm:unweighted discrete almost} Let $0<p<t<r<\infty$ or
	$0<p\le t<r=\infty$. Let $0<q\le\infty$, $s\in\mathbb{R}$ and $B\in\mathbf{ad}_{p}^{s,q}$. Then $B$ defines a bounded operator
	on $\dot{f}_{p,t,r}^{s,q}$, where   $\dot{f}_{p,t,r}^{s,q}$ is   the space  $\dot{f}_{p,t,r}^{s,q}(W)$ with $W \equiv 1$.
\end{thm}

\begin{proof}
	We use the idea from \cite[Theorem 3.3]{FJ90}.		
	We assume $s=0$, since this case implies the general case. 
	Let $J=\min(1,q,p)-\delta$. Let $\delta>0$ be sufficiently small
	such that (\ref{eq:omega QP}) is still satisfied with $\min(1,q,p)$
	replaced by $J$. We write
	\[
	\vec{t}_{Q}=\sum_{\ell(P)\le\ell(Q)}b_{QP}\vec{s}_{P}+\sum_{\ell(P)>\ell(Q)}b_{QP}\vec{s}_{P}:=(B_{0}\vec{s})_{Q}+(B_{1}\vec{s})_{Q}.
	\]
	If $\ell(Q)=2^{-v}$ and $x\in Q$, use Lemma \ref{lem:discrete HL general} and we obtain
	\begin{align*}
		\sum_{\ell(P)>\ell(Q)} | b_{QP}\vec{s}_{P} | & \le\sum_{\mu<v}\sum_{P\in\mathcal{D}_{\mu}}2^{(\mu-v)(n+\epsilon)/2}(1+2^{\mu}|x_{Q}-x_{P}|)^{-n/J-\epsilon} | \vec{s}_{P} |\\
		& \le\sum_{\mu<v}2^{(\mu-v)(n+\epsilon)/2}\Big\{\mathcal{M}\Big(\sum_{P\in\mathcal{D}_{\mu}}|\vec{s}_{P}|^{J}\chi_{P}\Big)(x)\Big\}^{1/J}.
	\end{align*}
	Hence by Lemma \ref{lem:hardy} (since $J<\min\{1,p,q\}$), we have
	\begin{align*}
	&	\|B_{1}\vec{s}\|_{\dot{f}_{p,t,r}^{0,q}} \\
		& \le\bigg\|\bigg(\sum_{v=-\infty}^{\infty}\bigg( 2^{vn/2}\sum_{\mu<v}2^{(\mu-v)(n+\epsilon)/2}\bigg\{\mathcal{M}\bigg(\sum_{P\in\mathcal{D}_{\mu}}|\vec{s}_{P}|^{J}\chi_{P}\big)(x)\bigg\}^{1/J}\bigg)^{q}\bigg)^{1/q}\bigg\|_{M_{p}^{t,r}}\\
		& \lesssim\bigg\|\bigg(\sum_{\mu=-\infty}^{\infty}\sum_{P\in\mathcal{D}_{\mu}}\bigg(|P|^{-1/2}|\vec{s_{P}}|\chi_{P}\bigg)^{q}\bigg)^{1/q}\bigg\|_{M_{p}^{t,r}}.
	\end{align*}
	
	Now we estimate $B_{0}\vec{s} $.
	For $x\in Q$, by Lemma \ref{lem:discrete HL general}, we have
	\begin{align*}
	&	\sum_{\ell(P)\le\ell(Q)} | b_{QP}\vec{s}_{P} | \\
	& \le\sum_{\mu\ge v}\sum_{P\in\mathcal{D}_{\mu}}2^{(v-\mu)((n+\epsilon)/2+n/J-n)}(1+2^{v}|x_{Q}-x_{P}|)^{-n/J-\epsilon} |\vec{s}_{P} | \\
		& \lesssim \sum_{\mu\ge v}  2^{(v-\mu)((n+\epsilon)/2-n)}  \left(    \mathcal M \left(  \sum_{P \in \mathcal D _{\mu}}  |\vec s _P|^J \chi_P \right)  \right)^{1/J}.
	\end{align*}
	Hence 
	\begin{align*}
		& \bigg(\sum_{v=-\infty}^{\infty}\sum_{Q\in\mathcal{D}_{v}}\Big(|Q|^{-1/2}|(B_{0}\vec{s})_{Q}|\chi_{Q}\Big)^{q}\bigg)^{1/q} \\
		& \lesssim    \left(  \sum_{ \mu = -\infty}^\infty \sum_{v =-\infty  }^\mu  \left(           2^{vn/2} 2^{(v-\mu)((n+\epsilon)/2-n)} 
		\left(    \mathcal M \left(  \sum_{P \in \mathcal D _{\mu}}  |\vec s _P|^J \chi_P \right)  \right)^{1/J} 
		\right) ^q \right)^{1/q}\\
		& \lesssim \left(  \sum_{ \mu = -\infty}^\infty 2^{ \mu n q /2 } \left(    \mathcal M \left(  \sum_{P \in \mathcal D _{\mu}}  |\vec s _P|^J \chi_P \right)  \right)^{q/J} \right)^{1/q} .
	\end{align*}
	Using Lemma \ref{lem:hardy}, we obtain
	\begin{align*}
		\|B_{0}\vec{s}\|_{\dot{f}_{p,t,r}^{0,q}}
		& \lesssim \bigg\|\left(  \sum_{ \mu = -\infty}^\infty 2^{ \mu n q /2 } \left(    \mathcal M \left(  \sum_{P \in \mathcal D _{\mu}}  |\vec s _P|^J \chi_P \right)  \right)^{q/J} \right)^{1/q} \bigg\|_{M_{p}^{t,r}}\\
		& \lesssim\bigg\|\bigg(\sum_{\mu=-\infty}^{\infty}\sum_{P\in\mathcal{D}_{\mu}}\Big(|P|^{-1/2}|\vec{s_{P}}|\chi_{P}\Big)^{q}\bigg)^{1/q}\bigg\|_{M_{p}^{t,r}}. 
	\end{align*}
	Thus the proof is finished.
\end{proof}

\begin{defn} \label{almost diag d delta}
	Let $0<p<\infty$, $0<q\le\infty$, $s\in\mathbb{R}$. 
	Let $d, \tilde d, \Delta$  be the same as in Lemma \ref{AQ AR improved}.
	A matrix $B=\{b_{QP}\}_{Q,P\in\mathcal{D}}$ is almost diagonal, written
	$B\in\mathbf{ad}_{p}^{s,q}(d, \tilde d, \Delta  )$, if there exists $c>0$ such that
	$|b_{QP}|\le c\omega_{QP}$ for all $Q,P\in\mathcal{D}$, where
	\begin{align}
		\nonumber
		\omega_{QP} &: =  \Big(\frac{\ell(Q)}{\ell(P)}\Big)^{s}\min\bigg\{\Big(\frac{\ell(P)}{\ell(Q)}\Big)^{(n+\epsilon)/2+n/\min(1,q,p)-n +\tilde d / p'},\Big(\frac{\ell(Q)}{\ell(P)}\Big)^{(n+\epsilon)/2 +  d/p}\bigg\}\\
		& \quad \times  \bigg(1+\frac{|x_{Q}-x_{P}|}{\max\{\ell(Q),\ell(P)\}}\bigg)^{-\frac{n}{\min(1,q,p)}-\epsilon - \Delta}. \label{eq:omega QP-1}
	\end{align}
\end{defn}

\begin{thm}
	\label{thm:almost dia bounded on discrete} Let $0<p<t<r<\infty$
	or $0<p\le t<r=\infty$. Let $0<q\le\infty$, $s\in\mathbb{R}$.
	Let $W \in \mathcal A_p$ and  
	$\{A_{Q}\}_{Q\in\mathcal{D}}$ be a family of reducing operators of order $p$ for $W$. Let $d, \tilde d, \Delta$  be the same as in Lemma \ref{AQ AR improved}.
	Suppose that $B\in\mathbf{ad}_{p}^{s,q}(d, \tilde d, \Delta )$.
	Then $B$ defines a bounded operator on $\dot{f}_{p,t,r}^{s,q}(A_{Q})$.
\end{thm}

\begin{proof}
	Let $\vec{s} \in \dot{f}_{p,t,r}^{s,q}(A_{Q})$.
	Suppose that $B\vec{s}=\vec{t}=\{\vec{t}_{Q}\}_{Q\in\mathcal{D}}$,
	where $\vec{t}_{Q}=\sum_{P\in\mathcal{D}}b_{QP}\vec{s}_{P}$. Define
	a scalar sequence $t_{A}=\{t_{A,Q}\}_{Q\in\mathcal{D}},$ where $t_{A,Q}=|A_{Q}\vec{t}_{Q}|$
	and $s_{A}=\{s_{A,Q}\}_{Q\in\mathcal{D}},$ where $s_{A,Q}=|A_{Q}\vec{s}_{Q}|$.
	Then
	\[
	\|\vec{t}\|_{\dot{f}_{p,t,r}^{s,q}(A_{Q})}=\|t_{A}\|_{\dot{f}_{p,t,r}^{s,q}},\;\|\vec{s}\|_{\dot{f}_{p,t,r}^{s,q}(A_{Q})}=\|s_{A}\|_{\dot{f}_{p,t,r}^{s,q}}.
	\]
	Let $\gamma_{QP}=\omega_{QP}\|A_{Q}A_{P}^{-1}\|$. Then
	\[
	t_{A,Q}\le c\sum_{P\in\mathcal{D}}\gamma_{QP}s_{A,P},
	\]
	where $c$ comes from Definition \ref{almost diag d delta}.
	By Lemma \ref{AQ AR improved}, we obtain 
	\begin{align*}
		\gamma_{QP}  & \lesssim  
		\Big(\frac{\ell(Q)}{\ell(P)}\Big)^{s}\min\bigg\{\Big(\frac{\ell(P)}{\ell(Q)}\Big)^{(n+\epsilon)/2+n/\min(1,q,p)-n + \tilde d / p' },\Big(\frac{\ell(Q)}{\ell(P)}\Big)^{(n+\epsilon)/2 + d/p}\bigg\}\\
		& \quad  \times  \bigg(1+\frac{|x_{Q}-x_{P}|}{\max\{\ell(Q),\ell(P)\}}\bigg)^{-\frac{n}{\min(1,q,p)}-\epsilon - \Delta}\\
		&  \quad\times  \max \left\{   \left( \frac{\ell (P)}{ \ell (Q) } \right)^{d/p},  \left( \frac{\ell (Q)}{ \ell (P) } \right)^{\tilde d /p' }   \right\} \left( 1+ \frac{|c_Q - c_P| }{ \max \{ \ell(Q) , \ell (P) \}   }  \right) ^\Delta \\
		& \lesssim \Big(\frac{\ell(Q)}{\ell(P)}\Big)^{s}\min\bigg\{\Big(\frac{\ell(P)}{\ell(Q)}\Big)^{(n+\epsilon)/2+n/\min(1,q,p)-n},\Big(\frac{\ell(Q)}{\ell(P)}\Big)^{(n+\epsilon)/2}\bigg\}\\
		& \quad \times\bigg(1+\frac{|x_{Q}-x_{P}|}{\max\{\ell(Q),\ell(P)\}}\bigg)^{-\frac{n}{\min(1,q,p)}-\epsilon} .
	\end{align*}
	Thus $\gamma_{QP}$ satisfies the almost diagonality condition
	(\ref{eq:omega QP}).
	Hence Theorem \ref{thm:unweighted discrete almost}
	implies the boundedness on $\dot{f}_{p,t,r}^{s,q}(A_{Q})$.
\end{proof}
We recall the notation of smooth molecules, as in \cite{FJ90}.
\begin{defn}
	\label{def:molecules} Let $0<\delta\le1$, $M>0$ and $N,K\in\mathbb{Z}$.
	We say that $\{m_{Q}\}_{Q\in\mathcal{D}}$ is a family of smooth $(N,K,M,\delta)$-molecules
	if there exist $\epsilon,c>0$ such that, for all $Q\in\mathcal{D}$,
	\begin{align*}
		(\mathrm{M}1) & \int x^{\gamma}m_{Q}(x)\mathrm{d}x=0,\;\mathrm{for}\;|\gamma|\le N,\\
		(\mathrm{M}2) & \;|m_{Q}(x)|\le c|Q|^{-1/2}\Big(1+\frac{|x-x_{Q}|}{\ell(Q)}\Big)^{-\max(M,N+1+n+\epsilon)},\\
		(\mathrm{M}3) & \;|\partial^{\gamma}m_{Q}(x)|\le c|Q|^{-1/2-|\gamma|/n}\Big(1+\frac{|x-x_{Q}|}{\ell(Q)}\Big)^{-M},\;\mathrm{if}\;|\gamma|\le K,\\
		(\mathrm{M}4) & \;|\partial^{\gamma}m_{Q}(x)-\partial^{\gamma}m_{Q}(y)|\le c|Q|^{-1/2-|\gamma|/n-\delta/n}|x-y|^{\delta}\\
		& \times\sup_{|z|\le|x-y|}\Big(1+\frac{|x-z-x_{Q}|}{\ell(Q)}\Big)^{-M},\;\mathrm{if}\;|\gamma|=K.
	\end{align*}	
	It is understood that (M1) is void if $N<0$ and (M3), (M4) are void
	if $K<0$.
\end{defn}

\begin{lem}
	[Appendix B, \cite{FJ90}] \label{lem:mQ} Suppose that $\varphi\in\mathcal{A}$
	and $\{m_{Q}\}_{Q\in\mathcal{D}}$ is a family of smooth $(N,K,M,\delta)$-molecules.
	Then there exists a positive constant $c$ such that
		
	$(\mathrm{i})$ for all $P\in\mathcal{D}$ with $\ell(P)=2^{-k}\ge2^{-j}$,
	we have
	\begin{equation*}
		|\varphi_{j}*m_{P}(x)|\le c2^{kn/2}2^{-(j-k)(K+\delta)}(1+2^{k}|x-x_{P}|)^{-M},
	\end{equation*}
	and
	
	$(\mathrm{ii})$ for all $P\in\mathcal{D}$ with $\ell(P)=2^{-k}\le2^{-j}$,
	we have
	\begin{equation*}
		|\varphi_{j}*m_{P}(x)|\le c2^{kn/2}2^{-(k-j)(N+1+n)}(1+2^{j}|x-x_{P}|)^{-M}.
	\end{equation*}
\end{lem}

\begin{thm}
	\label{thm:varphi transform} Let $0<p<t<r<\infty$ or $0<p\le t<r=\infty$.
	Let $0<q\le\infty$, $s\in\mathbb{R}$. 
	Let $W \in \mathcal A_p$ and  
	$\{A_{Q}\}_{Q\in\mathcal{D}}$ be a family of reducing operators of order $p$ for $W$. Let $d, \tilde d, \Delta$  be the same as in Lemma \ref{AQ AR improved}.
	Suppose that $N,K\in\mathbb{Z}$,
	$M>0$ and $\delta\in(0,1]$ satisfy $ N > -s+ n/\min(1,q,p) + \tilde d / p'  - 1- n$,
	$K +\delta > s  + d/p $, and $ M >  n / \min(1,q,p)+  \Delta$. 
	Suppose that $\{m_{Q}\}_{Q\in\mathcal{D}}$
	is a family of smooth $(N,K,M,\delta)$-molecules. Let $\vec{s}\in\dot{f}_{p,t,r}^{s,q}(A_{Q})$
	. Then $\vec{f}=\sum_{Q\in\mathcal{D}} m_{Q} \vec{s}_{Q}\in\dot{F}_{p,t,r}^{s,q}(A_{Q})$,
	\begin{equation*}
		\|\vec{f}\|_{\dot{F}_{p,t,r}^{s,q}(A_{Q})}\le c\|\vec{s}\|_{\dot{f}_{p,t,r}^{s,q}(A_{Q})}
	\end{equation*}
	where the constant $c$ independent of $\vec{s}$. In
	particular, for $\varphi\in\mathcal{A},$ we have
	\begin{equation}
		\|\vec{f}\|_{\dot{F}_{p,t,r}^{s,q}(A_{Q})}\le c\|\{\langle\vec{f},\varphi_{Q}\}_{Q\in\mathcal{D}}\|_{\dot{f}_{p,t,r}^{s,q}(A_{Q})} \label{eq:F le f}
	\end{equation}
	where the constant $c$ independent of $\vec{f}$.
\end{thm}

\begin{proof}
	For $Q\in\mathcal{D}_{j}$, let $g_{Q}=|Q|^{1/2}|A_{Q}\varphi_{j}*\sum_{P\in\mathcal{D}} m_{P} \vec{s}_{P}|$. Then
	\[
	\|\vec{f}\|_{\dot{F}_{p,t,r}^{s,q}(A_{Q})}=\bigg\|\bigg(\sum_{j\in\mathbb{Z}}\sum_{Q\in\mathcal{D}_{j}}\Big(|Q|^{-s/n-1/2}g_{Q}\chi_{Q}\Big)^{q}\bigg)^{1/q}\bigg\|_{M_{p}^{t,r}}.
	\]
	For any $P,Q\in\mathcal{D}$ and $x\in Q$, we have
	\[
	1+\frac{|x-x_{P}|}{\max\{\ell(P),\ell(Q)\}}\approx1+\frac{|x_{Q}-x_{P}|}{\max\{\ell(P),\ell(Q)\}}.
	\]
	Hence, by Lemma \ref{lem:mQ}, 
	if $\ell(P) = 2^{-k} \ge 2^{-j} = \ell (Q)$, then
	\begin{align*}
		& |Q|^{1/2}|\varphi_{j}*m_{P}(x)| \\
		& \lesssim  |Q|^{1/2} 2^{kn/2}2^{-(j-k)(K+\delta)}(1+2^{k}|x-x_{P}|)^{-M} \\
		&\le  \min\bigg\{\Big(\frac{\ell(P)}{\ell(Q)}\Big)^{(n+\epsilon)/2+n/\min(1,q,p)-n +  \tilde d / p' },\Big(\frac{\ell(Q)}{\ell(P)}\Big)^{(n+\epsilon)/2 + d/p }\bigg\}\\
		& \quad  \times \Big(\frac{\ell(Q)}{\ell(P)}\Big)^{s} \bigg(1+\frac{|x_{Q}-x_{P}|}{\max\{\ell(Q),\ell(P)\}}\bigg)^{-\frac{n}{\min(1,q,p)}-\epsilon - \Delta} : =  \omega_{QP} .
	\end{align*}
	If  $\ell(P)=2^{-k}\le2^{-j} = \ell (Q)$, then by  Lemma \ref{lem:mQ}, we have
	\begin{align*}
	&	|Q|^{1/2}|\varphi_{j}*m_{P}(x)|\\
	 & \lesssim |Q|^{1/2} 2^{kn/2}2^{-(k-j)(N+1+n)}(1+2^{j}|x-x_{P}|)^{-M} \\
		&\le  \min\bigg\{\Big(\frac{\ell(P)}{\ell(Q)}\Big)^{(n+\epsilon)/2+n/\min(1,q,p)-n + \tilde d / p' },\Big(\frac{\ell(Q)}{\ell(P)}\Big)^{(n+\epsilon)/2 + d/p }\bigg\}\\
		& \quad  \times  \Big(\frac{\ell(Q)}{\ell(P)}\Big)^{s} \bigg(1+\frac{|x_{Q}-x_{P}|}{\max\{\ell(Q),\ell(P)\}}\bigg)^{-\frac{n}{\min(1,q,p)}-\epsilon - \Delta} : =  \omega_{QP} .
	\end{align*}
	Hence we obtain 
	$|Q|^{1/2}|\varphi_{j}*m_{P}(x)|\le c\omega_{QP}$
	for all $x\in Q$, where $\omega_{QP}$ is the same as (\ref{eq:omega QP-1}).
	Therefore,
	\[
	g_{Q}\chi_{Q}\le\sum_{P\in\mathcal{D}}|Q|^{1/2}|\varphi_{j}*m_{P}\|A_{Q}\vec{s}_{\text{P }}|\chi_{Q}\le c\sum_{P\in\mathcal{D}}\omega_{QP}\|A_{Q}A_{P}^{-1}\||A_{P}\vec{s}_{\text{P }}|\chi_{Q}.
	\]
	Define $G=\{\omega_{QP}\|A_{Q}A_{P}^{-1}\|\}_{Q,P\in\mathcal{D}}$
	and $s_{A}=\{|A_{P}\vec{s}_{\text{P }}|\}_{P\in\mathcal{D}}$. 
	Then we have
	\begin{equation*}
		g_{Q}\chi_{Q}\le c \left( G(s_{A}) \right)_Q \chi_{Q}.
	\end{equation*}
	Similar
	to the proof of Theorem \ref{thm:almost dia bounded on discrete}, we have
	\[
	\|\vec{f}\|_{\dot{F}_{p,t,r}^{s,q}(A_{Q})}\le c\|G(s_{A})\|_{\dot{f}_{p,t,r}^{s,q}}\le c\|s_{A}\|_{\dot{f}_{p,t,r}^{s,q}}=c\|\vec{s}\|_{\dot{f}_{p,t,r}^{s,q}(A_{Q})}.
	\]
	Then (\ref{eq:F le f}) follows since $\vec{f}=\sum_{Q\in\mathcal{D}}\langle\vec{f},\varphi_{Q}\rangle\psi_{Q}$
	by (\ref{eq:converge}), and $\{\psi_{Q}\}_{Q\in\mathcal{D}}$ is
	a family of smooth $(N,K,M,\delta)$ molecules for all possible $N,K,M,\delta$.
	This finishes the proof.
\end{proof}
Next we obtain the ``$\varphi$-transform'' in spaces $\dot{F}_{p,t,r}^{s,q}(A_{Q})$ and show that $\dot{F}_{p,t,r}^{s,q}(A_{Q})$ are independent of the choice of $\varphi\in\mathcal{A}$.

\begin{thm}
	Let $0<p<t<r<\infty$ or $0<p\le t<r=\infty$.
	Suppose that $s\in\mathbb{R},$ $0<q\le\infty$,   $\varphi \in \mathcal A$.
	Let $W \in \mathcal A_p$ and  
	$\{A_{Q}\}_{Q\in\mathcal{D}}$ be a family of reducing operators of order $p$ for $W$. 
	For $\vec{f}\in\mathscr{S}'/\mathcal{P}(\mathbb{R}^{n})$, let $\vec{s}=\{\vec{s}_{Q}\}_{Q\in\mathcal{D}}$,
	where $\vec{s}_{Q}=\langle\vec{f},\varphi_{Q}\rangle.$ Then
	\begin{equation}\label{F Aq app f}
		\|\vec{f}\|_{\dot{F}_{p,t,r}^{s,q}(A_{Q})}\approx\|\vec{s}\|_{\dot{f}_{p,t,r}^{s,q}(A_{Q})}.
	\end{equation}
	In addition, function spaces  $\dot{F}_{p,t,r}^{s,q}(A_{Q})$
	are independent
	of the choice of $\varphi\in\mathcal{A},$ in the sense that different
	choices yield equivalent quasi-norms.
\end{thm}

\begin{proof}
	We first prove the spaces are independent of the choice of the test
	functions $\varphi\in\mathcal{A}$. 
	Suppose that $\varphi,\kappa\in\mathcal{A}$. We label spaces defined
	by $\varphi$ and $\kappa$ as $\dot{F}_{p,t,r}^{s,q}(A_{Q},\varphi)$
	and $\dot{F}_{p,t,r}^{s,q}(A_{Q},\kappa)$, respectively. Then we
	can select $\psi,\lambda\in\mathcal{A}$ such that
	\[
	\sum_{v\in\mathbb{Z}}\overline{\mathcal{F}\varphi_{v}(\xi)}\mathcal{F}\psi_{v}(\xi)=\sum_{v\in\mathbb{Z}}\overline{\mathcal{F}\kappa_{v}(\xi)}\mathcal{F}\lambda_{v}(\xi)=1
	\]
	for all $\xi\neq0$. Define $\vec{\tilde{s}}=\{\vec{\tilde{s}}_{Q}\}_{Q\in\mathcal{D}}$
	by $\vec{\tilde{s}}_{Q}=\langle\vec{f},\tilde{\varphi}_{Q}\rangle$
	and $\vec{t}=\{\vec{t}_{Q}\}_{Q\in\mathcal{D}}$ by $\vec{t}_{Q}=\langle\vec{f},\kappa_{Q}\rangle$.
	From Theorem \ref{thm:varphi transform}, we have
	\[
	\|\vec{f}\|_{\dot{F}_{p,t,r}^{s,q}(A_{Q},\kappa)}\le c\|\vec{t}\|_{\dot{f}_{p,t,r}^{s,q}(A_{Q})}.
	\]
	Applying (\ref{eq:converge}) with $\varphi,\psi$, and $\vec{f}$
	replaced by $\tilde{\psi},\tilde{\varphi}$, and $\kappa_{Q}$, respectively,
	we have
	\[
	\kappa_{Q}=\sum_{P\in\mathcal{D}}\langle\kappa_{Q},\tilde{\psi}_{P}\rangle\tilde{\varphi}_{P}.
	\]
	Note that $\kappa_{Q}\in\mathscr{S}_{0}(\mathbb{R}^{n})$, hence $\sum_{P\in\mathcal{D}}\langle\kappa_{Q},\tilde{\psi}_{P}\rangle\tilde{\varphi}_{P}$
	converges in $\mathscr{S}(\mathbb{R}^{n})$. Therefore, since $\vec{f}\in\mathscr{S}'/\mathcal{P}(\mathbb{R}^{n})$,
	\[
	\vec{t}_{Q}=\langle\vec{f},\kappa_{Q}\rangle=\bigg\langle\vec{f},\sum_{P\in\mathcal{D}}\langle\kappa_{Q},\tilde{\psi}_{P}\rangle\tilde{\varphi}_{P}\bigg\rangle=\sum_{P\in\mathcal{D}}\langle\kappa_{Q},\tilde{\psi}_{P}\rangle\vec{\tilde{s}}_{Q}.
	\]	
	Notice that, for $\ell(Q)=2^{-j},$ $\langle\kappa_{Q},\tilde{\psi}_{P}\rangle=|Q|^{1/2}\kappa_{j}*\psi_{P}(x_{Q})$.
	Since $\{\psi_{Q}\}_{ Q \in \mathcal D }$ is a family of smooth ($N,K,M,\delta$)-molecules
	for all possible $N,K,M,\delta$, Lemma \ref{lem:mQ} implies that
	the matrix $B=\{b_{QP}\}_{Q,P \in \mathcal{D} }$ defined by $b_{QP}=\langle\kappa_{Q},\tilde{\psi}_{P}\rangle$
	is almost diagonal, that is $B\in\mathbf{ad}_{p}^{s,q}(d, \tilde d, \Delta )$, for
	all possible $s,q,p,$ and $d, \tilde d, \Delta$  where  $d, \tilde d, \Delta$  are the same as in Lemma \ref{AQ AR improved}. From Theorem \ref{thm:almost dia bounded on discrete},
	$B$ is bounded on $\dot{f}_{p,t,r}^{s,q}(A_{Q})$. Hence
	\[
	\|\vec{t}\|_{\dot{f}_{p,t,r}^{s,q}(A_{Q})}=\|B\vec{\tilde{s}}\|_{\dot{f}_{p,t,r}^{s,q}(A_{Q})}\le c\|\vec{\tilde{s}}\|_{\dot{f}_{p,t,r}^{s,q}(A_{Q})}\le c\|\vec{f}\|_{\dot{F}_{p,t,r}^{s,q}(A_{Q},\varphi)}
	\]
	where the last step is by Theorem \ref{thm:sup A_Q LE}. 
	Hence, we prove 
	\begin{equation*}
		\|\vec{f}\|_{\dot{F}_{p,t,r}^{s,q}(A_{Q},\kappa)} \le c \|\vec{f}\|_{\dot{F}_{p,t,r}^{s,q}(A_{Q},\varphi)} .
	\end{equation*}
	Thus we show
	the independence of the choice of admissible functions by interchanging
	$\varphi$ and $\kappa$. 
	
	Next we prove (\ref{F Aq app f}). 
	Since  $\tilde \varphi \in \mathcal A$ for  $\varphi \in \mathcal A$,
	applying Theorem \ref{thm:sup A_Q LE} with $\varphi$ replaced by $\tilde \varphi$, we  obtain
	\begin{equation*}
		\|\vec{s}\|_{\dot{f}_{p,t,r}^{s,q}(A_{Q})}  =	\|\{\langle\vec{f}, {\varphi}_{Q}\rangle\}_{Q\in\mathcal{D}}\|_{\dot{f}_{p,t,r}^{s,q}(A_{Q})}\le c\|\vec{f}\|_{\dot{F}_{p,t,r}^{s,q}(A_{Q}, \tilde \varphi )} .
	\end{equation*}
	Then by the independence of the choice of the test
	functions $\varphi$ and $ \tilde \varphi $  and Theorem \ref{thm:varphi transform},  we obtain
	\begin{equation*}
		\|\vec{f}\|_{\dot{F}_{p,t,r}^{s,q}(A_{Q}, \tilde \varphi )} \approx \|\vec{f}\|_{\dot{F}_{p,t,r}^{s,q}(A_{Q},  \varphi )}  \lesssim 	\|\{\langle\vec{f}, {\varphi}_{Q}\rangle\}_{Q\in\mathcal{D}}\|_{\dot{f}_{p,t,r}^{s,q}(A_{Q})} = \|\vec{s}\|_{\dot{f}_{p,t,r}^{s,q}(A_{Q})} .
	\end{equation*}
	Thus we prove (\ref{F Aq app f}) and the proof is finished.
\end{proof}

Similar to  \cite[Proposition 2.7]{FJ90},
the following result gives a characterization of the $\dot{f}_{p,t,r}^{s,q}$-norm via sequences of sparse sets.
\begin{thm}
	\label{thm:EQ}  Suppose that for each $Q\in\mathcal{D}$,
	$E_{Q}\subset Q$ is a set such that $|E_{Q}|/|Q|\ge\epsilon  >0$. Then
	\[
	\|\{s_{Q}\}_{Q\in\mathcal{D}}\|_{\dot{f}_{p,t,r}^{s,q}}\approx\bigg\|\Big(\sum_{Q\in\mathcal{D}}(|Q|^{-s/n-1/2}|s_{Q}|\chi_{E_{Q}})^{q}\Big)^{1/q}\bigg\|_{M_{p}^{t,r}} .
	\]
	
\end{thm}

\begin{proof}
	Since $\chi_{E_{Q}}\le\chi_{Q}$, one direction is trivial. For another
	direction, we use the fact that for any $A>0$,
	\[
	\chi_{Q}\le\epsilon^{-1/A}(\mathcal{M}(\chi_{E_{Q}}^{A}))^{1/A}.
	\]
	Let $A\in(0,\min(1,q,p ))$. Then by Lemma \ref{lem:hardy}, we have
	\begin{align*}
		\|\{s_{Q}\}_{Q\in\mathcal{D}}\|_{\dot{f}_{p,t,r}^{s,q}} & \lesssim\epsilon^{-1/A}\bigg\|\Big(\sum_{Q}(|Q|^{-s/n-1/2}|s_{Q}|\chi_{E_{Q}})^{q}\Big)^{1/q}\bigg\|_{M_{p}^{t,r}}.
	\end{align*}
	Thus we finish the proof.
\end{proof}

\begin{thm}\label{thm:A_Q le W discrete}
	Let $0<p<t<r<\infty$ or $0<p\le t<r=\infty$. Let $0<q\le\infty$,
	$s\in\mathbb{R}$. 
	Let $W \in \mathcal A_p$ and  
	$\{A_{Q}\}_{Q\in\mathcal{D}}$ be a family of reducing operators of order $p$ for $W$.
	Then there
	exists a constant $c>0$ such that for all $\vec{s}=\{\vec{s}_{Q}\}_{Q\in\mathcal{D}}$,
	
	\begin{equation}
		\|\vec{s}\|_{\dot{f}_{p,t,r}^{s,q}(A_{Q})}\le c\|\vec{s}\|_{\dot{f}_{p,t,r}^{s,q}(W)}. \label{eq: seq f AQ le W}
	\end{equation}
\end{thm}

\begin{proof}
	We divide the proof into two cases.
	
	Case $p\in(0,1]$. From (\ref{eq:p le 1 AQW}), we have
	\[
	|A_{Q}\vec{s}_{Q}|\chi_{Q}\le\|A_{Q}W^{-1/p}\||W^{1/p}\vec{s}_{Q}|\chi_{Q}\lesssim|W^{1/p}\vec{s}_{Q}|\chi_{Q},
	\]
	which implies (\ref{eq: seq f AQ le W}).
	
	Case $p\in(1,\infty)$. Let $C_{1}$ be the constant from (\ref{eq:p>1 AQW ^-1/p})
	when $v=1$. For each $Q\in\mathcal{D}$, let
	\[
	E_{Q}:=\{x\in Q:\|A_{Q}W^{-1/p}(x)\|\le2C_{1}\}.
	\]
	By Chebychev's inequality and (\ref{eq:p>1 AQW ^-1/p}),
	\[
	2C_{1}|Q\backslash E_{Q}|\le\int_{Q\backslash E_{Q}}\|A_{Q}W^{-1/p}(x)\|\mathrm{d}x\le\int_{Q}\|A_{Q}W^{-1/p}(x)\|\mathrm{d}x\le C_{1}|Q|.
	\]
	Thus $|Q\backslash E_{Q}|\le|Q|/2$.  This means $|E_{Q}|\ge|Q|/2$. By Theorem
	\ref{thm:EQ} and $|A_{Q}\vec{s}_{Q}|\le\|A_{Q}W^{-1/p}\||W^{1/p}\vec{s}_{Q}|$,
	we have
	\begin{align*}
		\|\vec{s}\|_{\dot{f}_{p,t,r}^{s,q}(A_{Q})} 
		& \lesssim \bigg\|\Big(\sum_{v=-\infty}^{\infty}\sum_{Q\in\mathcal{D}_{v}}\Big[|Q|^{-s/n-1/2}|A_{Q}\vec{s}_{Q}|\chi_{E_{Q}}\Big]^{q}\Big)^{1/q}\bigg\|_{M_{p}^{t,r}}\\
		& \lesssim\bigg\|\Big(\sum_{v=-\infty}^{\infty}\sum_{Q\in\mathcal{D}_{v}}\Big[|Q|^{-s/n-1/2}|W^{1/p}\vec{s}_{Q}|\chi_{E_{Q}}\Big]^{q}\Big)^{1/q}\bigg\|_{M_{p}^{t,r}}\\
		& \lesssim \|\vec{s}\|_{\dot{f}_{p,t,r}^{s,q}(W)}. \qedhere
	\end{align*}
\end{proof}

\begin{thm}\label{thm:Aq le W function}
	Let $0<p<t<r<\infty$ or $0<p\le t<r=\infty$. Let $0<q\le\infty$, $s\in \mathbb{R}$.
	Let $W \in \mathcal A_p$ and  
	$\{A_{Q}\}_{Q\in\mathcal{D}}$ be a family of reducing operators of order $p$ for $W$. Let $d, \tilde d, \Delta$  be the same as in Lemma \ref{AQ AR improved}.
	Then there exists a constant
	$c>0$ such that
	\begin{equation}
		\|\vec{f}\|_{\dot{F}_{p,t,r}^{s,q}(A_{Q})}\le c\|\vec{f}\|_{\dot{F}_{p,t,r}^{s,q}(W)}. \label{eq:F Aq le F W}
	\end{equation}
	
\end{thm}

\begin{proof}
	Let $\varphi\in\mathcal{A}$ be the test function in the
	definition of $\dot{F}_{p,t,r}^{s,q}(A_{Q})$, $\dot{F}_{p,t,r}^{s,q}(W)$.
	Let $\tilde{\varphi}(x)=\overline{\varphi(-x)}$ as usual. We will
	show
	\begin{equation}
		\Big\|\big\{\langle\vec{f},\tilde{\varphi}_{Q}\rangle\big\}_{Q\in\mathcal{D}}\Big\|_{\dot{f}_{p,t,r}^{s,q}(A_{Q})}\le c\|\vec{f}\|_{\dot{F}_{p,t,r}^{s,q}(W)}.\label{eq:f tilde varphi le F w}
	\end{equation}
	Then (\ref{eq:F Aq le F W}) follows from (\ref{eq:f tilde varphi le F w})
	and (\ref{eq:F le f}) with $\varphi$ replaced by $\tilde{\varphi}\in\mathcal{A}$.
	
	Note that $\{A_{Q}\}_{Q\in\mathcal{D}}$ is strongly doubling of order $(d, \tilde d, \Delta,p)$
	and weakly doubling of order $\Delta$ by Lemma \ref{AQ AR improved}.
	
	Case $0<p\le1$. For $x\in Q_{j, \ell}$, from Lemma \ref{lem:W AQ},
	we have
	\[
	|A_{Q_{j, \ell}}\varphi_{j}*\vec{f}(x)|^{A}\le c|W^{1/p}(x)\varphi_{j}*\vec{f}(x)|^{A}.
	\]
	By Lemma \ref{lem:A inequality}, we obtain
	\[
	\sup_{x\in Q_{j, k}}|A_{Q_{j, k}}\varphi_{j}*\vec{f}(x)|^{A}\lesssim  \sum_{l\in\mathbb{Z}^{n}}(1+|k-l|)^{-A(R-\Delta)}2^{jn}\int_{Q_{j, \ell}}|W^{1/p}(z)\varphi_{j}*\vec{f}(z)|^{A}\mathrm{d}z
	\]
	for any $A\in(0,1]$, $R>0$. Proceeding as in
	the proof of Theorem \ref{thm:sup A_Q LE}, we obtain (\ref{eq:f tilde varphi le F w}).
	
	Case $1<p<\infty$. Let $w=p'/(p'-A)$, $w'=p'/A$, where $p'=p/(p-1)$.
	From \cite[Theorem 3.5]{FraRou19}, we have
	\begin{align*}
	&	\sup_{x\in Q_{j, k}}|A_{Q_{j, k}}\varphi_{j}*\vec{f}(x)|^{Aw}  \\
	& \lesssim  \sum_{l\in\mathbb{Z}^{n}}(1+|k-l|)^{-A(R-\Delta)}2^{jn}\int_{Q_{j, \ell}}|W^{1/p}(z)\varphi_{j}*\vec{f}(z)|^{Aw}\mathrm{d}z
	\end{align*}
	for $R$ sufficiently large. Proceeding as case $0<p\le1$, we obtain
	(\ref{eq:f tilde varphi le F w}).
\end{proof}

We need some preparation for the another direction of Theorems  \ref{thm:A_Q le W discrete} and  \ref{thm:Aq le W function}  with certain restriction on $q$.

Given a quasi-Banach space $Y$  with norm $\|\cdot\|_{Y}$, its dual $Y^\ast$  is the space of all continuous linear functionals $T$ on $Y$ equipped with the norm 
\begin{equation*}
	\|T\|_{Y^\ast } = \sup_{ \|x\|_Y  \le 1}  |T(x)|.
\end{equation*}
Note that the dual of a quasi-Banach space is always a Banach space.
A Banach space $Y$ is called a predual of  a Banach space $X$ if $Y^\ast$  is isometric to $X$.

Let $X$ be a Banach space.
In \cite{BGX25}, the authors of the paper  introduced $X$-valued Bourgain-Morrey spaces and obtain their preduals. In particular, when $X = \ell^{q}, q\in (1,\infty)$, the authors of the paper  obtained the powered Hardy-Littlewood maximal operator $M_\eta $ is bounded on $\mathcal{H}_{p'}^{t',r'} (\ell^{q'}) $.

\begin{defn}
	Let $1\le p\le t\le\infty$ and $X$ be a Banach space such that $  ( ^\ast X )^ \ast = X $. A $^\ast X$-valued function $b$ is a $(p',t',{^\ast X})$-block
	if there exists a cube $Q$ which supports $b$ such that
	\[
	\left\| \|b\|_{^\ast X} \right\|_{L^{p'}}\le |Q|^{1/t-1/p} = |Q|^{1/p' -1/t'}.
	\]
	If we need to indicate $Q$, we  say that $b$ is a $(p',t', {^\ast X})$-block supported on $Q$.
	
	The function space $\mathcal{H}_{p'}^{t',r'} (^\ast X)$
	is the set of  all $^\ast X$-valued measurable functions $f$ such that $f$ is realized
	as the sum
	\begin{equation}\label{eq:block f X}
		f = \sum_{(j,k)\in\mathbb{Z}^{n+1}}\lambda_{j,k}b_{j,k}
	\end{equation}
	with some $\lambda=\{\lambda_{j,k}\}_{(j,k)\in\mathbb{Z}^{n+1}}\in\ell^{r'}(\mathbb{Z}^{n+1})$
	and $b_{j,k}$ is a  $(p',t', {^\ast X})$-block supported on $Q_{j,k}$ where (\ref{eq:block f X}) converges in $^\ast X$ almost everywhere on $\rn$. The norm of $\mathcal{H}_{p'}^{t',r'} (^\ast X)$
	is defined by
	\[
	\|f\|_{\mathcal{H}_{p'}^{t',r'} (^\ast X)} :=\inf_{\lambda}\|\lambda\|_{\ell^{r'}},
	\]
	where the infimum is taken over all admissible sequence $\lambda$
	such that (\ref{eq:block f X}) holds.
\end{defn}

\begin{lem}  [Theorem  3.8, \cite{BGX25}] \label{predual seq BM}
	Let $X$ be a Banach space  such that $  ( ^\ast X )^ \ast = X $.
	Let $^\ast X$ be reflexive or $X$ be separable. Let $1 < p < t < r <\infty$  or $1 <  p \le t < r =\infty $. Then the dual  space of
	$\mathcal{H}_{p'}^{t',r'} (^\ast X)$, denoted by $ \Big( \mathcal{H}_{p'}^{t',r'} (^\ast X) \Big)^* $, is
	$M_{p}^{t,r} (X)$, that is,  
	\begin{equation*}
		\Big( \mathcal{H}_{p'}^{t',r'} (^\ast X) \Big)^*   = M_{p}^{t,r} (X)
	\end{equation*}
	in the following sense:
	
	{\rm (i)} if $f \in M_{p}^{t,r} (X)$, then the linear functional 
	\begin{equation} \label{Jf}
		J_f : f \to  J_f (g) :=   \int_\rn  \langle g (x), f (x) \rangle \d x
	\end{equation}
	is bounded on $ \mathcal{H}_{p'}^{t',r'} (^\ast X) $.
	
	{\rm (ii)} conversely, any continuous linear functional on $\mathcal{H}_{p'}^{t',r'} (^\ast X)$ arises as in  (\ref{Jf}) with a unique $f \in M_{p}^{t,r} (X)$.
	
	Moreover, $\|f\|_{  M_{p}^{t,r} (X)}  = \| J_f \|_{ ( \mathcal{H}_{p'}^{t',r'} (^\ast X) )^*}$ and 
	\begin{equation*} 
		\|g\|_{\mathcal{H}_{p'}^{t',r'} (^\ast X)  } =\max \left\{\left|\int_\rn  \langle g (x), f (x) \rangle \d x    \right|: f\in  M_{p}^{t,r} (X), \|f\|_{  M_{p}^{t,r} (X) } \le 1 \right\}.
	\end{equation*}
\end{lem}

\begin{rem}
	When $X = \mathbb C$, the block spaces $\mathcal{H}_{p'}^{t',r'} (^\ast X)$ are introduced in \cite{M23} and Lemma \ref{predual seq BM} become \cite[Theorem 2.17]{M23}. 
\end{rem}

\begin{lem} [Theorem 5.3, \cite{BGX25}] \label{M eta r < infty = infty}
	Let $1<q < \infty $.  Let $1<p<t<r<\infty$ or let $1<p \le t <r =\infty$.  Let $ 0<\eta < \min\{p' , q' \}$. Then $\M _\eta$ is bounded on $\mathcal{H}_{p'}^{t',r'} (\ell^{q'})$.
\end{lem}
\begin{rem}
	In \cite[Theorem 4.3]{ST09}, Sawano and Tanaka proved that the Hardy-Littlewood  maximal operator $\M$ is bounded on vector valued block spaces $ \mathcal{H}_{p'}^{t',1} (\ell^{q'})$; see also \cite[Theorem 2.12]{IST15}. Their method also works for $\M _\eta $ with $1< \eta < \min\{ p' ,q' \}$. 
\end{rem}

Define $E_{j}$, the averaging operator at level $j\in\mathbb{Z}$,
acting on a locally integrable function $f$ on $\mathbb{R}^{n}$,
by
\[
E_{j}(f)=\sum_{Q\in\mathcal{D}_{j}}\bigg(\frac{1}{|Q|}\int_{Q}f(y)\mathrm{d}y\bigg)\chi_{Q}.
\]

The following lemma comes from  \cite[Theorem 3.7]{FraRou19}.
\begin{lem}
	\label{lem:Gamma_j le Lp l q} Suppose that $\{\gamma_{j}\}_{j\in\mathbb{Z}}$
	is a sequence of non-negative measurable functions on $\mathbb{R}^{n}$.
 Suppose that $0<q\le p<\infty$ and $\{\gamma_{j}\}_{j\in\mathbb{Z}}$
	satisfies
	\begin{equation*}
		\sup_{Q\in\mathcal{D}_j}\frac{1}{|Q|}\int_{Q}\gamma_{j}^{p(1+\delta)}(x)\mathrm{d}x\le c
	\end{equation*}
	for some $c,\delta>0,$ independent of $j\in\mathbb{Z}$. Then there
	exists $C>0$ such that for any sequence $\{f_{j}\}_{j\in\mathbb{Z}}$ of measurable
	functions on $\mathbb{R}^{n}$,
	\begin{equation*}
		\|\{\gamma_{j}E_{j}(f_{j})\}\|_{L^{p}(\ell^{q})}\le C\|\{E_{j}(f_{j})\}\|_{L^{p}(\ell^{q})}.
	\end{equation*}
\end{lem}

	The following result extends  Lemma \ref{lem:Gamma_j le Lp l q} to sequence valued Bourgain-Morrey spaces $M_{p}^{t,r}(\ell^{q})$.

	\begin{thm}
		\label{thm:gamma j} Let $0<p<t<r<\infty$ or $0<p\le t<r=\infty$.
		Suppose that $\{\gamma_{j}\}_{j\in\mathbb{Z}}$ is a sequence of non-negative measurable
		functions on $\mathbb{R}^{n} $ such that
		\begin{equation}
			\sup_{Q\in\mathcal{D}_j}\frac{1}{|Q|}\int_{Q}\gamma_{j}^{p +\delta }(x)\mathrm{d}x\le c\label{eq:condition 1}
		\end{equation}
		for some $c,\delta_W>0, \delta \in (0,\delta_W)$ independent of $j\in\mathbb{Z}$.

		{\rm (i)}
		Suppose that $0<q\le p<\infty$. 
		Then there
		exists $C>0$ such that for any sequence $\{f_{j}\}_{j\in\mathbb{Z}}$ of measurable
		functions on $\mathbb{R}^{n}$,
		\begin{equation}
			\|\{\gamma_{j}E_{j}(f_{j})\}\|_{M_{p}^{t,r}(\ell^{q})}\le C\|\{E_{j}(f_{j})\}\|_{M_{p}^{t,r}(\ell^{q})}.\label{eq:Ej 1}
		\end{equation}
		
		{\rm (ii)} 
		Suppose that $0< p< q   <  p + \delta_W  $. Then there
		exists $C>0$ such that for any sequence $\{f_{j}\}_{j\in\mathbb{Z}}$ of measurable
		functions on $\mathbb{R}^{n}$,
		\begin{equation*}
			\|\{\gamma_{j}E_{j}(f_{j})\}\|_{M_{p}^{t,r}(\ell^{q})}\le C\|\{E_{j}(f_{j})\}\|_{M_{p}^{t,r}(\ell^{q})} .
		\end{equation*}
		
	\end{thm}
	
	\begin{proof}
		Note that $E_j (f_j)$  is constant on each $Q \in \D_j$; we denote that constant value by $(E_j (f_j))_Q $.
		
		(i)
		Observer that if $0<p=t<r=\infty$, $M_{p}^{t,r}=L^{p}$. Then (\ref{eq:Ej 1})
		comes from  Lemma \ref{lem:Gamma_j le Lp l q}.
		Then we consider  the case of $ 0< q \le p <t <r \le \infty $. Note that in this case, we have $t/q >1$.
		By Lemma \ref{predual seq BM} (the scalar version, or \cite[Theorem 2.17]{M23} for $0<q=p <t<r \le \infty$), we have
		\begin{align*}
			\|\{\gamma_{j}E_{j}(f_{j})\}\|_{M_{p}^{t,r}(\ell^{q})}^{q} & =\bigg\|\sum_{j}|\gamma_{j}E_{j}(f_{j})|^{q}\bigg\|_{M_{p/q}^{t/q,r/q}}\\
			& =\sup_{\|g\|_{\mathcal{H}_{(p/q)'}^{(t/q)',(r/q)'}}=1}\bigg|\int_{\mathbb{R}^{n}}\sum_{j}|\gamma_{j}E_{j}(f_{j})|^{q}g\mathrm{d}x\bigg|\\
			& =\sup_{\|g\|_{\mathcal{H}_{(p/q)'}^{(t/q)',(r/q)'}}=1}\sum_{j}\sum_{Q\in\mathcal{D}_{j}}|(E_{j}(f_{j}))_{Q}|^{q}\int_{Q}\gamma_{j}^{q}|g|\mathrm{d}x.
		\end{align*}	
		Let $\delta \in (0,\delta_W)$.
		Using H\"{o}lder's inequality with exponents $w=( p+\delta )/q$, we obtain
		that for $Q\in\mathcal{D}_{j}$,
		\begin{align*}
			\frac{1}{|Q|}\int_{Q}\gamma_{j}^{q}|g|\mathrm{d}x 
			& \le c\frac{1}{|Q|}\int_{Q}\big[\mathcal{M}(|g|^{w'})\big]^{1/w'}\mathrm{d}x.
		\end{align*}
		Substituting above gives
		\begin{align*}
		& \sum_{j}\sum_{Q\in\mathcal{D}_{j}}|(E_{j}(f_{j}))_{Q}|^{q} \int_{Q}\big[\mathcal{M}(|g|^{w'})\big]^{1/w'}\mathrm{d}x \\
			& \lesssim 
			\int_\rn
			\sum_{j}\sum_{Q\in\mathcal{D}_{j}}|(E_{j}(f_{j}))_{Q}|^{q} \big[\mathcal{M}(|g|^{w'})\big]^{1/w'}\mathrm{d}x \\
			&  \lesssim  
			\bigg\| \sum_{j}\sum_{Q\in\mathcal{D}_{j}}|(E_{j}(f_{j}))_{Q}|^{q}
			\bigg\|_{ M_{p/q}^{t/q,r/q} }   
			\left\|  \big[\mathcal{M}(|g|^{w'})\big]^{1/w'}    \right\|_{ \mathcal{H}_{(p/q)'}^{(t/q)',(r/q)'} }.
		\end{align*}
		Note that $ 1< w' < (p/q)'$. Then by  Lemmas \ref{predual seq BM} and \ref{M eta r < infty = infty},
		we obtain
		\begin{align*}
			 \|\{\gamma_{j}E_{j}(f_{j})\}\|_{M_{p}^{t,r}(\ell^{q})}^{q} 
			 \lesssim \|\{E_{j}(f_{j})\}\|_{M_{p}^{t,r}(\ell^{q})}^{q}.
		\end{align*}

		(ii) 
		Let $ 0<A<p<\infty$.
		Then \begin{equation*}
			\|\{\gamma_{j}E_{j}(f_{j})\}  _{j \in \mathbb Z} \|_{M_{p}^{t,r}(\ell^{q})} ^A = \|\{\gamma_{j} ^ A |E_{j}(f_{j})| ^A \}  _{j \in \mathbb Z} \|_{M_{p/A}^{t/A,r/A}(\ell^{q/A})} .
		\end{equation*}
		By the condition (\ref{eq:condition 1}),  the sequence $\{\gamma_j^A \}_{j \in \mathbb Z}$ satisfies
		\begin{equation*}
			\sup_{ Q \in \D} \frac{1}{|Q|}  \int_Q   \left(  \gamma_j^{A}\right) ^{w/A }  <\infty 
		\end{equation*}
		for all $0< w < p +\delta_W$.  From Lemma \ref{predual seq BM}, we have
		\begin{align*}
			& \|\{\gamma_{j} ^ A |E_{j}(f_{j})| ^A \}  _{j \in \mathbb Z} \|_{M_{p/A}^{t/A,r/A}(\ell^{q/A})}  \\
			& = \sup_{  \| \vec g\|_{ \mathcal H_{(p/A)'} ^{(t/A)', (r/A)'} (\ell^{(q/A)'}) } \le 1 } \int_\rn \sum_j \gamma_j^A  | E_j (f_j)|^A g_j \d x \\
			& =  \sup_{  \| \vec g\|_{  \mathcal H_{(p/A)'} ^{(t/A)', (r/A)'} (\ell^{(q/A)'})} \le 1 }  \sum_j \sum_{Q \in \D_j}  ( | E_j (f_j)|^A )_Q \int_Q \gamma_j^A  g_j \d x .
		\end{align*}
		For $1 < w /A < ( p +\delta_W)/A $, we get 
		\begin{align*}
			\frac{1}{|Q|}  \int_Q  \gamma_j^A  g_j \d x 
			& \le  c\frac{1}{|Q|}\int_{Q}\big[\mathcal{M}(|g_j|^{(w/A)'})\big]^{1/(w/A)'}\mathrm{d}x.
		\end{align*}
		Substituting above gives
		\begin{align*}
			&    \sum_j \sum_{Q \in \D_j}  (| E_j (f_j)|^A  )_Q \int_{Q}\big[\mathcal{M}(|g_j|^{(w/A)'})\big]^{1/(w/A)'} \mathrm{d}x \\
			& \lesssim   
			\int_\rn
			\sum_{j}\sum_{Q\in\mathcal{D}_{j}} (| E_j (f_j)|^A  )_Q  \big[\mathcal{M}(|g_j|^{(w/A)'})\big]^{1/(w/A)'}  \mathrm{d}x \\
			& \le  \left\|  \{  | E_{j}(f_{j}))_{Q} |^A \}_{j\in \mathbb Z} \right\|_{M_{p/A}^{t/A,r/A}  (\ell^{q/A}) }   \left\|  \Big\{ \mathcal{M}_{(w/A)' }   (|g_j|) \Big\}_{j\in \mathbb Z} \right\|_{ \mathcal H_{(p/A)'} ^{(t/A)', (r/A)'} (\ell^{(q/A)'}) } .
		\end{align*}
		Since $ p<q < p+\delta_W$,  we  can select $w \in (q, p+\delta_W)$. Then  $1<  (w/A)' < (q/A) ' < (p/A)' $. Using Lemmas \ref{predual seq BM} and \ref{M eta r < infty = infty},
		we obtain 
		\begin{align*}
			\|\{\gamma_{j}E_{j}(f_{j})\}\|_{M_{p}^{t,r}(\ell^{q})} 
			\lesssim  \left\|  \{ E_{j}(f_{j}))_{Q} \}_{j\in \mathbb Z} \right\|_{M_p^{t,r}  (\ell^q) } . 
		\end{align*}
		Thus we prove (ii).
	\end{proof}
	
	\begin{cor}
		\label{cor:gamma le} Let $0<p<t<r<\infty$ or $0<p\le t<r=\infty$.  Let $W\in \mathcal  A_{p}$ and $\{A_{Q}\}_{Q\in\mathcal{D}}$
		be a sequence of reducing operators of order $p$ for $W$. For $j\in\mathbb{Z}$,
		let
		\[
		\gamma_{j}(x)=\sum_{Q\in\mathcal{D}_{j}}\|W^{1/p}(x)A_{Q}^{-1}\|\chi_{Q}(x).
		\]
		Let $q \in  (0,p + \delta_W ) $ where $\delta_W$  is the same as in Lemma \ref{lem:W AQ}. 
		Then there exists $C>0$ such that for any sequence $\{f_{j}\} _{j \in \mathbb Z}$ of
		measurable functions on $\mathbb{R}^{n}$,
		\begin{equation*}
			\|\{\gamma_{j}E_{j}(f_{j})\}  _{j \in \mathbb Z} \|_{M_{p}^{t,r}(\ell^{q})}\le C\|\{E_{j}(f_{j})\}  _{j \in \mathbb Z} \|_{M_{p}^{t,r}(\ell^{q})}. 
		\end{equation*}
	\end{cor}

	\begin{proof}
		Indeed, the  condition (\ref{eq:condition 1}) in Theorem  \ref{thm:gamma j} holds  for some $\delta_W >0$  by Lemma \ref{lem:W AQ}. Then the results follow from  Theorem \ref{thm:gamma j}.
	\end{proof}
		
		\begin{cor}\label{cor:W le AQ}
			Let $0<p<t<r<\infty$ or $0<p\le t<r=\infty$. Suppose that $s\in\mathbb{R}$,
			$W\in  \mathcal A_{p}$ and $\{A_{Q}\}_{Q\in\mathcal{D}}$ is a sequence of reducing
			operators of order $p$ for $W$. 
			Let $q \in (0, p+ \delta_W )$ where $\delta_W$  is the same as in Lemma \ref{lem:W AQ}. 
			Then there exists a constant $c>0$
			such that for any sequence $\vec{s}=\{\vec{s}_{Q}\}_{Q\in\mathcal{D}}$,
			\begin{equation}
				\|\vec{s}\|_{\dot{f}_{p,t,r}^{s,q}(W)}\le c\|\vec{s}\|_{\dot{f}_{p,t,r}^{s,q}(A_{Q})},\label{eq:discrete Triebel W le AQ}
			\end{equation}
			and, for any $\vec{f}\in\mathscr{S}'/\mathcal{P}(\mathbb{R}^{n})$,
			\begin{equation}
				\|\vec{f}\|_{\dot{F}_{p,t,r}^{s,q}(W)}\le c\|\vec{f}\|_{\dot{F}_{p,t,r}^{s,q}(A_{Q})} .\label{eq:F W le AQ}
			\end{equation}
		\end{cor}
		
		\begin{proof}
			We first prove (\ref{eq:discrete Triebel W le AQ}).
			Let
			\[
			\gamma_{j}(x)=\sum_{Q\in\mathcal{D}_{j}}\|W^{1/p}(x)A_{Q}^{-1}\|\chi_{Q}(x).
			\]
			For $\vec{s}=\{\vec{s}_{Q}\}_{Q\in\mathcal{D}}$, define
			\[
			f_{j}=\sum_{Q\in\mathcal{D}_{j}}|Q|^{-s/n-1/2}|A_{Q}\vec{s}_{Q}|\chi_{Q},\ j\in \mathbb{Z}.
			\]
			Note that $f_{j}$ is constant on each $Q\in\mathcal{D}_{j}$, hence, $E_j (f_j) = f_j$ . Let
			\[
			g_{j}=\sum_{Q\in\mathcal{D}_{j}}|Q|^{-s/n-1/2}|W^{1/p}\vec{s}_{Q}|\chi_{Q}.
			\]
			Then $g_{j}\le\gamma_{j}f_{j}=\gamma_{j}E_{j}(f_{j})$. Thus by Corollary
			\ref{cor:gamma le}, we have
			\begin{align*}
				\|\vec{s}\|_{\dot{f}_{p,t,r}^{s,q}(W)} & =\|\{g_{j}\}\|_{M_{p}^{t,r}(\ell^{q})}\le\|\{\gamma_{j}E_{j}(f_{j})\}\|_{M_{p}^{t,r}(\ell^{q})}
				\\
				&\lesssim \|\{E_{j}(f_{j})\}\|_{M_{p}^{t,r}(\ell^{q})}=\|\vec{s}\|_{\dot{f}_{p,t,r}^{s,q}(A_{Q})}.
			\end{align*}
			This proves (\ref{eq:discrete Triebel W le AQ}). Next we  show (\ref{eq:F W le AQ}).

			Define $h_{j}(x)=2^{js}|W^{1/p}(x)\varphi_{j}*\vec{f}(x)|$ and
			\[
			k_{j}=\sum_{Q\in\mathcal{D}_{j}}|Q|^{-s/n}\Big(\sup_{x\in Q}|A_{Q}\varphi_{j}*\vec{f}(x)|\Big)\chi_{Q}
			\]
			for $j\in\mathbb{Z}$. Note that each $k_{j}$ is constant on each
			$Q\in\mathcal{D}_{j}$. Then $h_{j}\le\gamma_{j}k_{j}$. Hence by
			Corollary \ref{cor:gamma le}, we have
			\begin{align*}
				\|\vec{f}\|_{\dot{F}_{p,t,r}^{s,q}(W)} & =\|\{h_{j}\}\|_{M_{p}^{t,r}(\ell^{q})}\le\|\{\gamma_{j}E_{j}(k_{j})\}\|_{M_{p}^{t,r}(\ell^{q})} \\
				& \lesssim \|\{E_{j}(k_{j})\}\|_{M_{p}^{t,r}(\ell^{q})}\le\|\vec{f}\|_{\dot{F}_{p,t,r}^{s,q}(A_{Q})}
			\end{align*}
			where the last step was used Theorem \ref{thm:sup A_Q LE}.
			Hence we complete the proof.
		\end{proof}
		
		Now, Theorems \ref{thm:A_Q le W discrete} and \ref{thm:Aq le W function} and Corollary \ref{cor:W le AQ} yield Theorem \ref{thm:equivalence}.
		
	\subsection{Calder\'{o}n-Zygmund operators}		
		First, we recall the definition of almost diagonal operator.
				
		\begin{defn}
			\label{def:a d ope} Suppose that $T$ is a linear operator mapping
			$\mathscr{S}(\mathbb{R}^{n})$ continuously to $\mathscr{S}'(\mathbb{R}^{n})$.
			Let $0<p<\infty$, $0<q\le\infty$, $s\in\mathbb{R}$. 
			Let $W\in \mathcal A_{p}$  and let $\{A_Q\}_{Q \in \D}$  be a family of reducing operators of order $p$ for $W$. Let $d, \tilde d, \Delta$  be the same as in Lemma \ref{AQ AR improved}.
			We say that $T$ is almost diagonal on $\dot{F}_{p,t,r}^{s,q}(W)$
			and write $T \in \mathbf{AD}_{p}^{s,q}(d, \tilde d, \Delta  )$
			if there exists $\varphi\in\mathcal{A}$
			such that the matrix $A=(a_{QP})_{Q,P\in\mathcal{D}}$ defined by
			$a_{QP}=\langle T\psi_{P},\varphi_{Q}\rangle$ belongs to $\mathbf{ad}_{p}^{s,q}(d, \tilde d, \Delta  )$,
			where $\varphi_{Q}$ and $\psi_{Q}$ are the same as (\ref{eq:varphi  Q}).
		\end{defn}
		
		\begin{rem}
			Let $0 < p<t<r<\infty$. Let 
			$s\in\mathbb{R}$, $W\in\mathcal A_{p}$ with the $\mathcal  A_p$-dimension $d \in [0,n)$. Let $q \in (0, p+\delta_W)$ where  $\delta_W >0$  be the same as in Lemma \ref{lem:W AQ}. Then $\mathscr{S}_{0}(\mathbb{R}^{n})$
			is dense in $\dot{F}_{p,t,r}^{s,q}(W)$. Indeed, let $\vec{f}\in\dot{F}_{p,t,r}^{s,q}(W)$
			have $\varphi$-transform representation $\vec{f}=\sum_{Q}\vec{s}_{Q}\psi_{Q}$
			with $\vec{s}_{Q}=\langle\vec{f},\varphi_{Q}\rangle$. Then $\vec{s}\in\dot{f}_{p,t,r}^{s,q}(A_{Q})$
			by Theorem \ref{thm:equivalence}, where $\{A_{Q}\}_{Q\in\mathcal{D}}$ is a sequence
			of  reducing operators of order $p$ for $W$. For $N\in\mathbb{N},$
			let
			\[
			\mathcal{D}_{N}:=\{Q\in\mathcal{D}:2^{-N}\le\ell(Q)\le2^{N}\;\mathrm{and}\;|x_{Q}|\le N\}.
			\]
			Define sequences $\vec{s}_{N}=((\vec{s}_{N})_{Q})_{Q\in\mathcal{D}}$
			by
			\[
			(\vec{s}_{N})_{Q}=\begin{cases}
				\vec{s}_{Q} & \mathrm{if}\;Q\in\mathcal{D}(N),\\
				0 & \mathrm{if}\;Q\notin\mathcal{D}(N).
			\end{cases}
			\]
			Let $\vec{f}_{N}=\sum_{Q\in\mathcal{D}}(\vec{s}_{N})_{Q}\psi_{Q}$.
			Then $\vec{f}_{N}\in\mathscr{S}_{0}(\mathbb{R}^{n})$ and
			\[
			\|\vec{f}-\vec{f}_{N}\|_{\dot{F}_{p,t,r}^{s,q}(W)}=\Big\|\sum_{Q\in\mathcal{D}\backslash\mathcal{D}(N)}s_{Q}\psi_{Q}\Big\|_{\dot{F}_{p,t,r}^{s,q}(W)}\lesssim\|\vec{s}-\vec{s}_{N}\|_{\dot{f}_{p,t,r}^{s,q}(A_{Q})}\to0,
			\]
			as $N\to\infty$, by Theorems \ref{thm:equivalence} and \ref{thm:varphi transform}
			and the dominated convergence theorem. 
		\end{rem}
		
		\begin{lem}
			\label{lem:ad ope on F} Suppose that $0 <  p<t<r<\infty$.
			Let $W\in \mathcal A_{p}$ and $s\in\mathbb{R}$.
			Let $d, \tilde d, \Delta$  be the same as in Lemma \ref{AQ AR improved}.
			Let $ q \in (0, p + \delta_W)  $ where $\delta_W >0$  is the same as in Lemma \ref{lem:W AQ}.
			If $T  \in \mathbf{AD}_{p}^{s,q}(d, \tilde d, \Delta  )$, 
			then there
			exists a constant $c>0$, depending only $n,s,q,p, d, \tilde d, \Delta$, such that
			\begin{equation}
				\|T\vec{f}\|_{\dot{F}_{p,t,r}^{s,q}(W)}\le c\|\vec{f}\|_{\dot{F}_{p,t,r}^{s,q}(W)}\label{eq:T triebel}
			\end{equation}
			for all $\vec{f}\in  \mathscr{S}_{0}(\mathbb{R}^{n}) $. Hence, $T$
			can be extended to be a continuous linear operator on $\dot{F}_{p,t,r}^{s,q}(W)$.
		\end{lem}
		
		\begin{proof}
			Let $\varphi,\psi,a_{QP},A=(a_{QP})_{Q,P\in\mathcal{D}}$ be as in Definition \ref{def:a d ope}.
			Let $\{A_{Q}\}_{Q \in \mathcal D}$ be a  sequence of reducing operators of order $p$
			for $W$. Since $\vec{f}\in\mathscr{S}_{0}(\mathbb{R}^{n})$, the
			$\varphi$-transform representation $\vec{f}=\sum_{P\in\mathcal{D}}\vec{s}_{P}\psi_{P}$
			converges in $\mathscr{S}(\mathbb{R}^{n})$, where $\vec{s}_{P}=\langle\vec{f},\varphi_{P}\rangle$.
			Let $\vec{s}=(\vec{s}_{P})_{P\in\mathcal{D}}$. Since $T$ is a linear
			operator mapping $\mathscr{S}(\mathbb{R}^{n})$ continuously to $\mathscr{S}'(\mathbb{R}^{n})$,
			we have
			\[
			\langle T\vec{f},\varphi_{Q}\rangle=\bigg\langle\sum_{P\in\mathcal{D}}\vec{s}_{P}T\psi_{P},\varphi_{Q}\bigg\rangle=\sum_{P\in\mathcal{D}}a_{QP}\vec{s}_{P}=(A\vec{s})_{Q}.
			\]
			From Theorem \ref{thm:almost dia bounded on discrete}, $A$ is bounded
			on $\dot{f}_{p,t,r}^{s,q}(A_{Q})$.
			Hence, by Theorem \ref{thm:equivalence},
			\begin{align*}
				\|T\vec{f}\|_{\dot{F}_{p,t,r}^{s,q}(W)} & \le c\|\{\langle T\vec{f},\varphi_{Q}\rangle\}_{Q\in\mathcal{D}}\|_{\dot{f}_{p,t,r}^{s,q}(A_{Q})}=c\|A\vec{s}\|_{\dot{f}_{p,t,r}^{s,q}(A_{Q})}\\
				& \le c\|\vec{s}\|_{\dot{f}_{p,t,r}^{s,q}(A_{Q})}\le c\|\vec{f}\|_{\dot{F}_{p,t,r}^{s,q}(W)}.
			\end{align*}
			Thus, we prove (\ref{eq:T triebel}). 
		\end{proof}
		
		For $N\in\mathbb{N}_{0}$, a function $a_{Q}\in\mathcal{D}(\mathbb{R}^{n})$ (the
		space of compactly supported $C^{\infty}$-functions) is a smooth
		$N$-atom for $Q$ if supp $a_{Q}\subset3Q$, $\int x^{\gamma}a_{Q}=0$
		for all $|\gamma|\le N$, and $|\partial^{\gamma}a_{Q}(x)|\le c_{\gamma}\ell(Q)^{-|\gamma|-n/2}$
		for all multi-indices $\gamma$.
		\begin{lem}
			\label{lem:T atom molecule} Suppose that
			$T$ is a linear operator mapping $\mathscr{S}(\mathbb{R}^{n})$ continuously
			to $\mathscr{S}'(\mathbb{R}^{n})$. Suppose that $0<p<\infty$,
			$s\in\mathbb{R}$, and  $W\in \mathcal A_{p}$.  Let $d, \tilde d, \Delta$  be the same as in Lemma \ref{AQ AR improved}. Let $ q \in (0, p + \delta_W)  $ where $\delta_W >0$  is the same as in Lemma \ref{lem:W AQ}.
			Suppose that $N,K\in\mathbb{Z}$,
			$M>0$ and $\delta\in(0,1]$ satisfy $ N > -s+ n/\min(1,q,p) + \tilde d / p'- 1- n$,
			$K +\delta > s  + d/p  $, and $ M >  n / \min(1,q,p)+  \Delta$. 
			Suppose that there exists $N_{0}$
			such that for each smooth $N_{0}$-atom for $Q,$ the molecule $m_{Q}=Ta_{Q}$
			satisfies (M1), (M2), (M3), (M4) in Definition \ref{def:molecules},
			with a constant $c$ independent of $Q$. Then $T\in\mathbf{AD}_{p}^{s,q}(d, \tilde d, \Delta  )$, and hence, extends to be a bounded linear operator on $ \dot{F}_{p,t,r}^{s,q}(W) $.
		\end{lem}
		
		\begin{proof}
			We use the idea from \cite[Lemma 5.3]{FraRou19}.
			Following the proof of \cite[Lemma 9.2]{R03}, we have $\psi_P = \sum_{S \in \D}  t_{SP} a_S^{ (P) } $ where 
			\begin{equation*}
				t_{SP} = |S|^{1/2}  \sup_{y \in S}  |\varphi_j *  \psi_P (y) |
			\end{equation*}
			for $\ell (S) = 2^{-j}$, and each $a_S^{ (P) } $ is a smooth $N_0$-atom for $S$. By assumption, $m_S^{(P)} = T a_S^{ (P) } $ satisfies conditions (M1)-(M4)  with $Q$ replaced by $S$, in Definition \ref{def:molecules}. Then for $Q \in \D_j $,
			\begin{equation*}
				\langle m_S^{(P)}, \varphi_Q \rangle = 2^{-jn/2} \tilde{\varphi} _j  *  m_S^{(P)} (x_Q)
			\end{equation*}
			where $\tilde \varphi (x) = \bar \varphi (-x)$. Since $\tilde \varphi \in \mathcal A$, from Lemma \ref{lem:mQ} and the proof of Theorem \ref{thm:varphi transform}, we have
			\begin{equation*}
				| 	\langle m_S^{(P)}, \varphi_Q \rangle  | \lesssim \omega_{QS}
			\end{equation*}
			where  $\omega_{QS}$ defined as in (\ref{eq:omega QP-1}). Hence,
			\begin{equation*}
				|\langle T \psi_P , \varphi_Q \rangle | = \left|  \sum_{S \in \D} \langle t_{SP} T a_S^{ (P)} , \varphi_Q \rangle  \right| \le  \left|  \sum_{S \in \D} \langle  m_S^{(P)} , \varphi_Q \rangle  t_{SP} \right| \lesssim  \sum_{S \in \D} \omega_{QS} t_{SP}.
			\end{equation*}
			From the proof of \cite[Lemma 5.3]{FraRou19}, we have 
			\begin{equation*}
				t_{SP}   \le  C_M \left( 1+ \frac{ |x_S - x_P|  }{ \max(\ell (S) , \ell (P)  )  } \right)^{-M} .
			\end{equation*}
			Using the fact (see \cite[Lemma D.1]{FJ90}) that
			\begin{align*}
				& \sum_{S \in \D_j }  \left( 1+  \frac{ |x_Q - x_S| }{\max (\ell (Q), \ell (S)  )}\right) ^{-R} \left( 1+  \frac{ |x_S - x_P| }{\max (\ell (S), \ell (P)  )}\right) ^{-R}  \\
				& \lesssim  \left( 1+  \frac{ |x_Q - x_P| }{\max (\ell (Q), \ell (P)  )}\right) ^{-R} ,
			\end{align*}
			we obtain, with $\ell (P)  =2^{-k}$,
			\begin{align*}
				&	|\langle T \psi_P , \varphi_Q \rangle | \\
				& \lesssim \sum_{j = k-1}^{k+1}  \sum_{\ell (S) =2^{-j}} \Big(\frac{\ell(Q)}{\ell(S)}\Big)^{s} \\
				& \quad \times \min\bigg\{\Big(\frac{\ell(S)}{\ell(Q)}\Big)^{(n+\epsilon)/2+n/\min(1,q,p)-n + \tilde d / p' },\Big(\frac{\ell(Q)}{\ell(S)}\Big)^{(n+\epsilon)/2 + d/p }\bigg\}\\
				& \quad  \times  \bigg(1+\frac{|x_{Q}-x_{S}|}{\max\{\ell(Q),\ell(S)\}}\bigg)^{-\frac{n}{\min(1,q,p)}-\epsilon - \Delta} \left( 1+ \frac{ |x_S - x_P| }{ \max(\ell (S) , \ell (P)  )  } \right)^{-M} \\
				& \lesssim \Big(\frac{\ell(Q)}{\ell(P)}\Big)^{s}\min\bigg\{\Big(\frac{\ell(P)}{\ell(Q)}\Big)^{(n+\epsilon)/2+n/\min(1,q,p)-n + \tilde d / p' },\Big(\frac{\ell(Q)}{\ell(P)}\Big)^{(n+\epsilon)/2 + d/p }\bigg\}\\
				& \quad \times  \bigg(1+\frac{|x_{Q}-x_{P}|}{\max\{\ell(Q),\ell(P)\}}\bigg)^{-\frac{n}{\min(1,q,p)}-\epsilon - \Delta} = \omega_{QP} .
			\end{align*}
			Hence $T \in \mathbf{AD}_{p}^{s,q}(d, \tilde d, \Delta  )$.
		\end{proof} 
		
		\begin{rem}
			An argument similar to that in the last proof shows that if $	|\langle T \psi_P , \varphi_Q \rangle | \in \mathbf{ad}_{p}^{s,q}(d, \tilde d, \Delta  ) $ for one choice of $\varphi \in \mathcal A$, then the same is true for any other $\varphi \in\mathcal A$, as in the proof of \cite[Remark 5.4]{FraRou19} and \cite[Remark 8.8]{R03}. Hence, the definition of $  \mathbf{AD}_{p}^{s,q}(d, \tilde d, \Delta  )$  is independent of the	choice of  $\varphi \in \mathcal A$.
		\end{rem}

		\begin{defn}
			Let $L\in\mathbb{N}$. We say that $T$ is an $L$-smooth classical
			Calder\'{o}n-Zygmund operator on $\mathbb{R}^{n}$ if
			\[
			Tf(x)=\lim_{\epsilon\to0^{+}}\int_{\mathbb{R}^{n}\backslash B(0,\epsilon)}K(y)f(x-y)\mathrm{d}y,
			\]
			where the kernel $K$ satisfies
			\begin{align*}
				(K1) & \;|K(x)|\le\frac{c}{|x|^{n}},\;\mathrm{for}\;x\in\mathbb{R}^{n}\backslash\{0\},\\
				(K2) & \;|\partial^{\gamma}K(x)|\le\frac{c}{|x|^{n+|\gamma|}}\;\mathrm{for}\;x\in\mathbb{R}^{n}\backslash\{0\},|\gamma|\le L,\\
				(K3) & \;\int_{R_{1}<|x|<R_{2}}K(x)\mathrm{d}x=0,\;\mathrm{for}\;\mathrm{all}\;0<R_{1}<R_{2}<\infty.
			\end{align*}
		\end{defn}
		
		Let $m \in\mathbb{N}$ and $T$ be an $L$-smooth classical Calder\'{o}n-Zygmund
		operator on $\mathbb{R}^{n}$. Similarly, we define $T\vec{f}=(Tf_{1},Tf_{2},\ldots,Tf_{m})^{T}$.

		\begin{lem}
			[Lemma 5.7, \cite{FraRou19}] \label{lem:L s c CZ} Let $N_{0}\in\mathbb{N}_{0}$
			and $L\in\mathbb{N}$. Let $a_{Q}$ be a smooth $N_{0}$-atom for
			$Q\in\mathcal{D}.$ Let $T$ be an $L$-smooth classical Calder\'{o}n-Zygmund
			operator on $\mathbb{R}^{n}$. Then $m_{Q}=Ta_{Q}$ satisfies
			\begin{equation*}
				\int_\rn x^{\gamma}m_{Q}(x) \d x=0,\;\mathrm{for}\;|\gamma|\le N_{0},
			\end{equation*}
			and
			\begin{equation}
				|\partial^{\gamma}m_{Q}(x)|\le c|Q|^{-1/2-|\gamma|/n}\bigg(1+\frac{|x-x_{Q}|}{\ell(Q)}\bigg)^{-n-L},\;\mathrm{for}\;|\gamma|\le L.\label{eq:parital mQ}
			\end{equation}
		\end{lem}
		
		\begin{thm} \label{CZ}
			Let $0<p<t<r<\infty$ and $s\in\mathbb{R}$. Let $W\in \mathcal A_{p}$ with the $\mathcal A_p$-dimension $d \in [0,n)$.  Let $ \tilde d, \Delta$  be the same as in Lemma \ref{AQ AR improved}.
			Let  $q\in (0, p + \delta_W) $  where  $ \delta_W > 0$ is the same as in Lemma \ref{lem:W AQ}.
			Suppose that $L\in\mathbb{N}$ with
			\begin{equation*}
				L > \max \left\{ s  + d/p ,  -s+ n/\min(1,q,p) +\tilde d / p' -n , \frac{n}{\min(1,q,p)}+  \Delta - n  \right\} .
			\end{equation*}
			If $T$ is an $L$-smooth classical
			Calder\'{o}n-Zygmund operator on $\mathbb{R}^{n}$, then $T$ extends
			to be a bounded operator on $\dot{F}_{p,t,r}^{s,q}(W)$.
		\end{thm}
		
		\begin{proof}
			Let $\delta=1$ and $K=L-1$. Let $a_{Q}$ be a smooth $N_{0}$-atom
			for $Q$, where $N_{0}>-s+ n/\min(1,q,p) + \tilde d / p' - 1- n$. Then we
			show that $m_{Q}=Ta_{Q}$ satisfies (M1)-(M4) in Definition \ref{def:molecules}
			for $\delta=1$, $N> -s+ n/\min(1,q,p) +\tilde d / p'- 1- n$, $K+\delta>s  + d/p$,
			and $M> n/ \min(1,q,p) +  \Delta$.
			
			By Lemma \ref{lem:L s c CZ}, $m_{Q}$ satisfies (M1) for $N=N_{0}$
			and (M3) with $M=L+n> n / \min(1,q,p)+  \Delta$ for $|\gamma| \le L$. Inequality (\ref{eq:parital mQ})
			with $|\gamma|=L$, followed by the mean value theorem, gives (M4)
			for $\delta=1$ and $K=L-1$; note that $K+1=L>s  +  d/p $. Inequality
			(\ref{eq:parital mQ}) with $|\gamma|=0$ and $L>\max\{ -s+ n/\min(1,q,p) + \tilde d / p'  -n , n/ \min(1,q,p)+  \Delta  -n \}$
			gives (M2) with $N>-s+ n/\min(1,q,p) +\tilde d / p' - 1- n$. 
			
			The result follows
			by Lemma \ref{lem:T atom molecule}.
		\end{proof}
		
		\begin{rem}
			The Hilbert transform in $\mathbb R$ and the Riesz transforms in $\mathbb R^n, n \ge 2$ (see, for instance, \cite[Definition 5.1.13]{G14}) are $L$-smooth classical Calder\'{o}n-Zygmund operator for all $L \in \mathbb N$, and hence, extend to be bounded on $\dot{F}_{p,t,r}^{s,q}(W)$ with the same conditions in Theorem \ref{CZ}.
		\end{rem}
		
		\section{Inhomogeneous matrix weighted Bourgain-Morrey Triebel-Lizorkin spaces} \label{sec:Inhomogeneous-spaces}
		
		Let $\Phi,\varphi\in\mathscr{S}(\mathbb{R}^{n})$. We say that $(\Phi,\varphi)$
		is admissible, and write $(\Phi,\varphi)\in\mathcal{A}_{+}$, if
		\[
		\mathrm{supp}\;\mathcal{F}(\Phi)\subset\{\xi \in \rn :|\xi|\le2\},
		\]
		\[
		\mathcal{F}(\Phi)(\xi)\ge c>0\quad\mathrm{if}\;|\xi|\le5/3,
		\]
		\[
		\mathrm{supp}\;\mathcal{F}(\varphi)\subset\{\xi\in \rn:1/2\le|\xi|\le2\},
		\]
		\[
		\mathcal{F}(\varphi)(\xi)\ge c>0\quad\mathrm{if}\;3/5\le|\xi|\le5/3.
		\]
		For $j\in\mathbb{Z}$, from \cite[Section 11]{R03}, for $(\Phi,\varphi)\in\mathcal{A}_{+}$,
		we can find $(\Psi,\psi)\in\mathcal{A}_{+}$ such that
		\[
		\overline{\mathcal{F}(\Phi)}\mathcal{F}(\Psi)+\sum_{j=1}^{\infty}\overline{\mathcal{F}(\varphi_{j})}\mathcal{F}(\psi_{j})(\xi)=1,\;\xi\in\mathbb{R}^{n}.
		\]
		Let $\varphi_{0}=\Phi$. Let $\mathcal{D}_{+}:=\{Q\in\mathcal{D}:\ell(Q)\le1\}$.
		
		For the following function spaces, we suppose that $s\in\mathbb{R},$
		$0<q\le\infty$, $W$ is a matrix weight. Let $0<p<t<r<\infty$ or
		$0<p\le t<r=\infty$.
		
		(i) The matrix weighted inhomogeneous Bourgain-Morrey Triebel-Lizorkin
		spaces $F_{p,t,r}^{s,q}(W)$ is the set of all distribution $\vec{f}\in\mathscr{S}'(\mathbb{R}^{n})$
		such that
		\[
		\|\vec{f}\|_{F_{p,t,r}^{s,q}(W)}:=\bigg\|\bigg(\sum_{v=0}^{\infty}2^{vsq}|W^{1/p}\varphi_{v}*\vec{f}|^{q}\bigg)^{1/q}\bigg\|_{M_{p}^{t,r}}<\infty.
		\]

		(ii) The discrete inhomogeneous Bourgain-Morrey Triebel-Lizorkin spaces	$f_{p,t,r}^{s,q}(W)$ is the set of all sequences $\vec{s}=\{\vec{s}_{Q}\}_{Q\in\mathcal{D}_{+}}$
		such that
		\[
		\|\vec{s}\|_{f_{p,t,r}^{s,q}(W)}:=\bigg\|\bigg(\sum_{Q\in\mathcal{D}_{+}}\big[|Q|^{-s/n-1/2}|W^{1/p}\vec{s}_{Q}|\chi_{Q}\big]^{q}\bigg)^{1/q}\bigg\|_{M_{p}^{t,r}}<\infty.
		\]

		Suppose that for each $Q\in\mathcal{D}$, $A_{Q}$ is a $m \times m$ non-negative definite matrix.
		
		(iii) The \{$A_{Q}$\}-inhomogeneous Bourgain-Morrey Triebel-Lizorkin
		spaces $F_{p,t,r}^{s,q}(A_{Q})$ is the set of all distribution $\vec{f}\in\mathscr{S}'(\mathbb{R}^{n})$
		such that
		\[
		\|\vec{f}\|_{F_{p,t,r}^{s,q}(A_{Q})}:=\bigg\|\bigg(\sum_{v=0}^{\infty}\sum_{Q\in\mathcal{D}_{v}}2^{vsq}|A_{Q}\varphi_{v}*\vec{f}|^{q}\chi_{Q}\bigg)^{1/q}\bigg\|_{M_{p}^{t,r}}<\infty.
		\]

		(iv) The \{$A_{Q}$\}-inhomogeneous discrete Bourgain-Morrey Triebel-Lizorkin
		spaces $f_{p,t,r}^{s,q}(A_{Q})$ is the set of all distribution $\vec{s}=\{\vec{s}_{Q}\}_{Q\in\mathcal{D}_{+}}$
		such that
		\[
		\|\vec{s}\|_{f_{p,t,r}^{s,q}(A_{Q})}:=\bigg\|\bigg(\sum_{v=0}^{\infty}\sum_{Q\in\mathcal{D}_{v}}\Big[|Q|^{-s/n-1/2}|A_{Q}\vec{s}_{Q}|\chi_{Q}\Big]^{q}\bigg)^{1/q}\bigg\|_{M_{p}^{t,r}}<\infty.
		\]

		We have the following analogues of Theorem \ref{thm:equivalence} for the inhomogeneous spaces.

		\begin{thm}
			\label{thm:equivalence inhomogeneous} 
			Let $0<p<t<r<\infty$ or $0<p\le t<r=\infty$.
			Let $s\in\mathbb{R}$, $W\in  \mathcal  A_{p}$, and
			$\{A_{Q}\}_{Q\in\mathcal{D}}$ be a sequence of reducing operators of order p for $W$. 
			Let $ q \in (0, p + \delta_W) $ where $\delta_W >0$  is the same as in Lemma \ref{lem:W AQ}. 
			Let $(\Phi,\varphi)\in\mathcal{A}_{+}$ and let $\varphi_{0}=\Phi$. For $\vec{f}\in\mathscr{S}'(\mathbb{R}^{n})$, let $\vec{s}=\{\vec{s}_{Q}\}_{Q\in\mathcal{D}_{+}}$,
			where $\vec{s}_{Q}=\langle\vec{f},\varphi_{Q}\rangle.$ 
			Then
			\[
			\|\vec{f}\|_{F_{p,t,r}^{s,q}(W)}\approx\|\vec{f}\|_{F_{p,t,r}^{s,q}(A_{Q})}\approx\|\vec{s}\|_{f_{p,t,r}^{s,q}(A_{Q})}\approx\|\vec{s}\|_{f_{p,t,r}^{s,q}(W)}.
			\]
			In addition, function spaces $F_{p,t,r}^{s,q}(W)$, $F_{p,t,r}^{s,q}(A_{Q})$
			are independent of
			the choice of $(\Phi,\varphi)\in\mathcal{A}_{+},$ in the sense that
			different choices yield equivalent quasi-norms.
		\end{thm}

		The proof of Theorem \ref{thm:equivalence inhomogeneous} 
		is virtually the same as for the homogeneous spaces and we omit the
		detail here.

		\section{Characterizations}\label{sec:Characterizations}
		
		In this section, we study some characterizations of matrix weighted Bourgain-Morrey Triebel-Lizorkin spaces in terms of Peetre type maximal functions, Lusin-area functions, Littlewood-Paley $g_\lambda^*$-functions, wavelets and atoms.
		
		\subsection{Characterizations by Peetre type maximal functions}
		
		Let $p\in(0,\infty)$, $a>0$, and $W\in \mathcal A_{p}$. The matrix-weighted
		Peetre-type maximal functions associated with $\{\varphi_{j}\}_{j\in\mathbb{Z}}$
		are defined by
		\begin{equation*}
			\Big(\varphi_{j}^{*}\vec{f}\Big)_{a}^{(W,p)}(x):=\sup_{y\in\mathbb{R}^{n}}\frac{|W^{1/p}(x)(\varphi_{j}*\vec{f})(y)|}{(1+2^{j}|x-y|)^{a}}
		\end{equation*}
		for any $j\in\mathbb{Z}$ and $x\in\mathbb{R}^{n}$.
		\begin{lem}[(3.10), \cite{WYZ22}] \label{lem:peetre le sum} Let $\varphi\in\mathcal{A}$ and $a>0$.
			Suppose that $\{A_{Q}\}_{Q\in\mathcal{D}}$ is a weakly doubling sequence of order $ \Delta \in(0,\infty)$
			of positive definite matrices.  Then for any $A\in(0,1]$,
			there exists a positive constant $c$, depending
			on both $A,$ and $\{A_{Q}\}_{ Q \in \D} $ such that for any $j\in\mathbb{Z},$
			$k\in\mathbb{Z}^{n}$, $x\in Q_{j, k}$,
			\begin{align*}
			&	\sup_{y\in\mathbb{R}^{n}}\frac{|A_{Q_{j, k}}(\varphi_{j}*\vec{f})(y)|^{^{A}}}{(1+2^{j}|x-y|)^{aA}} \\
			& \leq c\sum_{\ell\in\mathbb{Z}^{n}}(1+|k-\ell|)^{-A(a- \Delta)}2^{jn}\int_{Q_{j\ell}}\Big|A_{Q_{j\ell}}(\varphi_{j}*\vec{f})(z)\Big|^{A}\mathrm{d}z.
			\end{align*}
		\end{lem}
		
		\begin{thm}
			\label{thm:Peetre char} 
			Let $0<p<t<r<\infty$ or $0<p\le t<r=\infty$.
			Let $s\in\mathbb{R}$, $W\in  \mathcal  A_{p}$. 
			Let $ q \in (0, p + \delta_W) $ where $\delta_W >0$  is the same as in Lemma \ref{lem:W AQ}. 
			If $a>n/\min(1,q,p )+\Delta$ where $\Delta$  is the same as in Lemma \ref{AQ AR improved}, then
			\begin{equation*}
				\|\vec{f}\|_{\dot{F}_{p,t,r}^{s,q}(W)}\approx\|\vec{f}\|_{\dot{F}_{p,t,r}^{s,q}(W)}^{\star}
			\end{equation*}
			where
			\begin{equation*}
				\|\vec{f}\|_{\dot{F}_{p,t,r}^{s,q}(W)}^{\star}:=\bigg\|\bigg\{\sum_{j\in\mathbb{Z}}2^{jsq}\bigg[\big(\varphi_{j}^{*}\vec{f}\big)_{a}^{(W,p)}\bigg]^{q}\bigg\}^{1/q}\bigg\|_{M_{p}^{t,r}}.
			\end{equation*}

		\end{thm}
		
		\begin{proof}
			By the definition of $\Big(\varphi_{j}^{*}\vec{f}\Big)_{a}^{(W,p)}$, we have $\|\vec{f}\|_{\dot{F}_{p,t,r}^{s,q}(W)}\le\|\vec{f}\|_{\dot{F}_{p,t,r}^{s,q}(W)}^{\star}.$
			It remains to prove that, for any $\vec{f}\in\mathscr{S}'/\mathcal{P}(\mathbb{R}^{n})$,
			\begin{equation}
				\|\vec{f}\|_{\dot{F}_{p,t,r}^{s,q}(W)}\lesssim\|\vec{f}\|_{\dot{F}_{p,t,r}^{s,q}(W)}^{\star}.\label{eq:le W peetre}
			\end{equation}
			Let $\{A_{Q}\}_{Q\in\mathcal{D}}$ be a sequence of reducing operators
			of order $p$ for $W$. 
			By Lemma \ref{AQ AR improved},  $\{A_{Q}\}_{Q\in\mathcal{D}}$ is strongly doubling of order $ (d, \tilde d, \Delta,p)$ where $ d, \tilde d, \Delta,p$  is the same as in Lemma \ref{AQ AR improved}.
			We set
			\[
			\|\vec{f}\|_{\dot{F}_{p,t,r}^{s,q}(A_{Q})}^{\star}:=\bigg\|\bigg\{\sum_{j\in\mathbb{Z}}\sum_{Q\in\mathcal{D}_{j}}2^{jsq}\sup_{y\in\mathbb{R}^{n}}\frac{|A_{Q}(\varphi_{j}*\vec{f})(y)|^{q}}{(1+2^{j}|\cdot-y|)^{aq}}\chi_{Q}\bigg\}^{1/q}\bigg\|_{M_{p}^{t,r}}.
			\]
			To prove (\ref{eq:le W peetre}), we first show that
			\[
			\|\vec{f}\|_{\dot{F}_{p,t,r}^{s,q}(A_{Q})}^{\star}\lesssim\|\vec{f}\|_{\dot{F}_{p,t,r}^{s,q}(A_{Q})}.
			\]
			By Lemma \ref{lem:peetre le sum}, we obtain
			\begin{align*}
			&	\sum_{Q\in\mathcal{D}_{j}}2^{jsq}\Big[\sup_{y\in\mathbb{R}^{n}}\frac{|A_{Q}\varphi_{j}* \vec{f}(y)|}{(1+2^{j}|\cdot-y)^{a}}\chi_{Q}\Big]^{q} \\
				& \lesssim\Bigg[\sum_{k\in\mathbb{Z}^{n}}\sum_{l\in\mathbb{Z}^{n}}(1+|k-l|)^{-A(a-\Delta )}2^{jn}
			\int_{Q_{j,l}}|2^{js}A_{Q_{j,l}}\varphi_{j}* \vec{f}(z)|^{A}\mathrm{d}z\chi_{Q_{j,k}}(\cdot)\Bigg]^{q/A}.
			\end{align*}
			Let $a\in(n/\min(1,q,p )+\Delta,\infty)$. Then we can select $A\in(0,\min(1,p, q))$,
			such that $A(a-\Delta )>n$. Then by Lemmas \ref{lem:hardy} 
			and \ref{lem:Littlewood max}, we have
			\begin{align*}
				& \bigg\|\bigg(\sum_{j=-\infty}^{\infty}\sum_{Q\in\mathcal{D}_{j}}2^{jsq}\sup_{y\in\mathbb{R}^{n}}\frac{|A_{Q}\varphi_{j}*\vec{f}(y)|^{q}}{(1+2^{j}|\cdot-y|)^{aq}}\chi_{Q}\bigg)^{1/q}\bigg\|_{M_{p}^{t,r}}\\
				& \lesssim\bigg\|\bigg\{\sum_{j=-\infty}^{\infty}\Big(\mathcal{M}\Big(\sum_{Q\in\mathcal{D}_{j}}[2^{js}|A_{Q}\varphi_{j}*\vec{f}|\chi_{Q}]^{A}\Big)\Big)^{q/A}\bigg\}^{A/q}\bigg\|_{M_{p/A}^{t/A.r/A}}^{1/A}\\
				& \lesssim\|\vec{f}\|_{\dot{F}_{p,t,r}^{s,q}(A_{Q})}.
			\end{align*}
			From Theorem \ref{thm:equivalence}, we infer that
			\begin{equation}
				\|\vec{f}\|_{\dot{F}_{p,t,r}^{s,q}(A_{Q})}^{\star}\lesssim\|\vec{f}\|_{\dot{F}_{p,t,r}^{s,q}(A_{Q})}\approx\|\vec{f}\|_{\dot{F}_{p,t,r}^{s,q}(W)}.\label{eq:peetre AQ le}
			\end{equation}
			Let
			\begin{equation*}
				\|\vec{f}\|_{\dot{F}_{p,t,r}^{s,q}(\{A_{Q}\})}^{\star\star}:=\bigg\|\bigg(\sum_{j=-\infty}^{\infty}\sum_{Q\in\mathcal{D}_{j}}2^{jsq}\sup_{z\in Q}\sup_{y\in\mathbb{R}^{n}}\frac{|A_{Q}\varphi_{j}*\vec{f}(y)|^{q}}{(1+2^{j}|z-y|)^{aq}}\chi_{Q}\bigg)^{1/q}\bigg\|_{M_{p}^{t,r}}.
			\end{equation*}
			By (\ref{eq:peetre AQ le}), we still need to show that
			\begin{equation}
				\|\vec{f}\|_{\dot{F}_{p,t,r}^{s,q}(W)}^{\star}\lesssim\|\vec{f}\|_{\dot{F}_{p,t,r}^{s,q}(A_{Q})}^{\star\star}\lesssim\|\vec{f}\|_{\dot{F}_{p,t,r}^{s,q}(A_{Q})}^{\star}.\label{eq:star le star star le}
			\end{equation}
			For $j\in\mathbb{Z}$ and $x\in\mathbb{R}^{n}$, let
			\[
			h_{j}(x):=2^{js}\sup_{y\in\mathbb{R}^{n}}\frac{|W^{1/p}(x)\varphi_{j}*\vec{f}(y)|}{(1+2^{j}|x-y|)^{a}},
			\]
			\[
			k_{j}(x):=\sum_{Q\in\mathcal{D}_{j}}|Q|^{-s/n}\sup_{z\in Q}\sup_{y\in\mathbb{R}^{n}}\frac{|A_{Q}\varphi_{j}*\vec{f}(y)|}{(1+2^{j}|z-y|)^{a}}\chi_{Q}(x)
			\]
			and
			\[
			\gamma_{j}(x):=\sum_{Q\in\mathcal{D}_{j}}\|W^{1/p}(x)A_{Q}^{-1}\|\chi_{Q}(x).
			\]
			It is obvious that, for any $j\in\mathbb{N}_{0}$ and $x\in\mathbb{R}^{n}$, $	h_{j}(x) \lesssim \gamma_{j}(x)k_{j}(x).$
			By Corollary \ref{cor:gamma le}, we have
			\begin{align*}
				\|\vec{f}\|_{\dot{F}_{p,t,r}^{s,q}(W)}^{\star} & =\bigg\|\bigg(\sum_{j=-\infty}^{\infty}h_{j}^{q}\bigg)^{1/q}\bigg\|_{M_{p}^{t,r}} \lesssim\|\vec{f}\|_{\dot{F}_{p,t,r}^{s,q}(A_{Q})}^{\star\star},
			\end{align*}
			which is just the first inequality of (\ref{eq:star le star star le}).
			Next, we show the right inequality of (\ref{eq:star le star star le}).
			We use the fact that
			\[
			1+2^{j}|x-y|\approx1+|l-k|\approx1+2^{j}|z-y|
			\]
			for any $x,z\in Q_{j,k}$ and $y\in Q_{j,l}$ to deduce that for any
			$a>0$, $j\in\mathbb{Z}$, $k\in\mathbb{Z}^{n}$, and $x\in Q_{j,k}$,
			\[
			\sup_{z\in Q_{j,k}}\sup_{y\in\mathbb{R}^{n}}\frac{|A_{Q_{j,k}}\varphi_{j}*\vec{f}(y)|}{(1+2^{j}|z-y|)^{a}}\approx\sup_{y\in\mathbb{R}^{n}}\frac{|A_{Q_{j,k}}\varphi_{j}*\vec{f}(y)|}{(1+2^{j}|x-y|)^{a}},
			\]
			which implies that
			\begin{equation*}
				\|\vec{f}\|_{\dot{F}_{p,t,r}^{s,q}(A_{Q})}^{\star\star}\approx\|\vec{f}\|_{\dot{F}_{p,t,r}^{s,q}(A_{Q})}^{\star}.
			\end{equation*}
			Hence the proof is completed.
		\end{proof}
		
		\begin{thm}
			\label{thm:Peetre char inho} 
			Let $0<p<t<r<\infty$ or $0<p\le t<r=\infty$.
			Let $s\in\mathbb{R}$, $W\in  \mathcal  A_{p}$. 
			Let $ q \in (0, p + \delta_W) $ where $\delta_W >0$  is the same as in Lemma \ref{lem:W AQ}. 
			If $a>n/\min(1,q,p )+\Delta$  where $\Delta$  is the same as in Lemma \ref{AQ AR improved}, then
			\begin{equation*}
				\|\vec{f}\|_{{F}_{p,t,r}^{s,q}(W)}\approx\|\vec{f}\|_{{F}_{p,t,r}^{s,q}(W)}^{\star}
			\end{equation*}
			where
			\begin{equation*}
				\|\vec{f}\|_{{F}_{p,t,r}^{s,q}(W)}^{\star}:=\bigg\|\bigg\{\sum_{j \ge 0 }2^{jsq}\bigg[\big(\varphi_{j}^{*}\vec{f}\big)_{a}^{(W,p)}\bigg]^{q}\bigg\}^{1/q}\bigg\|_{M_{p}^{t,r}}.
			\end{equation*}
			
		\end{thm}
		\begin{proof}
			The proof is similar to Theorem \ref{thm:Peetre char} and we omit it here.
		\end{proof}
		
		\subsection{Characterizations by Lusin-area function}
		\begin{thm}
			\label{thm:char lusin homo} Let $0<p<t<r<\infty$ or $0<p\le t<r=\infty$.
			Let $s\in\mathbb{R}$, $W\in  \mathcal  A_{p}$. 
			Let $ q \in (0, p + \delta_W) $ where $\delta_W >0$  is the same as in Lemma \ref{lem:W AQ}. 
			Then
			\begin{equation*}
				\|\vec{f}\|_{\dot{F}_{p,t,r}^{s,q}(W)}\approx\|\vec{f}\|_{\dot{F}_{p,t,r}^{s,q}(W)}^{\bullet},
			\end{equation*}
			where
			\begin{equation*}
				\|\vec{f}\|_{\dot{F}_{p,t,r}^{s,q}(W)}^{\bullet}:=\bigg\|\bigg\{\sum_{j\in\mathbb{Z}}2^{jsq}2^{jn}\int_{B(\cdot,2^{-j})}\Big|W^{1/p}(\cdot)\big(\varphi_{j}*\vec{f}\big)(y)\Big|^{q}\mathrm{d}y\bigg\}^{1/q}\bigg\|_{M_{p}^{t,r}}.
			\end{equation*}
		\end{thm}
		
		\begin{proof}
			From  Theorem \ref{thm:Peetre char}, it is enough to show that
			\[
			\|\vec{f}\|_{\dot{F}_{p,t,r}^{s,q}(W)}^{\bullet}\approx\|\vec{f}\|_{\dot{F}_{p,t,r}^{s,q}(W)}^{\star}.
			\]
			For any $a>0$, we first prove that
			\begin{equation}
				\|\vec{f}\|_{\dot{F}_{p,t,r}^{s,q}(W)}^{\bullet}\lesssim\|\vec{f}\|_{\dot{F}_{p,t,r}^{s,q}(W)}^{\star}.\label{eq:box le star}
			\end{equation}
		From \cite[Proof of Theorem 3.11]{WYZ22}, for any  $a>0$, $j\in\mathbb{Z}$, $x\in\mathbb{R}^{n}$,
		we have
			\begin{equation*}
			2^{jn}\int_{B(x,2^{-j})}\Big|W^{1/p}(x)\big(\varphi_{j}*\vec{f}\big)(y)\Big|^{q}\mathrm{d}y \lesssim\Big(\big(\varphi_{j}^{*}\vec{f}\big)_{a}^{(W,p)}(x)\Big)^{q}
			\end{equation*}
			which implies that (\ref{eq:box le star}) holds true. Next, we show
			that
			\[
			\|\vec{f}\|_{\dot{F}_{p,t,r}^{s,q}(W)}^{\star}\lesssim\|\vec{f}\|_{\dot{F}_{p,t,r}^{s,q}(W)}^{\bullet}
			\]
			if $a$ is sufficiently large. From (\ref{eq:peetre AQ le}), it suffices
			to prove that
			\[
			\|\vec{f}\|_{\dot{F}_{p,t,r}^{s,q}(A_{Q})}^{\star}\lesssim\|\vec{f}\|_{\dot{F}_{p,t,r}^{s,q}(W)}^{\bullet}.
			\]

			Case $p\in(0,1]$. Let
			\[
			g_{j}(z):=\sum_{Q\in\mathcal{D}_{j}}2^{js}2^{jn}\int_{B(0,2^{-j})}|A_{Q}(\varphi_{j}*\vec{f})(v+z)|^{A}\mathrm{d}v\chi_{Q}(z).
			\]
			From \cite[(3.26)]{WYZ22}, if $A(a-\Delta )>2n$, for any $x\in\mathbb{R}^{n}$,
			
			\begin{align}
				\sum_{Q\in\mathcal{D}_{j}}\left[2^{js} \sup_{y\in\mathbb{R}^{n}}\frac{|A_{Q}(\varphi_{j}*\vec{f})(y)| }{(1+2^{j}|x-y|)^{a }}\chi_{Q}(x) \right] ^q 
				& \lesssim\mathcal{M}(g_{j}) ^{q/A}(x).\label{eq:le HL}
			\end{align}
			By (\ref{eq:le HL}), Lemma \ref{lem:hardy}, the H\"{o}lder inequality,
			and Lemma \ref{lem:W AQ}, we conclude 
			\begin{align*}
				& \|\vec{f}\|_{\dot{F}_{p,t,r}^{s,q}(W)}^{\star} \\
				& \lesssim\bigg\|\bigg\{\sum_{j\in\mathbb{Z}}\big(\mathcal{M}(g_{j})\big)^{q/A}\bigg\}^{A/q}\bigg\|_{M_{p/A}^{t/A,r/A}}^{1/A}\\
				& \lesssim\bigg\|\bigg\{\sum_{j\in\mathbb{Z}}\sum_{Q\in\mathcal{D}_{j}}2^{jsq}\Big(2^{jn}\int_{B(0,2^{-j})}|A_{Q}(\varphi_{j}*\vec{f})(\cdot+z)|^{A}\mathrm{d}z\Big)^{q/A}\chi_{Q}\bigg\}^{1/q}\bigg\|_{M_{p}^{t,r}} \\
	 &\lesssim\bigg\|\bigg\{\sum_{j\in\mathbb{Z}}\sum_{Q\in\mathcal{D}_{j}}2^{jsq}2^{jn}\int_{B(0,2^{-j})}|A_{Q}(\varphi_{j}*\vec{f})(\cdot+z)|^{q}\mathrm{d}z\chi_{Q}\bigg\}^{1/q}\bigg\|_{M_{p}^{t,r}}\\
				& \lesssim\|\vec{f}\|_{\dot{F}_{p,t,r}^{s,q}(W)}^{\bullet}.
			\end{align*}
			
			Case $p\in(1,\infty)$.
			The following estimate comes from \cite[(3.28)]{WYZ22}. Let $a>0$
			such that $A(a-\Delta)-(p'-A)n/p'>n$. Let $A\in(0,1)$ sufficiently
			small such that $(p'-A)q/(Ap')>1$, and $q/A>1$. Then we have
			\begin{align}
				& \sum_{k\in\mathbb{Z}^{n}}\sup_{y\in\mathbb{R}^{n}}\frac{|A_{Q_{j,k}}\varphi_{j}* \vec f (y)|^{q}}{(1+2^{j}|x-y)^{q}}\chi_{Q_{j, k}}(x)\nonumber \\
				& \lesssim\bigg\{\mathcal{M}\bigg(\Big[2^{jn}\int_{B(0,2^{-j})}|W^{1/p}(\cdot)(\varphi_{j}*\vec{f})(\cdot+z)|^{A}\mathrm{d}z\bigg]^{\frac{p'}{p'-A}}\bigg)(x)\bigg\}^{\frac{(p'-A)q}{Ap'}}.\label{eq:3.28}
			\end{align}
			By (\ref{eq:3.28}), Lemma \ref{lem:hardy} and the H\"{o}lder inequality,
			we obtain
			\begin{align*}
				& \|\vec{f}\|_{\dot{F}_{p,t,r}^{s,q}(A_{Q})}^{\star}\\
				& \lesssim\bigg\|\bigg\{\sum_{j\in\mathbb{Z}}2^{jsq}\bigg[2^{jn}\int_{B(0,2^{-j})}|W^{1/p}(\cdot)(\varphi_{j}*\vec{f})(\cdot+z)|^{A}\mathrm{d}z\bigg]^{q/A}\bigg\}^{1/q}\bigg\|_{M_{p}^{t,r}}\\
				& \lesssim\bigg\|\bigg\{\sum_{j\in\mathbb{Z}}2^{jsq}2^{jn}\int_{B(0,2^{-j})}|W^{1/p}(\cdot)(\varphi_{j}*\vec{f})(\cdot+z)|^{q}\mathrm{d}z\bigg\}^{1/q}\bigg\|_{M_{p}^{t,r}}\\
				& \approx\|\vec{f}\|_{\dot{F}_{p,t,r}^{s,q}(W)}^{\bullet}.\qedhere
			\end{align*}
		\end{proof}
		Theorem \ref{thm:char lusin inhomo} has an analogue of  the inhomogeneous
		spaces. Since the proof is similar, we omit it here.
		\begin{thm}
			\label{thm:char lusin inhomo} 
			Let $0<p<t<r<\infty$ or $0<p\le t<r=\infty$.
			Let $s\in\mathbb{R}$, $W\in  \mathcal  A_{p}$. 
			Let $ q \in (0, p + \delta_W) $ where $\delta_W >0$  is the same as in Lemma \ref{lem:W AQ}. 
			Then
			\begin{equation*}
				\|\vec{f}\|_{ F_{p,t,r}^{s,q}(W)}\approx\|\vec{f}\|_{F_{p,t,r}^{s,q}(W)}^{\bullet},
			\end{equation*}
			where
			\begin{equation*}
				\|\vec{f}\|_{\dot{F}_{p,t,r}^{s,q}(W)}^{\bullet}:=\bigg\|\bigg\{\sum_{j= 0}^\infty 2^{jsq}2^{jn}\int_{B(\cdot,2^{-j})}\Big|W^{1/p}(\cdot)\big(\varphi_{j}*\vec{f}\big)(y)\Big|^{q}\mathrm{d}y\bigg\}^{1/q}\bigg\|_{M_{p}^{t,r}}.
			\end{equation*}
		\end{thm}
		
		\subsection{Characterizations by Littlewood-Paley $g_{\lambda}^{*}$-function}
		
		\begin{thm}
			\label{thm:char g lambda star homo} Let $0<p<t<r<\infty$ or $0<p\le t<r=\infty$.
			Let $s\in\mathbb{R}$, $W\in  \mathcal  A_{p}$.
			Let $ q \in (0, p + \delta_W) $ where $\delta_W >0$  is the same as in Lemma \ref{lem:W AQ}. 
			Let $\lambda\in(1/\min(1, p,q)+ \Delta / n ,\infty)$ where $\Delta$ is the same as in Lemma \ref{AQ AR improved}.
			Then
			\[
			\|\vec{f}\|_{\dot{F}_{p,t,r}^{s,q}(W)}\approx\|\vec{f}\|_{\dot{F}_{p,t,r}^{s,q}(W)}^{\clubsuit},
			\]
			where
			\[
			\|\vec{f}\|_{\dot{F}_{p,t,r}^{s,q}(W)}^{\clubsuit}:=\bigg\|\bigg\{\sum_{j\in\mathbb{Z}}2^{jsq}2^{jn}\int_{\mathbb{R}^{n}}\frac{|W^{1/p}(\cdot)(\varphi_{j}*\vec{f})(y)|^{q}}{(1+2^{j}|\cdot-y|)^{\lambda nq}}\mathrm{d}y\bigg\}^{1/q}\bigg\|_{M_{p}^{t,r}}.
			\]
		\end{thm}
		
		\begin{proof}
			For any $x\in\mathbb{R}^{n},$ $y\in B(x,2^{-j})$, $1+2^{j}|x-y|\approx1$.
			Then we conclude
			\begin{equation*}
				\int_{B(x,2^{-j})}\Big|W^{1/p}(x)\big(\varphi_{j}*\vec{f}\big)(y)\Big|^{q}\mathrm{d}y  \lesssim \int_{\mathbb{R}^{n}}\frac{\Big|W^{1/p}(x)\big(\varphi_{j}*\vec{f}\big)(y)\Big|^{q}}{(1+2^{j}|x-y|)^{\lambda nq}}\mathrm{d}y,
			\end{equation*}
			which implies that $\|\vec{f}\|_{\dot{F}_{p,t,r}^{s,q}(W)}\approx\|\vec{f}\|_{\dot{F}_{p,t,r}^{s,q}(W)}^{\bullet}\lesssim\|\vec{f}\|_{\dot{F}_{p,t,r}^{s,q}(W)}^{\clubsuit}$. 			
			Let $\{A_{Q}\}_{Q\in\mathcal{D}}$ be a sequence of reducing operators of order $p$ for
			$W$. Define
			\[
			\|\vec{f}\|_{\dot{F}_{p,t,r}^{s,q}(A_{Q})}^{\clubsuit}:=\bigg\|\bigg\{\sum_{j\in\mathbb{Z}}\sum_{Q\in\mathcal{D}_{j}}2^{jsq}2^{jn}\sup_{z\in Q}\int_{\mathbb{R}^{n}}\frac{|A_{Q}(\varphi_{j}*\vec{f})(y)|^{q}}{(1+2^{j}|z-y|)^{\lambda nq}}\mathrm{d}y\bigg\}^{1/q}\chi_{Q}\bigg\|_{M_{p}^{t,r}}.
			\]
			We first show that
			\[
			\|\vec{f}\|_{\dot{F}_{p,t,r}^{s,q}(W)}^{\clubsuit}\lesssim\|\vec{f}\|_{\dot{F}_{p,t,r}^{s,q}(A_{Q})}^{\clubsuit}.
			\]
			Let
			\[
			\gamma_{j}(x):=\sum_{Q\in\mathcal{D}_{j}}\|W^{1/p}(x)A_{Q}^{-1}\|\chi_{Q}(x),
			\]
			\[
			h_{j}(x):=2^{js}2^{jn/q}\bigg[\int_{\mathbb{R}^{n}}\frac{|W^{1/p}(x)(\varphi_{j}*\vec{f})(y)|^{q}}{(1+2^{j}|x-y|)^{\lambda nq}}\mathrm{d}y\bigg]^{1/q},
			\]
			and
			\[
			f_{j}(x):=\sum_{Q\in\mathcal{D}_{j}}2^{js}2^{jn/q}\bigg[\sup_{z\in Q}\int_{\mathbb{R}^{n}}\frac{|A_{Q}(\varphi_{j}*\vec{f})(y)|^{q}}{(1+2^{j}|z-y|)^{\lambda nq}}\mathrm{d}y\bigg]^{1/q}\chi_{Q}(x).
			\]
			Then $h_{j}\le\gamma_{j}f_{j}$. Note that $f_{j}$ is a constant
			on each cube $Q\in\mathcal{D}_{j}$. By Corollary \ref{cor:gamma le},
			we have
			\begin{align*}
				\|\vec{f}\|_{\dot{F}_{p,t,r}^{s,q}(W)}^{\clubsuit} & =\|\{h_{j}\}\|_{M_{p}^{t,r}(\ell^{q})}\le\|\{\gamma_{j}f_{j}\}\|_{M_{p}^{t,r}(\ell^{q})}  \\
				& \lesssim\|\{f_{j}\}\|_{M_{p}^{t,r}(\ell^{q})}=\|\vec{f}\|_{\dot{F}_{p,t,r}^{s,q}(A_{Q})}^{\clubsuit}.
			\end{align*}
			Next we prove
			\begin{equation*}
				\|\vec{f}\|_{\dot{F}_{p,t,r}^{s,q}(A_{Q})}^{\clubsuit} \lesssim  \|\vec{f}\|_{\dot{F}_{p,t,r}^{s,q}(W)}.
			\end{equation*}
		From \cite[Proof of Theorem 3.14.]{WYZ22},  for $A \in (0, \min\{p,q\})$,  we have
		\begin{align*}
		& \sum_{j\in\mathbb{Z}} 2^{jsq}\sum_{Q\in\mathcal{D}_{j}}2^{jn}\sup_{z\in Q}\int_{\mathbb{R}^{n}}\frac{|A_{Q}(\varphi_{j}*\vec{f})(y)|^{q}}{(1+2^{j}|z-y|)^{\lambda nq}}\mathrm{d}y\chi_{Q}(x) \\
		& \lesssim  \sum_{j\in\mathbb{Z}} \bigg[\mathcal{M}\bigg(\sum_{Q\in\mathcal{D}_{j}}| 2^{js} A_{Q}(\varphi_{j}*\vec{f})|^{A}\chi_{Q}\bigg)(x)\bigg]^{q/A}.
		\end{align*}
			Then by Lemma \ref{lem:hardy} and Theorem \ref{thm:equivalence}, we have
			\begin{align*}
				\|\vec{f}\|_{\dot{F}_{p,t,r}^{s,q}(A_{Q})}^{\clubsuit} & \lesssim\bigg\|\bigg\{\sum_{j\in\mathbb{Z}} \bigg[\mathcal{M}\bigg(\sum_{Q\in\mathcal{D}_{j}}|2^{js} A_{Q_{j, k}}(\varphi_{j}*\vec{f})|^{A}\chi_{Q}\bigg) \bigg]^{q/A}\Bigg\}^{1/q}\bigg\|_{M_{p}^{t,r}}\\
				& \lesssim\|\vec{f}\|_{\dot{F}_{p,t,r}^{s,q}(A_{Q})}\approx\|\vec{f}\|_{\dot{F}_{p,t,r}^{s,q}(W)}.\qedhere
			\end{align*}
		\end{proof}
		
		Theorem \ref{thm:char g lambda star homo} has an analogue of the
		inhomogeneous spaces. The proof is similar, so we omit it here.
		
		\begin{thm}
			\label{thm:char g lambda star inhomo} 
			Let $0<p<t<r<\infty$ or $0<p\le t<r=\infty$.
			Let $s\in\mathbb{R}$, $W\in  \mathcal  A_{p}$.
			Let $ q \in (0, p + \delta_W) $ where $\delta_W >0$  is the same as in Lemma \ref{lem:W AQ}. 	
			Let $\lambda\in(1/\min(1, p,q)+ \Delta / n ,\infty)$ where $\Delta$ is the same as in Lemma \ref{AQ AR improved}.
			Then
			\[
			\|\vec{f}\|_{F_{p,t,r}^{s,q}(W)}\approx\|\vec{f}\|_{F_{p,t,r}^{s,q}(W)}^{\clubsuit},
			\]
			where
			\[
			\|\vec{f}\|_{\dot{F}_{p,t,r}^{s,q}(W)}^{\clubsuit}:=\bigg\|\bigg\{\sum_{j=0}^\infty 2^{jsq}2^{jn}\int_{\mathbb{R}^{n}}\frac{|W^{1/p}(\cdot)(\varphi_{j}*\vec{f})(y)|^{q}}{(1+2^{j}|\cdot-y|)^{\lambda nq}}\mathrm{d}y\bigg\}^{1/q}\bigg\|_{M_{p}^{t,r}}.
			\]
		\end{thm}

		\subsection{Characterizations by wavelet}
		In this subsection, we characterize the matrix weighted Bourgain-Morrey Triebel-Lizorkin
		spaces via the Meyer wavelets and the Daubechies wavelets.		
		We first recall the Meyer and Lemari\'{e} construction of a wavelet basis with the generating
		function $\theta \in \mathscr S (\rn) $ (see  \cite{LM87} and \cite{Mer87}):
		\begin{lem} \label{Meyer}
			There exist real-valued functions $\theta^{(i)}  \in \mathscr{S}  (\rn)$,  $i = 1,\ldots, 2^{n} -1$
			such that the collection  $\{ \theta^{(i)} _{\nu,k} \} =\{ 2^{\nu n/2} \theta^{(i)}(2^\nu x - k) \}$ is an orthonormal basis for
			$L^2 (\rn)$. The functions $\theta^{(i)} $ satisfy
			\begin{equation*}
				\operatorname{supp}  \mathcal F \theta^{(i)} \subset  \left\{  [ -8\pi /3, 8\pi /3] ^n  \backslash [ -2\pi /3, 2\pi /3] ^n    \right\}
			\end{equation*}
			and, hence,
			\begin{equation*}
				\int_{\mathbb R ^n } x^\gamma \theta^{(i)} (x) \d x =0 
			\end{equation*}
			for all multi-indices $\gamma $.
		\end{lem}
		Thus, we have 
		\begin{equation*}
			f = \sum_{i=1}^{2^n -1} \sum_{Q \in \D}  \langle f , \theta_{Q}^{(i)}  \rangle \theta_{Q}^{(i)} 
		\end{equation*}
		for all $f  \in L^2(\rn)$.
		This identity
		extends to all $f \in \mathscr S ' / \mathcal P (\rn)$ (see, for instance,  \cite[proof of Theorem 7.20]{FJW91}).
		Here 
		\begin{equation*}
			\theta_{Q}^{(i)} =|Q|^{-1/2}\theta ^{(i)} ((x-x_{Q})/\ell(Q)).
		\end{equation*}		
		For any $m \in \mathbb N $, we use $C^{m}$ to denote the set of all $m$ times continuously differentiable functions on $\rn$.
		
		\begin{defn}\label{Def: Db wavelet}
			(Daubechies wavelets)
			Let $\mathcal N \in \mathbb N$. The functions $\{ \theta^{(i)}\}_{i=1}^{2^n -1} $ are Daubechies wavelets of class
			$C ^{\mathcal N  }$ if each $  \theta^{(i)}  \in C ^{\mathcal N  } $ is real-valued with bounded support, and if
			\begin{equation*}
				\left\{  \theta_{Q}^{(i)} : i \in \{1, \ldots, 2^n - 1  \} , Q \in \D             \right\}
			\end{equation*}
			is an orthonormal basis of $L^2 (\rn)$.
		\end{defn}

		The following wavelet basis was constructed by Daubechies (see, for instance,  \cite{Db88}).
		\begin{lem}
			For any  $\mathcal N \in \mathbb N$, there exist Daubechies wavelets of class  $C^\mathcal N $.
		\end{lem}
		
		\begin{lem}[Corollary 5.5.2, \cite{Db92}]  \label{Db wavele cancel}	
			If the $\psi_{j,k} (x) = 2^{-j/2} \psi (2^{-j} x -k)  $ constitute an orthonormal set in $L^2 (\mathbb R)$, with $|\psi (x) | \le C (1+|x|) ^{ -m -1 -\epsilon}  $, $\psi \in C^m (\mathbb R)$  and $  \partial ^{\ell}  \psi $ bounded for $ \ell \le m$, then $\int_\mathbb R x^\ell \psi (x) \d x =0   $ for $\ell = 0,1,\ldots,m$.
			
		\end{lem}

		Then we have the following characterizations of the matrix weighted Bourgain-Morrey Triebel-Lizorkin
		spaces via the Meyer wavelets and Daubechies wavelets.
		Examples of wavelets with the properties in Theorem \ref{thm:wavelet homo} are Meyer wavelets  and Daubechies wavelets  of class  $C^\mathcal N $ for sufficiently large $\mathcal N$.
		\begin{thm}\label{thm:wavelet homo}
			Let $0<p<t<r<\infty$ or $0<p\le t<r=\infty$. 
			Let $W \in \mathcal A_p$ and  
			$\{A_{Q}\}_{Q\in\mathcal{D}}$ be a family of reducing operators of order $p$ for $W$. 
			Let $ q \in (0, p + \delta_W) $ where $\delta_W >0$  is the same as in Lemma \ref{lem:W AQ}. 
			Suppose that for some sufficiently large positive integers
			$N_{0},R,S$ (depending on $p,q,s,n,W$), the generators $\{\psi^{(i)}\}_{i=1} ^{2^{n}-1 }$
			of a wavelet basis satisfy
			\[
			\int_{\mathbb{R}^n} x^{\gamma}\psi^{(i)}{\rm d}x=0,\;|\gamma|\le N_{0},
			\]
			and $|\partial^{\gamma}\psi^{(i)}(x)|\le c(1+|x|)^{-R}$ for all $|\gamma|\le S$.
			Then
			\begin{equation}
				\|\vec{f}\|_{\dot{F}_{p,t,r}^{s,q}(A_{Q})}\approx\sum_{i=1}^{2^{n}-1}\|\{\langle\vec{f},\psi_{Q}^{(i)}\rangle\}_{Q\in\mathcal{D}}\|_{\dot{f}_{p,t,r}^{s,q}(A_{Q})},\label{eq:wavelet Aq}
			\end{equation}
			and
			\begin{equation}
				\|\vec{f}\|_{\dot{F}_{p,t,r}^{s,q}(W)}\approx \sum_{i=1}^{2^{n}-1}\|\{\langle\vec{f},\psi_{Q}^{(i)}\rangle\}_{Q\in\mathcal{D}}\|_{\dot{f}_{p,t,r}^{s,q}(W)}. \label{eq:wavelet W}
			\end{equation}

		\end{thm}
		
		\begin{rem}\label{Db wavelet satisfy condition}
			By the definition of quasi-norms of $\dot{F}_{p,t,r}^{s,q}(W)$,  $\dot{F}_{p,t,r}^{s,q}(W)$ embed into matrix-weighted  Triebel-Lizorkin-type 	spaces  introduced in \cite{BHYY1}. From \cite[Corollary 4.8]{BHYY3}, $\langle\vec{f},\psi_{Q}^{(i)}\rangle$ is well-defined for $f \in \dot{F}_{p,t,r}^{s,q}(W)$ and Daubechies wavelets $\{\psi^{(i)}\}_{i=1} ^{2^{n}-1 }$.
			
			Let $s\in\mathbb{R}$, $0<p<\infty$, $W\in \mathcal  A_p$. 
			Let $d, \tilde d, \Delta$  be the same as in Lemma \ref{AQ AR improved}.	
			Let $\mathcal{N} >\max \{s  + d/p ,  -s+ n/\min(1,q,p) +\tilde d / p'- 1- n \}  $. Let $\{ {\theta^{(i)}} \}_{i=1}^{2^n-1} $ be Daubechies wavelets of class $C^\mathcal{N}$. By Lemma \ref{Db wavele cancel}, for every $i=1,\ldots,2^n-1$, each multi-index $\gamma$ with $|\gamma| \le \mathcal{N}$, $\int_{\mathbb{R}^n} x^\gamma \theta^{(i)}_Q {\rm d}x=0 $.
			Since $\theta^{(i)} $ has bounded support, the decay conditions are automatic.	
			Hence  $\{ {\theta^{(i)}} \} _{i=1}^{2^n-1} $ satisfy the conditions in Theorem \ref{thm:wavelet homo}.
		\end{rem}
		
		\begin{proof}[\bf Proof of Theorem \ref{thm:wavelet homo}:]			
			Let $W \in \mathcal A_p$ and  
			$\{A_{Q}\}_{Q\in\mathcal{D}}$ be a family of reducing operators of order $p$ for $W$. Let $d, \tilde d, \Delta$  be the same as in Lemma \ref{AQ AR improved}.
			As in Theorem \ref{thm:varphi transform}, let $\delta\in(0,1]$, $ N > -s+ n/\min(1,q,p) + \tilde d / p'  - 1- n$,
			$K +\delta > s  + d/p $, and $ M >  n / \min(1,q,p)+  \Delta$. 
			
			We first  prove (\ref{eq:wavelet Aq}).			
			Set $N_{0}\ge N$, $S\ge K+\delta$, $R>\max(M,N+1+n)$. Thus $\{\psi_{Q}^{(i)}\}_{Q\in\mathcal{D},i=1,\dots,2^{n}-1}$
			is a family of smooth $(N,K,M,\delta)$ molecules for each $i$. Hence
			\[
			\|\vec{f}\|_{\dot{F}_{p,t,r}^{s,q}(A_{Q})}\le c\sum_{i=1}^{2^{n}-1}\|\{\langle\vec{f},\psi_{Q}^{(i)}\rangle\}_{Q\in\mathcal{D}}\|_{\dot{f}_{p,t,r}^{s,q}(A_{Q})}
			\]
			follows from Theorem \ref{thm:varphi transform} and the wavelet expansion
			\[
			\vec{f}= \sum_{i=1}^{2^n -1}\sum_{Q\in \D}\langle\vec{f},\psi_{Q}^{(i)}\rangle\psi_{Q}^{(i)}.
			\]
			
			Next we show the  converse estimate. 
			For each $i$, define $\vec{s}^{(i)}=\{\vec{s}_{Q}^{(i)}\}_{Q\in\mathcal{D}}$
			by $\vec{s}_{Q}^{(i)}=\langle\vec{f},\psi_{Q}^{(i)}\rangle$. Using
			(\ref{eq:converge}), we have
			\[
			\vec{s}_{Q}^{(i)}=\bigg\langle\vec{f},\sum_{P\in\mathcal{D}}\langle\psi_{Q}^{(i)},\varphi_{P}\rangle\psi_{P}\bigg\rangle=\sum_{P\in\mathcal{D}}b_{QP}^{(i)}\vec{s}_{P}=(B^{(i)}\vec{s})_{Q},
			\]
			where $B^{(i)}=\{b_{QP}^{(i)}\}$ is defined by $b_{QP}^{(i)}=\langle\psi_{Q}^{(i)},\varphi_{P}\rangle$
			and $\vec{s}=\{\vec{s}_{Q}\}$ is defined by $\vec{s}_{Q}=\langle\vec{f},\psi_{Q}\rangle$.
			Since $\psi\in\mathcal{A}$, Theorem \ref{thm:equivalence} gives
			$\|\vec{s}\|_{\dot{f}_{p,t,r}^{s,q}(A_{Q})}\approx\|\vec{f}\|_{\dot{F}_{p,t,r}^{s,q}(A_{Q})}.$
			We claim that for $N_{0},R,S$ sufficiently large, $B^{(i)}\in\mathbf{ad}_{p}^{s,q}(d, \tilde d, \Delta )$,
			and thus by Theorem \ref{thm:almost dia bounded on discrete}, we
			have
			\[
			\|\vec{s}^{(i)}\|_{\dot{f}_{p,t,r}^{s,q}(A_{Q})}=\|B^{(i)}\vec{s}\|_{\dot{f}_{p,t,r}^{s,q}(A_{Q})}\le c\|\vec{s}\|_{\dot{f}_{p,t,r}^{s,q}(A_{Q})}\le c\|\vec{f}\|_{\dot{F}_{p,t,r}^{s,q}(A_{Q})}.
			\]
		
			To show $B^{(i)}\in\mathbf{ad}_{p}^{s,q}(d, \tilde d, \Delta ) $, for $P\in\mathcal{D}_{j}$,
			we have $b_{QP}^{(i)}=2^{-jn/2}\tilde{\varphi}_{j}*\psi_{Q}^{(i)}(x_{P}).$
			Applying Lemma \ref{lem:mQ}, we see that $B^{(i)}\in\mathbf{ad}_{p}^{s,q}(d, \tilde d, \Delta ) $
			if $\{\psi_{Q}^{(i)}\}_{Q\in\mathcal{D}}$ is a family of smooth $(N,K,M,\delta$)-molecules
			for
			\[
			N>N_{1}= -s+ n/\min(1,q,p) + \tilde d / p'  - 1- n,
			\]
			\[
			K+\delta>S_{1}= s  + d/p ,
			\]
			\[
			M>M_{1}= \frac{n}{\min(1,q,p)}+  \Delta,
			\]
			which in turn holds if $\psi^{(i)}$ satisfies the condition in the
			theorem for $N_{0}>N_{1}$, $S>S_1$, and $R>R_{1}=\max(M_{1},N_{1}+1+n)$.
			
			Next, we prove (\ref{eq:wavelet W}). By Theorem \ref{thm:equivalence}, it suffices to show
			\begin{equation*}
				\|\{\langle\vec{f},\psi_{Q}^{(i)}\rangle\}_{Q\in\mathcal{D}}\|_{\dot{f}_{p,t,r}^{s,q}(A_{Q})} \approx \|\{\langle\vec{f},\psi_{Q}^{(i)}\rangle\}_{Q\in\mathcal{D}}\|_{\dot{f}_{p,t,r}^{s,q}(W)}.
			\end{equation*}
			By Theorem \ref{thm:A_Q le W discrete}, we have
			\begin{equation*}
				\|\{\langle\vec{f},\psi_{Q}^{(i)}\rangle\}_{Q\in\mathcal{D}}\|_{\dot{f}_{p,t,r}^{s,q}(A_{Q})} \lesssim \|\{\langle\vec{f},\psi_{Q}^{(i)}\rangle\}_{Q\in\mathcal{D}}\|_{\dot{f}_{p,t,r}^{s,q}(W)}.
			\end{equation*}
			By Corollary \ref{cor:W le AQ}, we have
			\begin{equation*}
				\|\{\langle\vec{f},\psi_{Q}^{(i)}\rangle\}_{Q\in\mathcal{D}}\|_{\dot{f}_{p,t,r}^{s,q}(W)}	\lesssim \|\{\langle\vec{f},\psi_{Q}^{(i)}\rangle\}_{Q\in\mathcal{D}}\|_{\dot{f}_{p,t,r}^{s,q}(A_{Q})}  .
			\end{equation*}
			Hence, we complete the proof.
		\end{proof}
		
		\begin{rem}
			We make a few remarks about 	completeness of the spaces by using the idea from \cite[page 513]{FraRou19}.
			
			Let $p,t,r, s,q $ be the same as in Theorem \ref{thm:equivalence}.	
			Suppose that each $A_{Q}$ is an invertible $m \times m$ matrix, then
			$\dot{f}_{p,t,r}^{s,q}(A_{Q})$ is complete. Indeed, let $\vec{s}^{(j)}=\{\vec{s}_{Q}^{(j)}\}_{Q\in\mathcal{D}}$
			be a Cauchy sequence in $\dot{f}_{p,t,r}^{s,q}(A_{Q})$, then $\{A_{Q}\vec{s}_{Q}^{(j)}\}_{Q\in\mathcal{D}}$
			is a Cauchy sequence in $\mathbb{C}^{m}$. Since each $A_{Q}$ is
			invertible, $\vec{s}_{Q}^{(j)}$ is a Cauchy sequence in $\mathbb{C}^{m}$.
			Since $\mathbb{C}^{m}$ is complete, there exists some $\vec{s}_{Q}\in\mathbb{C}^{m}$
			such that $\vec{s}_{Q}^{(j)}\to\vec{s}_{Q}$ as $j\to\infty$. Let
			$\vec{s}=\{\vec{s}_{Q}\}_{Q\in\mathcal{D}}$. By Fatou's lemma, we
			have $\vec{s}\in\dot{f}_{p,t,r}^{s,q}(A_{Q})$ and $\vec{s}^{(j)}$
			converges to $\vec{s}$ in $\dot{f}_{p,t,r}^{s,q}(A_{Q})$.
			Let $W\in \mathcal A_{p}$. The equivalence in Theorem \ref{thm:equivalence}
			shows that  $\dot{f}_{p,t,r}^{s,q}(W)$ is complete.
			It follows then from Theorem \ref{thm:wavelet homo} that for $W \in \mathcal A_p$,  $\dot{F}_{p,t,r}^{s,q}(W)$ is complete. Indeed, a Cauchy sequence $\vec f_j$  in $\dot{F}_{p,t,r}^{s,q}(W)$ has wavelet coefficients, which are Cauchy, and thus, converge in $\dot{f}_{p,t,r}^{s,q}(W)$. If we let $\vec f$ have the wavelet coefficients of the limit sequence, then $\vec f \in \dot{F}_{p,t,r}^{s,q}(W)$   and $\vec f_j  $ converges to $\vec f$ in $\dot{F}_{p,t,r}^{s,q}(W)$.
			Finally, if $\{ A_Q\}$ is  a sequence of reducing operators
			of order $p$ for $W \in \mathcal A_p$, then Theorem \ref{thm:equivalence} implies that $\dot{F}_{p,t,r}^{s,q}(A_{Q}) $  is complete.
			
		\end{rem}
		
		Similar to Theorem \ref{thm:wavelet homo}, we have the following result.

		\begin{thm}
			\label{thm:wavelte inhomogeneous} 	Let $0<p<t<r<\infty$ or $0<p\le t<r=\infty$. 
			Let $W \in \mathcal A_p$ and  
			$\{A_{Q}\}_{Q\in\mathcal{D}}$ be a family of reducing operators of order $p$ for $W$.
			Let $ q \in (0, p + \delta_W) $ where $\delta_W >0$  is the same as in Lemma \ref{lem:W AQ}. 
			Suppose that for some sufficiently large positive integers
			$N_{0},R,S$ (depending on $p,q,s,n,W$), the generators $\{\psi^{(i)}\}_{i=1} ^{ 2^{n}-1}$
			of a wavelet basis satisfy
			\[
			\int_{\mathbb{R}^n} x^{\gamma}\psi^{(i)}{\rm d}x=0,\;|\gamma|\le N_{0},
			\]
			and $|\partial^{\gamma}\psi^{(i)}(x)|\le c(1+|x|)^{-R}$ for all $|\gamma|\le S$.
			Then
			\begin{equation*}
				\|\vec{f}\|_{F_{p,t,r}^{s,q}(A_{Q})}\approx\sum_{i=1}^{2^{n}-1}\|\{\langle\vec{f},\psi_{Q}^{(i)}\rangle\}_{Q\in\mathcal{D}_{+}}\|_{f_{p,t,r}^{s,q}(A_{Q})},
			\end{equation*}
			and
			\begin{equation*}
				\|\vec{f}\|_{F_{p,t,r}^{s,q}(W)}\approx\sum_{i=1}^{2^{n}-1}\|\{\langle\vec{f},\psi_{Q}^{(i)}\rangle\}_{Q\in\mathcal{D}_{+}}\|_{f_{p,t,r}^{s,q}(W)}.
			\end{equation*}
		\end{thm}
		
		\begin{rem}
			The characterization  of matrix weighted Besov spaces ($ 1< p < \infty $) can be seen in \cite[Theorem 10.2, Corollary 10.3]{R03},
			matrix weighted Besov spaces ($0< p \le 1$) in \cite[Theorem 4.4]{FrazierR04}, 
			matrix weighted Triebel-Lizorkin spaces in \cite[Theorem 1.2]{FraRou19},
			matrix weighted Besov type and Triebel-Lizorkin type spaces in  \cite[Theorem 4.10]{BHYY3}	
			via the Meyer wavelets and Daubechies wavelets.
		\end{rem}

		\subsection{Characterizations by atom}
			\begin{defn}
				Let $b,L,N\in(0,\infty)$.
				A function $a_Q $ is called a $(b,L,N)$-atom on a cube $Q$ if supp $a_Q \subset bQ$ for some $b>0$,  $\int_{\mathbb{R}^n} x^{\gamma}a_{Q}{\rm d}x=0$
				for all $|\gamma|\le L$, and $|\partial^{\gamma}a_{Q}(x)|\le c_{\gamma}\ell (Q)^{-|\gamma|-n/2}$, for all $|\gamma| \le N$.
			\end{defn}
	
			\begin{thm}\label{atomic}
				Let $0<p<t<r<\infty$ or $0<p\le t<r=\infty$. 
				Let $W \in \mathcal A_p$ and let $d, \tilde d, \Delta$  be the same as in Lemma \ref{AQ AR improved}.
				Let $\delta_W >0$  be the same as in Lemma \ref{lem:W AQ}. Let $ q \in (0, p + \delta_W)  $ and   $s\in\mathbb{R}$.
				Let $L,N\in \mathbb{N}_{0}$ with $L>-s+ n/\min(1,q,p) +\tilde d / p' - 1- n$, $N>s  +d/p $. If $\vec{f} \in {\dot{F}_{p,t,r}^{s,q}(W)} $, then there exists a sequence $\vec{t} \in {\dot{f}_{p,t,r}^{s,q}(W)} $ and smooth $(b,L,N)$-atoms $\{a_Q\}_{ Q\in\mathcal{D} }$ such that $\vec {f} = \sum_Q \vec{t}_Q a_Q$ and
				\begin{equation*}
					\| \vec{t} \|_{\dot{f}_{p,t,r}^{s,q}(W)} \le C \| \vec{f} \|_{\dot{F}_{p,t,r}^{s,q}(W)}.
				\end{equation*}
				Conversely, if $\{a_Q \}_{ Q\in\mathcal{D} }$ is a family of $(b,L,N)$-atoms. Then for each $\vec{t}:=\{ \vec{t}_Q\}_{Q\in\mathcal{D}} \in {\dot{f}_{p,t,r}^{s,q}(W)} $, there exists an $\vec{f}\in {\dot{F}_{p,t,r}^{s,q}(W)} $, such that $\vec{f} = \sum_{Q\in\mathcal{D}} \vec{t}_Q a_Q$ in $\mathscr{S}'/\mathcal{P} (\rn)$. Moreover, there exists a constant $C>0$, independent of both $\vec {t}$ and $\{a_Q\}_{ Q\in\mathcal{D} }$, such that
				\begin{equation*}
					\| \vec{f} \|_{\dot{F}_{p,t,r}^{s,q}(W)} \le C \| \vec{t} \|_{\dot{f}_{p,t,r}^{s,q}(W)} .
				\end{equation*}
			\end{thm}
			\begin{proof}
				The idea of the proof comes from \cite[Theorem 4.13]{BHYY3}.
				Let $\mathcal{N} >\max(L,N)$. Let $\theta^{(i)}$ be the Daubechies wavelets of class $C^{\mathcal{N}}$. By Remark \ref{Db wavelet satisfy condition} and Theorem \ref{thm:wavelet homo}, we have
				\begin{equation*}
					\vec{f} = \sum_{i=1}^{2^n-1} \sum_{Q\in\mathcal{D}}\langle \vec{f}, \theta^{(i)}_Q \rangle \theta^{(i)}_Q 
				\end{equation*}
				in $\mathscr{S}'/\mathcal{P} (\rn)$  (denote $ t_{Q}^{(i)} = \langle \vec{f}, \theta^{(i)}_Q \rangle $) and
				\begin{equation*}
					\|\vec{f}\|_{\dot{F}_{p,t,r}^{s,q}(W)}\approx\sum_{i=1}^{2^{n}-1}\|\{\langle\vec{f},\theta_{Q}^{(i)}\rangle\}_{Q\in\mathcal{D}}\|_{\dot{f}_{p,t,r}^{s,q}(W)} .
				\end{equation*}
				It remains to rearrange this sum  into a sum over $Q\in\mathcal{D}$ only.
				For each $Q\in\mathcal{D}$, let $Q_i$, $i=1,\ldots,2^n$, be an enumeration of the dyadic child-cubes	of $Q$. By Definition \ref{Def: Db wavelet} and Remark \ref{Db wavelet satisfy condition}, for some constants $C$ and $b$, the functions
				\begin{equation*}
					a_{Q_i}:=
					\begin{cases}
						C\theta_Q^{(i)}, & \mathrm{if} \; i\in\{1,\ldots,2^n-1 \}, \\
						0,  & \mathrm{if} \; i= 2^n, \\
					\end{cases}
				\end{equation*}
				are $(b, \mathcal{N},\mathcal{N})$-atoms, hence also $(b,L,N)$-atoms, over the respective cubes indicated by their subscripts. We also define the coefficient
				\begin{equation*}
					\vec{t}_{Q_i}	:=
					\begin{cases}
						C^{-1}\vec t_Q^{(i)}, & \mathrm{if} \; i\in\{1,\ldots,2^n-1 \}, \\
						\vec{0},  & \mathrm{if} \; i= 2^n. \\
					\end{cases}
				\end{equation*}
				Since every $P\in\mathcal {D}$ is of the form $Q_i$ for some $Q\in\mathcal {D}$ and $i=1,\ldots,2^n-1$, we have
				\begin{equation*}
					\sum_{P\in \mathcal {D}} \vec{t}_P a_P = \sum_{i=1}^{2^n-1} \sum_{Q\in\mathcal{D}} \vec{t}_{Q_i} a_{Q_i} = \sum_{i=1}^{2^n-1} \sum_{Q\in\mathcal{D}} \vec{t}_{Q}^{(i)} \theta_{Q}^{(i)} = \vec{f},
				\end{equation*}
				which is the desired series representation of $\vec{f}$. While $\vec{t}$ has the same set of coefficients as the union of  $\{  \vec{t}^{(i)} \}_{i=1} ^{2^n -1}$, the indexing of these coefficients is shifted by one level. However, it is easy to verify from the definition of the norm $\|\cdot \|_{\dot{f}_{p,t,r}^{s,q}(A_Q)}$ that such a shift changes the norm at most by a positive constant $C$ that is independent of $\vec{t}$. By Theorem \ref{thm:equivalence},  the norm $\|\cdot \|_{\dot{f}_{p,t,r}^{s,q}(W)}$ has the same property as above.
				
				Conversely, since a $(b,L,N)$-atom is a harmless constant multiple of $(L  ,N -1, M,\delta$)-molecule for each $M>0, 0 < \delta\le 1$ by mean value theorem. 
				Since we choose $\mathcal{N} >\max(L,N)$,
				Theorem \ref{thm:varphi transform} and Theorem \ref{thm:equivalence} lead to the result.
			\end{proof}
		
		\subsection{Characterizations by approximation}
		Recall that the Fourier
		multiplier $\varphi(D)$ with symbol $\varphi$ is defined by
		\[
		\varphi(D)\vec{f}(x) :=\mathcal{F}^{-1}[\varphi\mathcal{F}(\vec{f})](x).
		\]
		For  $s\in \mathbb R$, we denote
		\begin{equation*}
			H_2^s : =\left\{
			f\in \mathscr S' : \| f\| _{ H_2^s }  :=\left\|  (1+ |\cdot|^2)^{s/2}  \mathcal F f (\cdot)    \right\|     _{L^2}
			<\infty	\right\}.
		\end{equation*}
		
		\begin{defn}
			Let $ 0<p<t<r <\infty $ or $0<p \le t <r =\infty $.
			Let  $0<q\le \infty$, $W$ be a matrix weight. Then the matrix weighted sequence valued Bourgain-Morrey spaces
			$  M_p^{t,r}(W) (\ell^q , \mathbb Z) $  are the set of all measurable function sequences  $\{\vec f_k \}_{k=-\infty}^\infty $
			such that
			\begin{equation*}
				\left\|  \left \{ \vec f_k \right \} _{k=-\infty}^\infty  \right\|_{ M_p^{t,r}(W) (\ell^q) } : =   \left\|  \left(   \sum_{k=-\infty}^\infty \left| W^{1/p} \vec f_k \right|  ^q  \right)^{1/q} \right\| _{ M_p^{t,r}} <\infty .
			\end{equation*}
			Similarly, 
			$  M_p^{t,r}(W) (\ell^q , \mathbb N_0) $  are the set of all measurable function sequences  $\{\vec f_k \}_{k=0}^\infty $
			such that
			\begin{equation*}
				\left\|  \left \{ \vec f_k \right \} _{k=0}^\infty  \right\|_{ M_p^{t,r}(W) (\ell^q) } : =   \left\|  \left(   \sum_{k=0}^\infty \left| W^{1/p} \vec f_k \right|  ^q  \right)^{1/q} \right\| _{ M_p^{t,r}} <\infty .
			\end{equation*}
			
		\end{defn}
		
		Let $p\in(0,\infty)$, $a>0$, and $W\in \mathcal A_{p}$. Recall that the matrix-weighted
		Peetre-type maximal functions for $\{\vec{f} _k\}$ 
		are defined by
		\begin{equation*}
			\Big( \vec{f} _k \Big)_{a}^{(W,p)}(x):=\sup_{y\in\mathbb{R}^{n}}\frac{|W^{1/p}(x) \vec{f} _k (y)|}{(1+2^{k}|x-y|)^{a}}
		\end{equation*}
		for any $k\in \mathbb Z$ and $x\in\mathbb{R}^{n}$.
		
		\begin{thm} \label{char peetre mix}
			Let $0<p<t<r<\infty$ or $0<p\le t<r=\infty$.
			Let $s\in\mathbb{R}$, $W\in  \mathcal  A_{p}$ and
			$\{A_{Q}\}_{Q\in\mathcal{D}}$ be a sequence of reducing operators of order p for $W$. 
			Let $ q \in (0, p + \delta_W) $ where $\delta_W >0$  is the same as in Lemma \ref{lem:W AQ}. 
			Let
			$a>n/\min(1,q,p )+\Delta$  where $\Delta$  is the same as in Lemma \ref{AQ AR improved}.
			
			{\rm (i)}
			Let $\{  \vec f_k \}_{k=-\infty}^\infty $ be a sequence of function such that for each $k \in \mathbb Z$, supp $\mathcal F \vec f_k \subset  \{\xi : 2^{k-1} \le  |\xi |  \le 2^{k+1}\}$.
			Then
			\begin{equation*}
				\left\|  \left\{ \vec f_k  \right\}_{k=-\infty}^\infty   \right\|_{ M_p^{t,r} (W) (\ell^q)  }  \approx
				\left\| \left\{ 	\Big( \vec{f} _k \Big)_{a}^{(W,p)} \right\} _{k=-\infty}^\infty   \right\|_{  M_p^{t,r} (W) (\ell^q)  }  .
			\end{equation*}
			
			{\rm (ii)}
			Let $\{  \vec f_k \}_{k=0}^\infty $ be a sequence of function such that for each $k \in \mathbb N_0$, supp $\mathcal F \vec f_k \subset  \{\xi : |\xi |  \le 2^{k+1}\}$.
			Then
			\begin{equation*}
				\left\|  \left\{ \vec f_k  \right\}_{k=0}^\infty   \right\|_{ M_p^{t,r} (W) (\ell^q)  }  \approx
				\left\| \left\{ 	\Big( \vec{f} _k \Big)_{a}^{(W,p)} \right\} _{k=0}^\infty   \right\|_{  M_p^{t,r} (W) (\ell^q)  }  .
			\end{equation*}
		\end{thm}
		\begin{proof}
			The proof is similar to that of Theorems 
			\ref{thm:Peetre char} and
			\ref{thm:Peetre char inho}. Since the support of $\mathcal F \vec f_k$ is the same as the support of  $\mathcal F ( 2^{ks} \varphi _k * \vec f ) $ in the proof of  Theorems 
			\ref{thm:Peetre char} and
			\ref{thm:Peetre char inho}, we only need to replace $ 2^{ks} \varphi_k * \vec f$ by $\vec f_k$ in Theorem \ref{thm:Peetre char} (or Theorem \ref{thm:Peetre char inho}).
			We omit the detail here.
		\end{proof}
		Using Theorem \ref{char peetre mix}, we have the following Fourier multiplier theorem.		
		\begin{thm}\label{multiplier}
			Let $0<p<t<r<\infty$ or $0<p\le t<r=\infty$.
			Let $s\in\mathbb{R}$, $W\in  \mathcal  A_{p}$ and
			$\{A_{Q}\}_{Q\in\mathcal{D}}$ be a sequence of reducing operators of order p for $W$. 
			Let $ q \in (0, p + \delta_W) $ where $\delta_W >0$  is the same as in Lemma \ref{lem:W AQ}. 
			If $a>n/\min(1,q,p )+\Delta$ where $\Delta$  is the same as in Lemma \ref{AQ AR improved}, then the following holds.
			
			{\rm (i)}	Let $\{ \vec f_k\} _{k=-\infty}^\infty $ be a sequence of function such that for each $k \in \mathbb Z$, supp $\mathcal F \vec f_k \subset  \{\xi :  2^{k-1}\le  |\xi |  \le 2^{k+1}\}$.
			Then for any $\epsilon>0$, we have
			\begin{equation*}
				\left\| \left \{ m_k (D)\vec f_k  \right \}_{k \in \mathbb Z}  \right\|_{ M_p^{t,r} (W) (\ell^q)  } 	\lesssim 	\left\| \left\{ \vec f_k  \right\}_{k \in \mathbb Z}  \right\|_{ M_p^{t,r} (W) (\ell^q)  }  \sup_{k\in \mathbb Z}  \| m_k (2^k \cdot) \|_{H_2^ {a+n/2+\epsilon} }  .
			\end{equation*}
			
			{\rm (ii)}	Let $\{ \vec f_k\} _{k=0}^\infty $ be a sequence of function such that for each $k \in \mathbb N_0$, supp $\mathcal F \vec f_k \subset  \{\xi : |\xi |  \le 2^{k+1}\}$.
			Then for any $\epsilon>0$, we have
			\begin{equation*}
				\left\| \left \{ m_k (D)\vec f_k  \right \}_{k=0}^\infty   \right\|_{ M_p^{t,r} (W) (\ell^q)  } 	\lesssim 	\left\| \left\{ \vec f_k  \right\}_{k=0}^\infty  \right\|_{ M_p^{t,r} (W) (\ell^q)  }  \sup_{k\ge 0}  \| m_k (2^k \cdot) \|_{H_2^ {a+n/2+\epsilon} }  .
			\end{equation*}
			
		\end{thm}
		
		\begin{proof}
			We only prove (ii)  since (i)  is similar. From \cite[Proof of Theorem 3.11]{BX24}, we have
			\begin{equation*}
				\left|  W^{1/p} (x) m_ k (D) \vec f_k  ( x)  \right| \lesssim 	\Big( \vec{f} _k \Big)_{a}^{(W,p)}(x)   \| (m_k (2^k \cdot)) \|_{H_2^ {a+n/2+\epsilon} } .
			\end{equation*}
			Since $a>n/\min(1,q,p )+\Delta$,
			we obtain
			\begin{align*}
				&	\left\| \left\{ m_k (D)\vec f_k \right \} _{k=0}^\infty  \right\|_{ M_p^{t,r} (W) (\ell^q)  } \\
				& \lesssim 	\left\| \left\{ 	\Big( \vec{f} _k \Big)_{a}^{(W,p)} \right\} _{k=0}^\infty   \right\|_{ L^p (W) (\ell^q)  }  \sup_{k \ge 0 }  \| m_k (2^k \cdot) \|_{H_2^ {a+n/2+\epsilon} } \\
				&	\lesssim 	\left\| \left\{ \vec f_k \right \} _{k=0}^\infty  \right\|_{ M_p^{t,r} (W) (\ell^q)  }  \sup_{k\ge 0}  \| m_k (2^k \cdot) \|_{H_2^ {a+n/2+\epsilon} } . \qedhere
			\end{align*}
		\end{proof}
			
				Now we begin to show the  characterization of $F_{p,t,r}^{s,q}(W)$  by means of approximation.
			
			Let $0<p<t<r<\infty$ or $0<p\le t<r=\infty$. Define $\mathcal{E}_{p}:=\mathcal{E}_{p}(\mathbb{R}^{n})$
			as the collection of sequences of functions $\vec{u}:=\{\vec{u}_{k}\}_{k\in\mathbb{N}_{0}}$
			with $\vec{u}_{k}\in  M_{p}^{t,r}\cap\mathscr{S}'(\mathbb{R}^{n}) $
			and ${\rm supp}(\mathcal{F}\vec{u}_{k})\subset\{\xi:|\xi|\le2^{k+1}\}$
			for each $k\in\mathbb{N}_{0}$.
			
			Let $W\in \mathcal A_{p}$ and set 
			\begin{equation*}
				\|\vec{f}\|_{F_{p,t,r}^{s,q}(W)}^{\vec{u}}:=\|W^{1/p}\vec{u}_{0}\|_{M_{p}^{t,r}}+\bigg\|\bigg(\sum_{k=0}^{\infty}2^{ksq}|W^{1/p}(\vec{f}-\vec{u}_{k})|^{q}\bigg)^{1/q}\bigg\|_{M_{p}^{t,r}}.
			\end{equation*}

			\begin{thm}
				\label{thm:CHAR by app Tl} Let $0<p<t<r<\infty$ or $0<p\le t<r=\infty$.
				Let $W\in  \mathcal  A_{p}$ and  let $ q \in (0, p + \delta_W) $ where $\delta_W >0$  is the same as in Lemma \ref{lem:W AQ}.  Let  $\Delta$  be the same as in Lemma \ref{AQ AR improved}.
				If $s> n/\min(1,q,p )+\Delta $, then 
				$\vec{f}\in\mathscr{S}'(\mathbb{R}^{n})$ belongs
				to $F_{p,t,r}^{s,q}(W)$ if and only if there exists $\vec{u}:=\{\vec{u}_{k}\}_{k\in\mathbb{N}_{0}}\in\mathcal{E}_{p}$
				such that $\vec{f}=\lim_{k\to\infty}\vec{u}_{k}$ in $\mathscr{S}'(\mathbb{R}^{n})$
				and $\|\vec{f}\|_{F_{p,t,r}^{s,q}(W)}^{\vec{u}}<\infty.$ Moreover,
				\[
				\|\vec{f}\|_{F_{p,t,r}^{s,q}(W)}^{\diamondsuit}:=\inf\|\vec{f}\|_{F_{p,t,r}^{s,q}(W)}^{\vec{u}},
				\]
				where the infimum is taken over all admissible functions $\vec{u}:=\{\vec{u}_{k}\}_{k\in\mathbb{N}_{0}}\in\mathcal{E}_{p}$,
				is an equivalent quasi-norm in $F_{p,t,r}^{s,q}(W)$.

			\end{thm}
			
			\begin{proof}
				We first show that
				\begin{equation}
					\|\vec{f}\|_{F_{p,t,r}^{s,q}(W)}^{\diamondsuit}\lesssim\|\vec{f}\|_{F_{p,t,r}^{s,q}(W)}.\label{eq:clubsuit le}
				\end{equation}
				Let $\{\varphi_{j}\}_{j\in\mathbb{N}_{0}}$ be functions in $\mathbb{R}^{n}$
				as in Definition \ref{def:partition of unity inho}, then
				\[
				\vec{u}_{j}:=\sum_{k=0}^{j}\varphi_{k}(D)\vec{f}\to\vec{f}\;\mathrm{in}\;\mathscr{S}'(\mathbb{R}^{n})\;\mathrm{as}\;j\to\infty.
				\]
				Thus $\{\vec{u}_{j}\}_{j\in\mathbb{N}_{0}}\in\mathcal{E}_{p}$ and
				from \cite[page 16]{BX24},
				for $0<q\le\infty$,
				we have
				\[
				\bigg(\ensuremath{\sum_{j=0}^{\infty}2^{jsq}|W^{1/p}(\vec{f}-\vec{u}_{j})|^{q}\bigg)^{1/q}\lesssim\bigg(\ensuremath{\sum_{j=0}^{\infty}2^{jsq}|W^{1/p}(x)\varphi_{j}(D)\vec{f}(x)|^{q}\bigg)^{1/q}.}}
				\]
				Taking the $M_{p}^{t,r}$-quasi-norm on the above inequality, we obtain
				(\ref{eq:clubsuit le}) since
				\[
				\|W^{1/p}\vec{u}_{0}\|_{M_{p}^{t,r}}=\|W^{1/p}\varphi_{0}(D)\vec{f}\|_{M_{p}^{t,r}}\le\|\vec{f}\|_{F_{p,t,r}^{s,q}(W)}.
				\]
				
				Now we prove the opposite inequality of (\ref{eq:clubsuit le}).
				Suppose that $\vec{u}:=\{\vec{u}_{k}\}_{k\in\mathbb{N}_{0}}\in\mathcal{E}_{p}$
				such that $\vec{f}=\lim_{k\to\infty}\vec{u}_{k}$ in $\mathscr{S}' (\rn)$
				and $\|\vec{f}\|_{F_{p,t,r}^{s,q}(W)}^{\vec{u}}<\infty.$ 				
				Then for
				$j\in\mathbb{N}_{0}$,
				we have
				\[
				\left| W^{1/p} (x) \varphi_{j}(D)\vec{f}(x) \right|  = 	\left| \sum_{k=-1}^\infty  W^{1/p} (x) \varphi_{j}(D)    ( (\vec{u}_{k+j}(x)-\vec{u}_{k+j-1}) (x) ) \right|,
				\]
				where  $\vec{u}_{-1}=  \vec u _{-2} = 0.$
				From \cite[page 16]{BX24}, we have
				\begin{equation*}
					\sup_{j\ge 0} 	\| \varphi_j  (2^{k+j} \cdot) \|_{H_2^ {a+n/2+\epsilon} }  \le c 2^{k (a +\epsilon)}.
				\end{equation*}
				Since supp $\mathcal F  (\vec{u}_{k+j}-\vec{u}_{k+j-1}) \subset \{ \xi \in \rn : |\xi|  \le 2^{k+j+1}\}$,
				using Theorem \ref{multiplier}, letting $J = \min(1,p,q)$,  we get that $ \| \vec f \|_{ F^{s,q}_{p,t,r}(W) } ^J$ is bounded by
				\begin{align*}
					& 	\sum_{k=-1}^\infty   2^{-sk J}	\left\|
					\left(
					\sum_{j=0}^\infty 2^{s(k+j) q } \left|  W^{1/p} \varphi_j (D)  ( \vec{u}_{k+j} -\vec{u}_{k+j-1})  ) \right| ^q  \right)^{1/q}  \right\| _{ M_p^{t,r} } ^J	\\
					&\lesssim  \sum_{k=-1}^\infty   2^{-sk J}	 2^{k (a+\epsilon) J}  \left\|
					\left(
					\sum_{j=0}^\infty 2^{s(k+j) q } \left|  W^{1/p}  ( \vec{u}_{k+j} -\vec{u}_{k+j-1})  ) \right| ^q  \right)^{1/q}  \right\| _{ M_p^{t,r} } ^J \\
					&\lesssim \left(  \|W^{1/p}\vec{u}_{0}\|_{ M_p^{t,r}  }+\Big\|\Big(\sum_{k=0}^{\infty}2^{ksq}|W^{1/p}(\vec{f}-\vec{u}_{k})|^{q}\Big)^{1/q}\Big\|_{ M_p^{t,r}  } \right) ^J,
				\end{align*}
				where in the second step, we used the fact that $ s > a+\epsilon > n / \min(1,p,q) +\Delta $, where the $\epsilon$ is sufficiently small and $a $  sufficiently close to $ n / \min(1,p,q) +\Delta   $.
				Taking the infimum over $\vec{u}$, we conclude that $\|\vec{f}\|_{F_{p}^{s,q}(W)}\lesssim\|\vec{f}\|_{F_{p}^{s,q}(  W ) }^{\diamondsuit}$.
			\end{proof}

			\section{Pseudo-differential operators}\label{sec:Pseudo-differential-operators}
			
			In the last section, we will consider the boundedness of pseudo-differential operators on inhomogeneous matrix weighted Bourgain-Morrey Triebel-Lizorkin spaces. 
				Since inhomogeneous spaces are independent of the choice of $(\Phi,\varphi)\in\mathcal{A}_{+}$,
			we use the following partition of unity on $\mathbb{R}^{n}$ when dealing with inhomogeneous spaces.
			\begin{defn}
				\label{def:partition of unity inho} Let $\varphi_{0}$ be a Schwartz
				function such that supp$(\varphi_{0})\subset\{\xi\in\mathbb{R}^{n}:|\xi|\le2\}$
				and $\varphi_{0}(\xi)=1$ for $|\xi|\le1$. Moreover, put $\varphi_{j}(\xi)=\varphi_{0}(2^{-j}\xi)-\varphi_{0}(2^{-j+1}\xi)$
				for $j\in\mathbb{N}$. Then supp$(\varphi_{j})\subset\{\xi \in \rn :2^{j-1}\le|\xi|\le2^{j+1}\}$
				for all $j\in\mathbb{N}$ and
				\[
				\sum_{j=0}^{\infty}\varphi_{j}(\xi)=1
				\]
				for $\xi\in\mathbb{R}^{n}$. Hence $\{\varphi_{j}\}_{j\in\mathbb{N}_{0}}$
				is a partition of unity on $\mathbb{R}^{n}$ subordinated to the dyadic
				rings $\{\xi:2^{j-1}\le|\xi|\le2^{j+1}\}$, $j\in\mathbb{N}$, and
				$\overline{B(0,2)}$.
				
				We also set $\tilde{\varphi}_{0}:=\varphi_{0}+\varphi_{1}$ and $\tilde{\varphi}_{j}:=\varphi_{j-1}+\varphi_{j}+\varphi_{j+1}$
				for $j\in\mathbb{N}$. Note that $\varphi_{j}\tilde{\varphi}_{j}=\varphi_{j}$
				for $j\in\mathbb{N}_{0}$ and
				\[
				\mathrm{supp}(\tilde{\varphi_{j}})\subset\{\xi\in\mathbb{R}^{n}:2^{j-2}\le|\xi|\le2^{j+2}\}\;\mathrm{for}\;j\ge2,
				\]
				\[
				\mathrm{supp}(\tilde{\varphi_{j}})\subset\{\xi\in\mathbb{R}^{n}:|\xi|\le2^{j+2}\}\;\mathrm{for}\;j=0,1.
				\]
			\end{defn}
			
			For a suitable function $f$, for example the Schwartz functions, we define
			\begin{equation*}
				\varphi_j (D) f (x) : = \int_\rn \varphi_j(\xi) \mathcal F f (\xi )  e^{2\pi i  x \cdot \xi } \d \xi 
			\end{equation*}
			and for a vector function $\vec f$, we define $	\varphi_j (D)  \vec f = (\varphi_j (D)   f_1, \dots, \varphi_j (D)  f_d )^T$ where $T$ denotes the transpose of the row vector.
			
			Then 
			the matrix weighted inhomogeneous Bourgain-Morrey Triebel-Lizorkin
			spaces $F_{p,t,r}^{s,q}(W)$ can be defined by the set of all distribution $\vec{f}\in\mathscr{S}'(\mathbb{R}^{n})$
			such that
			\[
			\|\vec{f}\|_{F_{p,t,r}^{s,q}(W)}:=\bigg\|\bigg(\sum_{v=0}^{\infty}2^{vsq}|W^{1/p}\varphi_{v}(D)\vec{f}|^{q}\bigg)^{1/q}\bigg\|_{M_{p}^{t,r}}<\infty.
			\]
			Suppose that for each $Q\in\mathcal{D}$, $A_{Q}$ is a $d\times d$
			non-negative definite matrix.
			The \{$A_{Q}$\}-inhomogeneous Bourgain-Morrey Triebel-Lizorkin
			spaces $F_{p,t,r}^{s,q}(A_{Q})$ is the set of all distribution $\vec{f}\in\mathscr{S}'(\mathbb{R}^{n})$
			such that
			\[
			\|\vec{f}\|_{F_{p,t,r}^{s,q}(A_{Q})}:=\bigg\|\bigg(\sum_{v=0}^{\infty}\sum_{Q\in\mathcal{D}_{v}}2^{vsq}|A_{Q}\varphi_{v}(D)\vec{f}|^{q}\chi_{Q}\bigg)^{1/q}\bigg\|_{M_{p}^{t,r}}<\infty.
			\]

			Next we will show the boundedness of pseudo-differential operator with symbols in H\"{o}rmander classes.
			\subsection{H\"{o}rmander classes}
			\begin{defn}
				\label{def:hormander} Let $m\in\mathbb{R}$, $n\in\mathbb{N}$, $\rho,\delta\in[0,1]$.
				The H\"{o}rmander class $S_{\rho,\delta}^{m}(\mathbb{R}^{n}\times\mathbb{R}^{n})$
				is the set of all smooth functions $\sigma:\mathbb{R}^{n}\times\mathbb{R}^{n}\to\mathbb{C}$
				such that
				\[
				|\partial_{\xi}^{\alpha}\partial_{x}^{\beta}\sigma(x,\xi)|\le C_{\alpha,\beta}(1+|\xi|)^{m-\rho|\alpha|+\delta|\beta|}
				\]
				holds for all $\alpha,\beta\in\mathbb{N}_{0}^{n}$, where $C_{\alpha,\beta}$
				is independent of $x,\xi\in\mathbb{R}^{n}$. The function $\sigma$
				is called a pseudo-differential symbol and $m$ is called the order
				of $\sigma$.
			\end{defn}
			
			A pseudo-differential operator on $\mathscr{S}(\mathbb{R}^n)$ with symbol $\sigma\in S_{1,\delta}^{m}(\mathbb{R}^{n}\times\mathbb{R}^{n})$ is defined by
			\[
			\sigma(x,D)f(x) : =\int_{\mathbb{R}^{n}}\sigma(x,\xi)\mathcal{F}f(\xi)e^{2\pi ix\cdot\xi}\mathrm{d}\xi
			\]
			for $f\in \mathscr{S}(\mathbb{R}^n)$. Then, one can extend $a(\cdot,D)$ to a continuous linear operator on $\mathscr{S}^{\prime}(\mathbb{R}^n)$. Moreover, if $\sigma \in S^m_{1,0}$, then $\sigma(\cdot,D)$ is a bounded operator mapping  the Sobolev space $H_p^{s+m}(\mathbb{R}^n)$ into the Sobolev space $H_p^s(\mathbb{R}^n)$ for
			$s\in (-\infty,\infty)$ and $p\in (1,\infty)$; see \cite{w2}. Similarly, if $\sigma \in S^m_{1,0}$, $s\in (-\infty,\infty)$, $p\in (1,\infty)$ and $w\in A_p$ (Muckenhoupt scalar weight), then $\sigma(\cdot,D)$ is bounded from the weighted Sobolev space $H_p^{s+m}(\mathbb{R}^n, w)$ into the weighted Sobolev space $H_p^s(\mathbb{R}^n,w)$; see \cite{w1}. The continuity of $\sigma(\cdot,D)$ in Besov spaces and Triebel-Lizorkin spaces
			has been developed by Runst \cite{ru}, Torres \cite{tor}, and Johnsen \cite{joh}. 
			Indeed, let $m\in \mathbb{R}$, $0<p,q\leq \infty$ and $\sigma \in S^m_{1,1}.$ Then
			$\sigma(\cdot,D)$ is bounded from $F^{s+m}_{p,q}(\mathbb{R})$ to $F^{s}_{p,q}(\mathbb{R})$ if $p<\infty$ and $s>n/\min(1,p, q)-n$, and from $B^{s+m}_{p,q}(\mathbb{R})$ to $B^{s}_{p,q}(\mathbb{R})$ if $s>n/\min\{1,p, q\}-n$.
			In \cite{BX24}, the first author and the corresponding author of this paper obtained the boundedness of pseudo-differential operators with symbol in H\"ormander's class $S^m_{1,1} $ on matrix weighted Besov and Triebel-Lizorkin spaces.

			In the following, we consider the boundedness of $\sigma(\cdot,D)$ on matrix weighted Bourgain-Morrey Triebel-Lizorkin spaces.
			Denote
			\[ \sigma(x,D)\vec{f}(x) :=\int_{\mathbb{R}^{n}}\sigma(x,\xi)\mathcal{F}(\vec{f})(\xi)e^{2\pi ix\cdot\xi}\mathrm{d}\xi .\]
			To proceed, we decompose $\sigma(\cdot,D).$ Let $\{\varphi_{j}\}_{j\in\mathbb{N}_{0}}$ be a partition of unity
			as in Definition \ref{def:partition of unity inho}. Then
			\begin{align*}
				\sigma(x,D)\vec{f}(x)
				& =\sum_{j=0}\int_{\mathbb{R}^{n}}\sigma(x,\xi)\varphi_{j}(\xi)\mathcal{F}(\vec{f})(\xi)e^{2\pi ix\cdot\xi}\mathrm{d}\xi\\
				& =\sum_{j=0}\int_{\mathbb{R}^{n}}\bigg(\int_{\mathbb{R}^{n}}\hat{\sigma}^{1}(\zeta,\xi)e^{2\pi ix\cdot\zeta}\mathrm{d}\zeta\bigg)\varphi_{j}(\xi)\mathcal{F}(\vec{f})(\xi)e^{2\pi ix\cdot\xi}\mathrm{d}\xi
			\end{align*}
			where $\hat{\sigma}^{1}(\cdot,\cdot)$ denotes the Fourier transform
			of $\sigma(\cdot,\cdot)$ with respect to the first variable. Using
			a partition of unity in the variable $\zeta$ we write
			\[
			\sigma(x,D)\vec{f}(x)=\sum_{j=0}\sum_{l\ge0}\sigma_{j,l}(x,D)\vec{f}(x)
			\]
			with symbols given by
			\[
			\sigma_{j,0}(x,\xi):=\varphi_{j}(\xi)\int_{\mathbb{R}^{n}}\bigg(\sum_{v=0}^{j}\varphi_{v}(\zeta)\bigg)\hat{\sigma}^{1}(\zeta,\xi)e^{2\pi ix\cdot\zeta}\mathrm{d}\zeta,
			\]
			\[
			\sigma_{j,l}(x,\xi):=\varphi_{j}(\xi)\int_{\mathbb{R}^{n}}\varphi_{j+l}(\zeta)\hat{\sigma}^{1}(\zeta,\xi)e^{2\pi ix\cdot\zeta}\mathrm{d}\zeta,\;l\ge1.
			\]

			Put
			\[
			M_{j,l}(x,y):=\mathcal{F}^{2}(\sigma_{j,l}(x,\cdot))(y)\quad x,y\in\mathbb{R}^{n},
			\]
			where $\mathcal{F}^{2}(\sigma_{j,l}(x,\cdot))$ stands for the Fourier
			transform of $\sigma_{j,l}(x,\cdot)$ with respect to the second variable.

			\begin{lem}
				\label{lem:Lplq less F} Let $0<p<t<r<\infty$ or $0<p\le t<r=\infty$.
				Let $m\in\mathbb{R}$, $s\in\mathbb{R}$, $W\in  \mathcal  A_{p}$ and
				$\{A_{Q}\}_{Q\in\mathcal{D}}$ be a sequence of reducing operators of order p for $W$. 
				Let $ q \in (0, p + \delta_W) $ where $\delta_W >0$  is the same as in Lemma \ref{lem:W AQ}. 
				Let $\Delta$  be the same as in Lemma \ref{AQ AR improved}.
				Let $N,b\in\mathbb{N}_{0}$ be even numbers with $b>n/\min(1,p,q)+\Delta +n/2$.
				Then there exists a constant $c>0$ such that for $\sigma\in S_{1,1}^{m}$,
				$l\in\mathbb{N}_{0}$,
				\[
				\bigg\|\Big(2^{js}|W^{1/p}T_{\sigma_{j,l}}(\vec{f})|\Big\}_{j}\bigg\|_{M_{p}^{t,r}(\ell^{q})}\le2^{-lN}\|\sigma\|_{N,b}\|\vec{f}\|_{F_{p,t,r}^{s+m,q}(W)}.
				\]
				
			\end{lem}
			
			\begin{proof}
				Let $\vec{f_{j}}:=\tilde{\varphi_{j}}(D_{x})\vec{f}$ where $\tilde{\varphi_{j}}$ is the same as in Definition \ref{def:partition of unity inho}.				
				From \cite[Lemma 4.4]{BX24}, we have
				\begin{equation*}
					 \Big|W^{1/p}(x)\int_{\mathbb{R}^{n}}\sigma_{j,l}(x,\xi)\mathcal{F}\vec{f_{j}}(\xi)e^{2\pi ix\cdot\xi}\mathrm{d}\xi\Big| \lesssim2^{jm}\Big(\varphi_{j}^{*}\vec{f}\Big)_{a}^{(W,p)}(x)2^{-lN}\|\sigma\|_{N,b}.
				\end{equation*}
				Hence by Theorem \ref{thm:Peetre char},
				\[
				\bigg\|\Big(2^{js}|W^{1/p}T_{\sigma_{j,l}}(\vec{f})|\Big\}_{j}\bigg\|_{M_{p}^{t,r}(\ell^{q})}\le2^{-lN}\|\sigma\|_{N,b}\|\vec{f}\|_{F_{p,t,r}^{s+m,q}(W)}.\qedhere
				\]
			\end{proof}

			\begin{lem} \label{har pse}
				Let $0<p<t<r<\infty$ or $0<p\le t<r=\infty$. 
				Let $m\in\mathbb{R}$, $s\in\mathbb{R}$ and $W\in  \mathcal  A_{p}$. 
				Let $ q \in (0, p + \delta_W) $ where $\delta_W >0$  is the same as in Lemma \ref{lem:W AQ}. 
				Let $\Delta$  be the same as in Lemma \ref{AQ AR improved}.
				Let $s> n/ \min(1, p, q) + \Delta .$
				Let $N,b\in\mathbb{N}_{0}$ be even numbers with $b>n/\min(1,p,q)+\Delta +n/2$.
				Then there exists a constant $c>0$ such that for $\sigma\in S_{1,1}^{m}$,
				$l\in\mathbb{N}_{0}$,
				\[
				\bigg\|\sum_{j=0}^{\infty}T_{\sigma_{j,l}}(\vec{f})\bigg\|_{F_{p,t,r}^{s,q}(W)}\le c2^{sl}2^{-lN}\|\sigma\|_{N,b}\|\vec{f}\|_{F_{p,t,r}^{s+m,q}(W)}.
				\]
			\end{lem}
			
			\begin{proof}
				For each fixed $l\in\mathbb{N}_{0},$
				set $\vec{h}_{l}:=\sum_{j=0}^{\infty}T_{\sigma_{j,l}}(\vec{f})$.
				It suffices to estimate $\|\vec{h_{l}}\|_{F_{p}^{s,q}(W)}^{\vec{u_{l}}}$
				for an appropriate sequence of functions $\vec{u_{l}}$ from Theorem
				\ref{thm:CHAR by app Tl}. Define the sequence $\vec{u_{l}}:=\{\vec{u}_{k,l}\}_{k\in\mathbb{N}_{0}}$
				as follows
				\[
				\vec{u}_{k,l}:=\begin{cases}
					0 & \mathrm{if}\;k\le l-1,\\
					\sum_{v=0}^{k-l}T_{\sigma_{v,l}}(\vec{f}) & \mathrm{if}\;k>l.
				\end{cases}
				\]
				It is easy to see that $\vec{h}_{l}=\lim_{k\to\infty}\vec{u}_{k,l}$
				in $\mathscr{S}'(\mathbb{R}^{n})$. 				
				From \cite[proof of Lemma 4.6]{BX24}, supp $\mathcal{F}\vec{u}_{k,l}\subset\{\xi:|\xi|\le2^{k+3}\}$
				for all $k,l\in\mathbb{N}_{0}$.				
				For $s> n/ \min(1, q) + \Delta >0$, by Lemma \ref{lem:Lplq less F} we have
				\begin{align}
					\|\vec{h}_{l}\|_{M_{p}^{t,r}(W)} & =\bigg\|\sum_{j=0}^{\infty}T_{\sigma_{j,l}}(\vec{f})\bigg\|_{M_{p}^{t,r}(W)}\nonumber \\
					\nonumber
					& \le\|\{2^{js}W^{1/p}T_{\sigma_{j,l}}(\vec{f})\}_{j}\|_{M_{p}^{t,r}(\ell^{q})} \\
					& \lesssim2^{-lN}\|\sigma\|_{N,b}\|\vec{f}\|_{F_{p,t,r}^{s+m,q}(W)}.\label{eq:LPW F}
				\end{align}				
				Note that $\vec{u}_{0,l}=0$ if $l\in\mathbb{N}$, $\vec{u}_{0,0}=T_{\sigma_{0,0}}(\vec{f})$
				and that (\ref{eq:LPW F}) implies that
				\begin{align*}
					\|T_{\sigma_{0,l}}(\vec{f})\|_{M_{p}^{t,r}(W)} & \lesssim2^{-lN}\|\sigma\|_{N,b}\|\vec{f}\|_{F_{p,t,r}^{s+m,q}(W)}
				\end{align*}
				for all $l\in\mathbb{N}_{0}$. We now estimate $\|\{2^{ks}|\vec{h}_{l}-\vec{u}_{k,l}|\}_{k\in\mathbb{N}_{0}}\}\|_{M_{p}^{t,r}(W)(\ell^{q})}$
				by breaking the sum in $k$ into $k\le l-1$ and $k\ge l$. Since
				$\vec{u}_{k,l}=0$ if $k\le l-1$, for the first part we obtain
				\begin{align*}
					\bigg\|\bigg(\sum_{k=0}^{l-1}2^{ksq}|\vec{h}_{l}|^{q}\bigg)^{1/q}\bigg\|_{M_{p}^{t,r}(W)} & =\bigg(\sum_{k=0}^{l-1}2^{ksq}\bigg)^{1/q}\|\vec{h}_{l}\|_{M_{p}^{t,r}(W)}\\
					& \lesssim 2^{sl}2^{-lN}\|\sigma\|_{N,b}\|\vec{f}\|_{F_{p,t,r}^{s+m,q}(W)}.
				\end{align*}
				For the second part (that is, when $k\ge l$) we get
				\begin{align*}
					\vec{h}_{l}-\vec{u}_{k,l} & =\sum_{v=k-l+1}^{\infty}T_{\sigma_{v,l}}(\vec{f})=\sum_{v=1}^{\infty}T_{\sigma_{k-l+v,l}}(\vec{f})
				\end{align*}
				and then
				\begin{align*}
					& \|\{2^{ks}|\vec{h}_{l}-\vec{u}_{k,l}|\}_{k=l}^{\infty}\|_{M_{p}^{t,r}(W)(\ell^{q})}^{\min(p,q,1)}\\
					& \le \bigg\|\bigg\{\sum_{v=1}^{\infty}2^{(l-v)s}2^{(k-l+v)s}T_{\sigma_{k-l+v,l}}(\vec{f})\Big\}_{k=l}^{\infty}\bigg\|_{M_{p}^{t,r}(W)(\ell^{q})}^{\min(p,q,1)} \\
					& \lesssim\sum_{v=1}^{\infty}2^{(l-v)s\min(p,q,1)}\|\{2^{ks}T_{\sigma_{k,l}}(\vec{f})\}_{k=0}^{\infty}\|_{M_{p}^{t,r}(W)(\ell^{q})}^{\min(p,q,1)}\\
					& \lesssim\sum_{v=1}^{\infty}2^{(l-v)s\min(p,q,1)}\Big(2^{-lN}\|\sigma\|_{N,b}\|\vec{f}\|_{F_{p,t,r}^{s+m,q}(W)}\Big)^{\min(p,q,1)}\\
					& \lesssim \Big(2^{(s-N)l}\|\sigma\|_{N,b}\|\vec{f}\|_{F_{p,t,r}^{s,q}(W)}\Big)^{\min(p,q,1)} .
				\end{align*}
				Hence we obtain
				\begin{align*}
					\|\vec{h}_{l}\|_{F_{p,t,r}^{s,q}(W)} & \lesssim\|\vec{u}_{0,l}\|_{M_{p}^{t,r}(W)}+\|\{2^{ks}(\vec{h_{l}}-\vec{u}_{k,l})\}_{k\in\mathbb{N}_{0}}\|_{M_{p}^{t,r}(W)(\ell^{q})}\\
					& \lesssim2^{-lN}\|\sigma\|_{N,b}\|\vec{f}\|_{F_{p,t,r}^{s+m,q}(W)}+2^{sl}2^{-lN}\|\sigma\|_{N,b}\|\vec{f}\|_{F_{p,t,r}^{s+m,q}(W)}\\
					& +2^{(s-N)l}\|\sigma\|_{N,b}\|\vec{f}\|_{F_{p,t,r}^{s+m,q}(W)}\\
					& \lesssim2^{sl}2^{-lN}\|\sigma\|_{N,b}\|\vec{f}\|_{F_{p,t,r}^{s+m,q}(W)}.\qedhere
				\end{align*}
			\end{proof}
			\begin{thm}
				\label{thm:bounded all rho delta} Let $0<p<t<r<\infty$ or $0<p\le t<r=\infty$.
				Let $m\in\mathbb{R}$, $s\in\mathbb{R}$ and $W\in  \mathcal  A_{p}$. 
				Let $ q \in (0, p + \delta_W) $ where $\delta_W >0$  is the same as in Lemma \ref{lem:W AQ}. 
				Let $\Delta$  be the same as in Lemma \ref{AQ AR improved}.
				Let $s> n/ \min(1, p, q) + \Delta .$
				Let $N,b\in\mathbb{N}_{0}$ be even numbers with $b>n/\min(1,p,q)+ \Delta +n/2$,
				$N>s$. For $\sigma\in S_{1,1}^{m}$,
				\begin{align*}
					\|\sigma(\cdot,D)\vec{f}\|_{F_{p,t,r}^{s,q}(W)} & \lesssim\|\sigma\|_{N,b}\|\vec{f}\|_{F_{p,t,r}^{s+m,q}(W)}.
				\end{align*}
			\end{thm}
			
			\begin{proof}
				By Lemma \ref{har pse} and $N>s$, we obtain
				\begin{align*}
					\|\sigma(\cdot,D)\vec{f}\|_{F_{p,t,r}^{s,q}(W)}^{\min(1,p,q)} & \lesssim\sum_{l\ge0}^{\infty}2^{(s-N)l\min(1,p,q)}\|\sigma\|_{N,b}^{\min(1,p,q)}\|\vec{f}\|_{F_{p,t,r}^{s+m,q}(W)}^{\min(1,p,q)}\\
					& \lesssim\|\sigma\|_{N,b}^{\min(1,p,q)}\|\vec{f}\|_{F_{p,t,r}^{s+m,q}(W)}^{\min(1,p,q)}.\qedhere
				\end{align*}
			\end{proof}
			\begin{lem}[Theorem 5.14, \cite{S18}] \label{lem:decompositon of pseudo}Let $m_{1},m_{2}\in\mathbb{R}$,
				$0\le\delta<\rho\le1$. For $a\in S_{\rho,\delta}^{m_{1}}$, $b\in S_{\rho,\delta}^{m_{2}}$,
				let
				\[
				c(x,\xi):=\int_{\mathbb{R}^{2n}}a(x,\eta)b(y,\xi)e^{2\pi i(x-y)\cdot(\eta-\xi)}\mathrm{d}y\mathrm{d}\eta\quad x,\xi\in\mathbb{R}^{n}.
				\]
				Then $c\in S_{\rho,\delta}^{m_{1}+m_{2}}$ and c satisfies $c(\cdot,D)=a(\cdot,D)\circ b(\cdot,D)$.
			\end{lem}
			
			Let $\gamma\in\mathbb{R}.$ Denote the Bessel potential of order $\gamma$
			by $G_{\gamma}\vec{f}(x):=\mathcal{F}^{-1}[(1+|\cdot|^{2})^{-\gamma/2} \mathcal F \vec{f}](x)$ for each $x\in\mathbb{R}^{n}$.
			\begin{thm}
				\label{lem:Bessel T-L} Let $0<p<t<r<\infty$ or $0<p\le t<r=\infty$.
				Let  $s\in\mathbb{R}$, $W\in  \mathcal  A_{p}$. 
				Let $ q \in (0, p + \delta_W) $ where $\delta_W >0$  is the same as in Lemma \ref{lem:W AQ}. 
				Then
				\[
				\|G_{-\gamma}\vec{f}\|_{F_{p,t,r}^{s,q}(W)}\approx\|\vec{f}\|_{F_{p,t,r}^{s+\gamma,q}(W)} .
				\]
			\end{thm}
			
			\begin{proof}
				Let $\{\varphi_{j}\}_{j\in\mathbb{N}_{0}}$ be functions on $\mathbb{R}^{n}$
				as in Definition \ref{def:partition of unity inho}. For $j\in\mathbb{N}_{0}$,
				define $\psi(\xi):=\varphi(\xi)(1+|\xi|)^{\gamma/2}$ and $\psi_{j}(\xi):=\psi(2^{-j}\xi)$.
				Then we have
				\begin{align*}
					\varphi_{j}(D)(G_{-\gamma}\vec{f}) & =\mathcal{F}^{-1}(\varphi_{j}(1+|\cdot|^{2})^{\gamma/2}\mathcal{F}\vec{f})\\
					& \approx2^{j\gamma}\mathcal{F}^{-1}(\varphi_{j}(1+2^{-j}|\cdot|^{2})^{\gamma/2}\mathcal{F}\vec{f})\\
					& =2^{j\gamma}\mathcal{F}^{-1}(\psi_{j} \mathcal{F}\vec{f}) .
				\end{align*}				
				From Theorem \ref{thm:equivalence inhomogeneous}, the function spaces $F_{p,t,r}^{s,q}(W)$
				are independent of the choice of $\{\varphi_{j}\}_{j\in\mathbb{N}_{0}}$
				and $\{\psi_{j}\}_{j\in\mathbb{N}_{0}}$, thus we have
				\[
				\|G_{-\gamma}\vec{f}\|_{F_{p,t,r}^{s,q}(W)}\approx\|\vec{f}\|_{F_{p,t,r}^{s+\gamma,q}(W)} . \qedhere
				\]
			\end{proof}
			\begin{cor} \label{har m< 1}
				Let $0<p<t<r<\infty$ or $0<p\le t<r=\infty$.
				Let $m\in\mathbb{R}$, 	$0\le\delta<1$, $s\in\mathbb{R}$ and $W\in  \mathcal  A_{p}$. 
				Let $ q \in (0, p + \delta_W) $ where $\delta_W >0$  is the same as in Lemma \ref{lem:W AQ}. 
				For $\sigma\in S_{1,\delta}^{m}$,
				\begin{align*}
					\|\sigma(\cdot,D)\vec{f}\|_{F_{p,t,r}^{s,q}(W)} & \lesssim\|\vec{f}\|_{F_{p,t,r}^{s+m,q}(W)}.
				\end{align*}
				
			\end{cor}
			
			\begin{proof}
				Since $\sigma\in S_{1,\delta}^{m}$,
				use Lemma \ref{lem:decompositon of pseudo}, and we have a decomposition:
				\[
				\sigma(\cdot,D)=G_{-M}\circ\sigma'(x,D)\circ G_{M}
				\]
				for some $\sigma'\in S_{1,\delta}^{m}\subset S_{1,1}^{m}.$ From Lemma
				\ref{lem:Bessel T-L}, we have
				\[
				\|\sigma(\cdot,D)\vec{f}\|_{F_{p,t,r}^{s,q}(W)}\approx\|\sigma'(x,D)\circ G_{M}\vec{f}\|_{F_{p,t,r}^{s+M,q}(W)}.
				\]
				Choose $M$ sufficiently large such that $s+M>  n / \min(1,p,q) +\Delta  $ where 	 $\Delta$  is the same as in Lemma \ref{AQ AR improved}. Hence
				combining Theorem \ref{thm:bounded all rho delta} and Lemma \ref{lem:Bessel T-L}
				we have
				\begin{align*}
					\|\sigma'(\cdot,D)\circ G_{M}\vec{f}\|_{F_{p,t,r}^{s+M,q}(W)} & \lesssim\|G_{M}\vec{f}\|_{F_{p,t,r}^{s+M+m,q}(W)}\\
					& \approx\|\vec{f}\|_{F_{p,t,r}^{s+m,q}(W)}.\qedhere
				\end{align*}
			\end{proof}
			
			\subsection{H\"{o}lder-Zygmund Class $C_{*}^{\ell}S_{1,\delta}^{m}$}
			Finally we show the boundedness of pseudo-differential operators  with symbol in H\"{o}lder-Zygmund Class $C_{*}^{\ell}S_{1,\delta}^{m}$. To do so, we need some preparation.			
			Let $v\in\mathbb{N}_{0}$ and $m>0$. Define
			\[
			\eta_{v,m}(x):=\frac{2^{nv}}{(1+2^{v}|x|)^{m}},\ x\in\mathbb{R}^{n}.
			\]
			Note that $\eta_{v,m}\in L^{1}$ when $m>n$ and $\|\eta_{v,m}\|_{L^{1}}=c_{m}$
			is independent of $v$. 
			If $m>n$,
			\begin{equation}
				\eta_{v,m}*|f|(x)\lesssim\mathcal{M}(f)(x).\label{eq:eta le HL}
			\end{equation}
			Together Lemmas \ref{lem:hardy} and \ref{lem:hardy} and inequality (\ref{eq:eta le HL}),
			we have the following result.
			\begin{cor}
				\label{cor:eta eq Hl}Let $1<p<t<r<\infty$ or $1<p\le t<r=\infty$.
				Let $m>n$ and  $1<q\le\infty$. Then $\|\{\eta_{v,m}*f_v\}^\infty_{v=0}\|_{M_{p}^{t,r}(\ell^q)}\lesssim\|\{f_v\}^\infty_{v=0}\|_{M_{p}^{t,r}(\ell^q)}.$
			\end{cor}

			\begin{lem}
				\label{lem:r-trick}
				Let $A\in(0,1]$, $v\in\mathbb{N}_{0}$,
				$W\in  \mathcal  A_{p}$ and $\{A_{Q}\}_{Q\in\mathcal{D}}$ be a sequence of
				reducing operators of order $p$ for $W$. Let  $\Delta$  be the same as in Lemma \ref{AQ AR improved}. For any $R>0$, there exists a constant $c>0$ such
				that for all $\vec{f}\in\mathscr{S}'(\mathbb{R}^n)$ with supp $ \mathcal{F}\vec{f}\subset\{\xi \in \rn :|\xi|\le2^{v+1}\}$,
				we have
				\[
				\bigg(\sum_{Q\in\mathcal{D}_{v}}|A_{Q}\vec{f}(x)|\chi_{Q}(x)\bigg)^{A}\le c\eta_{v,A(R-\Delta)}*\bigg(\sum_{Q\in\mathcal{D}_{v}}|A_{Q}\vec{f}|^{A}\chi_{Q}\bigg).
				\]
				Furthermore, if $R>\Delta+n/A$,
				\[
				\sum_{Q\in\mathcal{D}_{v}}|A_{Q}\vec{f}(x)|\chi_{Q}(x)\le c\mathcal{M}\bigg(\sum_{Q\in\mathcal{D}_{v}}|A_{Q}\vec{f}|^{A}\chi_{Q}\bigg)^{1/A}(x).
				\]
			\end{lem}
			
			\begin{proof}
				Let $x\in Q_{v,k}$ for some $k\in\mathbb{Z}^{n}$. From Lemma \ref{lem:A inequality},
				we have
				\begin{equation*}
					\bigg(\sum_{Q\in\mathcal{D}_{v}}|A_{Q}\vec{f}(x)|\chi_{Q}(x)\bigg)^{A} 
					 \lesssim \sum_{l\in\mathbb{Z}^{n}}(1+|k-l|)^{-A(R-\Delta)}2^{vn}\int_{Q_{vl}}|A_{Q_{vl}}\vec{f}(z)|^{A}{\rm d}z.
				\end{equation*}
				For $x\in Q_{v,k}$, $z\in Q_{v,k+l}$, we have $(1+2^{v}|x-z|)\approx1+|l|$. From this, we
				conclude 
				\begin{equation*}
				\Big(\sum_{Q\in\mathcal{D}_{v}}|A_{Q}\vec{f}(x)|\chi_{Q}(x)\Big)^{A}  \lesssim \int_{\mathbb{R}^{n}}  
				\frac{2^{vn} }{ (1+2^{v}|x-z|)^{A(R-\Delta)}  }
				\Big(\sum_{Q\in\mathcal{D}_{v}}|A_{Q}\vec{f}|^{A}\chi_{Q}\Big){\rm d}z.
				\end{equation*}
				If we choose $R>\Delta +n/A$, then by Lemma \ref{lem:Littlewood max},
				we obtain
				\[
				\sum_{Q\in\mathcal{D}_{v}}|A_{Q}\vec{f}(x)|\chi_{Q}(x)\lesssim\mathcal{M}\bigg(\sum_{Q\in\mathcal{D}_{v}}|A_{Q}\vec{f}|^{A}\chi_{Q}\bigg)^{1/A}(x).\qedhere
				\]
			\end{proof}

			\begin{lem}[Lemma 3.1, \cite{DH19}]
				\label{lem: hardy type inequality} Let $\alpha>0$
				and $0<q\le\infty$. Then there exists a constant $c$ such that for
				any sequence of positive numbers $\{x_{k}\}_{k}\in\ell^{q}$
				\[
				\|\{\delta_{k}\}_{k}\|_{\ell^{q}}\le c\|\{x_{k}\}_{k}\|_{\ell^{q}}
				\]
				holds, where
				\[
				\delta_{k}:=\sum_{j=0}^{\infty}2^{-\alpha|k-j|}x_{j},\ k\in\mathbb{N}_{0}.
				\]
			\end{lem}

			\begin{lem}
				\label{pseudo poi}
				Let $p\in(0,\infty)$,
				$W\in  \mathcal  A_{p}$ and $A\in(0,1]$.
				Let  $\Delta$  be the same as in Lemma \ref{AQ AR improved}.
				Let $a:\mathbb{R}^{2n}\to\mathbb{C}$ be a bounded and measurable
				symbol such that $\mathrm{supp}$ $a(x,\cdot)\subset\{\xi \in \rn :|\xi|\le c2^{k}\}$
				for some $c>0$. Suppose that ${\rm supp}$ $\mathcal{F}\vec{f}\subset\{\xi \in \rn :|\xi|\le c2^{k}\}$.
				Let $\{A_{Q}\}_{Q\in\mathcal{D}}$ be a sequence of reducing operators
				of order $p$ for $W$. Then for $R>0$ we have
				\begin{align*}
				&	\bigg|\sum_{l\in\mathbb{Z}^{n}}A_{Q_{k,l}}a(x,D)\vec{f}(x)\chi_{Q_{k,l}}(x)\bigg|  \\
				& 	\lesssim \|a(x,2^{k}\cdot)\|_{B_{1,A}^{R}}  \left(  \eta_{k, R -  A \Delta } *  \left(  \sum_{Q \in \mathcal D _k}  |A_Q \vec  f  |^A \chi_Q  \right)  (x) \right)^{1/A} 
				\end{align*}
				for each $x\in\mathbb{R}^{n}$, where the implicit constant is independent
				of $f$ and $x$.
			\end{lem}
			
			\begin{proof}
				Let $\mathcal{F}\psi\in\mathscr{S}(\mathbb{R}^n)$ be supported in $\{\xi \in \rn :1/2\le|\xi|\le2\}$
				and such that $\sum_{v\in\mathbb{Z}}\mathcal{F}\psi(2^{-v}\xi)=1$
				for all $\xi\neq0$. For $v\in\mathbb{Z}$, let $\psi_{v}(x)=2^{vn}\psi(2^{v}x)$.
				Then $\mathcal{F}\psi_{v}(\xi)=\mathcal{F}\psi(2^{-v}\xi)$.
				Let
				\[
				K(x,x-y):=\int_{\mathbb{R}^{n}}a(x,\xi)e^{2\pi i(x-y)\cdot\xi}\mathrm{d}\xi
				=\mathcal{F}_{\xi}(a(x,\cdot))(y-x)
				\]
				be the kernel of $a(\cdot,D)$. Observe that
				\begin{align*}
					K(x,x-y)\mathcal{F}\psi_{v}(y-x) & = \left(  \mathcal{F}_{\xi}(a(x,\cdot))\mathcal{F}\psi_{v}   \right)(y-x)
					\\
					&= 
					\mathcal{F}_{\xi} ^{-1} \mathcal{F_{\xi}}^{-1}(\mathcal{F}\varphi_{v}\mathcal{F}_{\xi}a(x,\cdot))(x-y).
				\end{align*}
				Therefore
				\begin{align*}
					|K(x,x-y)\mathcal{F}\psi_{v}(y-x)| & \le\|\mathcal{F_{\xi}}^{-1}(\mathcal{F}\varphi_{v}\mathcal{F}_{\xi}a(x,\cdot))\|_{L^{1}}\\
					& =2^{kn}\|\mathcal{F}_{\xi}^{-1}(\mathcal{F}\varphi_{v+k}\mathcal{F}_{\xi}a(x,2^{k}\cdot))\|_{L^{1}}.
				\end{align*}
				Applying the Plancherel-{P\^{o}lya}-Nikol'skii inequality (see page 18 in
				\cite{T83}), we obtain
				\begin{align*}
					& |A_{Q_{k,l}}a(x,D)\vec{f}(x)|  \lesssim2^{kn(1/A-1)}\Big(\int_{\mathbb{R}^{n}}|A_{Q_{k,l}}K(x,x-y)\vec{f}(y)|^{A}\mathrm{d}y\Big)^{1/A}       \\
					& = 2^{kn(1/A-1)}\Big(  \sum_{v \in \mathbb Z } \int_{\mathbb{R}^{n}}|A_{Q_{k,l}}K(x,y-x)  \mathcal{F}\psi_{v} (y-x) \vec{f}(y)|^{A}\mathrm{d}y\Big)^{1/A}   
				\end{align*}
				for $x\in Q_{k,l}$. Raising the power $A$, we have
				\[
				|A_{Q_{k,l}}a(x,D)\vec{f}(x)|{^A}  \lesssim2^{kn(1-A)}(S_{1}+S_{2}),
				\]
				where
				\[
				S_{1}:=\sum_{v=-\infty}^{-k-1}\sup_{y}|K(x,x-y)\mathcal{F}\varphi_{v}(y-x)|^{A}\int_{\text{\ensuremath{2^{v-1}\le}}|x-y|\le2^{v+1}}|A_{Q_{k,l}}\vec{f}(y)|^{A}\mathrm{d}y,
				\]
				\[
				S_{2}:=\sum_{v=-k}^{\infty}\sup_{y}|K(x,x-y)\mathcal{F}\varphi_{v}(y-x)|^{A}\int_{\text{\ensuremath{2^{v-1}\le}}|x-y|\le2^{v+1}}|A_{Q_{k,l}}\vec{f}(y)|^{A}\mathrm{d}y.
				\]
				We first estimate $2^{kn(1-A)} S_2$. We first consider
				\begin{align*}
					&	\int_{\text{\ensuremath{2^{v-1}\le}}|x-y|\le2^{v+1}}|A_{Q_{k,l}}\vec{f}(y)|^{A}\mathrm{d}y \\
					& \lesssim \sum_{w \in \mathbb Z ^n, |l-w | \le 2^{v+k+1}}  \int_{Q_{k,w}} (1+|l-w|)^{A \Delta } |A_{Q_{k,w}} \vec f(y) |^A \d y \\
					& \approx \int_{B (x, 2^{v+1})} (1+ 2^k |x-y|) ^{A \Delta } \left(  \sum_{Q \in \mathcal D _k}  |A_Q \vec  f (y) |^A \chi_Q (y) \right) \d y \\
					& \lesssim 2^{(k+v)R }\int_{B (x, 2^{v+1})} (1+ 2^k |x-y|) ^{A \Delta-R} \left(  \sum_{Q \in \mathcal D _k}  |A_Q \vec  f (y) |^A \chi_Q (y) \right) \d y \\
					& \lesssim 2^{(k+v)R } 2^{-kn} \eta_{k, R -A \Delta  } * \left(  \sum_{Q \in \mathcal D _k}  |A_Q \vec  f  |^A \chi_Q  \right) (x) .
				\end{align*}
				Then 
				\begin{align*}
					2^{kn(1-A)}	S_2 	& \le 2^{kn(1-A)} 2^{knA}\sum_{v=-k}^{\infty}\|\mathcal{F}_{\xi}^{-1}(\mathcal{F}\varphi_{v+k}\mathcal{F}_{\xi}a(x,2^{k}\cdot))\|_{L^{1}}^{A} \\
					&  \quad \times2^{(k+v)R } 2^{-kn} \eta_{k, R -A \Delta } * \left(  \sum_{Q \in \mathcal D _k}  |A_Q \vec  f  |^A \chi_Q  \right) (x)\\
					& = \sum_{v=-k}^{\infty} 2^{(k+v)R } \|\mathcal{F}_{\xi}^{-1}(\mathcal{F}\varphi_{v+k}\mathcal{F}_{\xi}a(x,2^{k}\cdot))\|_{L^{1}}^{A} \\
					& \quad \times \eta_{k, R -A \Delta } * \left(  \sum_{Q \in \mathcal D _k}  |A_Q \vec  f  |^A \chi_Q  \right) (x) \\
					& \le \|a(x,2^{k}\cdot)\|_{B_{1,A}^{R}}^{A} \eta_{k, R -A \Delta } * \left(  \sum_{Q \in \mathcal D _k}  |A_Q \vec  f  |^A \chi_Q  \right) (x) .
				\end{align*}
				Next we estimate $2^{kn(1-A)}S_{1}$.
				Since supp $\mathcal{F}\vec{f}\subset\{\xi:|\xi|\le c2^{k}\}$, from Lemma \ref{lem:A inequality} (We can use this lemma. The key of proof of Lemma \ref{lem:A inequality} is using the compact support of Fourier transform of $\varphi_j * \vec f$), we obtain for any $N > 0$
				\begin{align*}
					&	\int_{\text{\ensuremath{2^{v-1}\le}}|x-y|\le2^{v+1}}|A_{Q_{k,l}}\vec{f}(y)|^{A}\mathrm{d}y \\
					& \lesssim 	2^{vn}  \sum_{w\in\mathbb{Z}^{n}}(1+|w-l|)^{-A(N-\Delta)}2^{kn}\int_{Q_{k,w}}|A_{Q_{k,w}} \vec{f}(z)|^{A}\mathrm{d}z \\
					& \approx 2^{vn}  \eta_{k, A(N -\Delta) } *  \left(  \sum_{Q \in \mathcal D _k}  |A_Q \vec  f  |^A \chi_Q  \right) (x) .
				\end{align*}
				Let $N = R /A$, then 
				\begin{equation*}
					\int_{\text{\ensuremath{2^{v-1}\le}}|x-y|\le2^{v+1}}|A_{Q_{k,l}}\vec{f}(y)|^{A}\mathrm{d}y  \lesssim  2^{vn}  \eta_{k, R - A \Delta } *  \left(  \sum_{Q \in \mathcal D _k}  |A_Q \vec  f  |^A \chi_Q  \right) (x) .
				\end{equation*}
				Note that
				\begin{align*}
					2^{kn}\|\mathcal{F}_{\xi}^{-1}(\mathcal{F}\varphi_{v+k}\mathcal{F}_{\xi}a(x,2^{k}\cdot))\|_{L^{1}} & \le 2^{kn} \|\varphi_{v+k} \|_{L^1} \| \|a(x,2^{k}\cdot))\|_{L^{1}}  \\
					& =  2^{kn} \| \varphi \|_{L^1}  \| \|a(x,2^{k}\cdot))\|_{L^{1}}  .
				\end{align*}
				Hence $2^{kn(1-A)}S_{1}$ is bounded by
				\begin{align*}
					& 2^{kn(1-A)} \sum_{v=-\infty}^{-k-1} 2^{knA}  \| \varphi \|_{L^1} ^A   \| \|a(x,2^{k}\cdot))\|_{L^{1}} ^A  2^{vn}  \eta_{k, R - A \Delta } *  \left(  \sum_{Q \in \mathcal D _k}  |A_Q \vec  f  |^A \chi_Q  \right) (x) \\
					&\lesssim 2^{kn} \|a(x,2^{k}\cdot))\|_{L^{1}} ^A \eta_{k, R - A \Delta } *  \left(  \sum_{Q \in \mathcal D _k}  |A_Q \vec  f  |^A \chi_Q  \right) (x) \sum_{v=-\infty}^{-k-1}2^{vn} 
			\\
					& \lesssim \|a(x,2^{k}\cdot)\|_{B_{1,A}^{R}}^A \eta_{k, R - A \Delta } *  \left(  \sum_{Q \in \mathcal D _k}  |A_Q \vec  f  |^A \chi_Q  \right) (x) .
				\end{align*}
				Hence we obtain
				\begin{equation*}
					|A_{Q_{k,l}}a(x,D)\vec{f}(x)| \lesssim  \|a(x,2^{k}\cdot)\|_{B_{1,A}^{R}}  \left(  \eta_{k, R -  A \Delta } *  \left(  \sum_{Q \in \mathcal D _k}  |A_Q \vec  f  |^A \chi_Q  \right)  (x)  \right)^{1/A} 
				\end{equation*}
				and the proof is finished.
			\end{proof}

				\begin{lem}
					\label{lem:ring}Let $0<p<t<r<\infty$ or $0<p\le t<r=\infty$.
					Let
					$B_0,B_1,B_2>0$.
					Let $s\in\mathbb{R}$, $W\in  \mathcal  A_{p}$ and
					$\{A_{Q}\}_{Q\in\mathcal{D}}$ be a sequence of reducing operators of order p for $W$. 
					Let $ q \in (0, p + \delta_W) $ where $\delta_W >0$  is the same as in Lemma \ref{lem:W AQ}. 
					Then there exists a constant
					$c>0$ such that
					\begin{equation*}
						\bigg\|\sum_{k=0}^{\infty}\vec{f}_{k}\bigg\|_{F_{p,t,r}^{s,q}(A_{Q})}\le c\|\{2^{ks}\vec{f}_{k}\}_{k \in \mathbb N_0}\|_{M_{p}^{t,r}(A_{Q})(\ell^{q})}
					\end{equation*}
					for any sequence of functions $\{\vec{f}_{k}\}_{k\in\mathbb{N}_{0}}$
					such that
					\[
					\mathrm{supp}\;\mathcal{F}\vec{f}_{0}\subset\{\xi:|\xi|\le B_0\}
					\]
					and
					\[
					\mathrm{supp}\;\mathcal{F}\vec{f}_{k}\subset\{\xi:B_12^{k+1}\le|\xi|\le B_22^{k+1}\},\;\mathrm{for}\;k\ge1.
					\]
				\end{lem}
				
				\begin{proof}
					Let  $\Delta$  be the same as in Lemma \ref{AQ AR improved}.
					Let $\{\varphi_{j}\}_{j\in\mathbb{N}_{0}}$ be a resolution
					of unity as in Definition \ref{def:partition of unity inho}. In the
					view of the support properties of $\mathcal{F}\vec{f}_{k}$ and $\varphi_{j}$,
					then
					\[
					\varphi_{j}(D)\sum_{k=0}^{\infty}\vec{f}_{k}=\sum_{l=-\kappa_{1}}^{\kappa_{2}}\varphi_{j}(D)\vec{f}_{j+l}
					\]
					for some $\kappa_{1},\kappa_{2}\in\mathbb{N}_{0}.$ We suppose simply
					$s=0$.  By Lemma \ref{lem:r-trick}, for $l=-\kappa_{1},\dots,\kappa_{2}$,
					any $A\in(0,1]$, any $R>0$, we have
					\begin{align*}
						& \bigg(\sum_{Q\in\mathcal{D}_{j}}|A_{Q} \varphi_{j}(D)\vec{f}_{j+l} (x)|\chi_{Q}(x)\bigg)^{A}  \\
						& 	\le c \eta_{j,A(R-\Delta )}*\bigg(\sum_{Q\in\mathcal{D}_{j}}|A_{Q}\varphi_{j}(D)\vec{f}_{j+l}|^{A}\chi_{Q}\bigg)(x).
					\end{align*}
					Therefore, with $A\in(0,\min(1,p,q))$, $A(R-\Delta)>n$ and by
					Corollary \ref{cor:eta eq Hl}, we obtain
					\begin{align*}
						& \bigg\|\bigg(\sum_{j=0}^{\infty}\sum_{Q\in\mathcal{D}_{j}}\bigg|A_{Q}\varphi_{j}(D)\sum_{k=0}^{\infty}\vec{f}_{k}\bigg|^{q}\bigg)^{1/q}\bigg\|_{M_{p}^{t,r}}\\
						& =\bigg\|\bigg(\sum_{j=0}^{\infty}\sum_{Q\in\mathcal{D}_{j}}\bigg|\sum_{l=-\kappa_{1}}^{\kappa_{2}}A_{Q}\varphi_{j}(D)\vec{f}_{j+l}
						\bigg|^{q}\bigg)^{1/q}\bigg\|_{M_{p}^{t,r}}\\
						& \lesssim\sum_{l=-\kappa_{1}}^{\kappa_{2}}\bigg\|\bigg(\sum_{j=0}^{\infty}\sum_{Q\in\mathcal{D}_{j}}\big|A_{Q}\varphi_{j}(D)
						\vec{f}_{j+l}\big|^{q}\bigg)^{1/q}\Big\|_{M_{p}^{t,r}}
						\\
						& \lesssim \sum_{l=-\kappa_{1}}^{\kappa_{2}}\bigg\|\bigg(\sum_{j=0}^{\infty}\bigg( 
						\eta_{j ,A(R-\Delta )}*\bigg(\sum_{Q\in\mathcal{D}_{j}}|A_{Q}\varphi_{j}(D)\vec{f}_{j+l}|^{A}\chi_{Q}\bigg)
						\bigg)^{q/A}\bigg)^{1/q}\bigg\|_{M_{p}^{t,r}}\\
						& \lesssim  \sum_{l=-\kappa_{1}}^{\kappa_{2}}\bigg\|\bigg(\sum_{j=0}^{\infty}\bigg( 
						\sum_{Q\in\mathcal{D}_{j}}|A_{Q}\varphi_{j}(D)\vec{f}_{j+l}|^{A}\chi_{Q} 
						\bigg)^{q/A}\bigg)^{1/q}\bigg\|_{M_{p}^{t,r}}\\
						& \lesssim\|\{\vec{f}_{k}\}_{k=0}^{\infty}\|_{M_{p}^{t,r}(A_{Q})(\ell^{q})},
					\end{align*}
					where in the last step we use Theorem \ref{multiplier} and the fact that $M_{p}^{t,r}(A_{Q})(\ell^{q})$ is invariant by shifting finite levels. 
					Hence, the proof is finished.
				\end{proof}
				\begin{lem}
					\label{lem:sum k F le } Let $0<p<t<r<\infty$ or $0<p\le t<r=\infty$.
					Let  $W\in  \mathcal  A_{p}$ and
					$\{A_{Q}\}_{Q\in\mathcal{D}}$ be a sequence of reducing operators of order p for $W$. 
					Let $\tilde d $  be the same as in Lemma \ref{AQ AR improved}.
					Let $ q \in (0, p + \delta_W) $ where $\delta_W >0$  is the same as in Lemma \ref{lem:W AQ}. 
					If $s> \tilde d / p'$, then there exists a constant
					$c$ such that
					\begin{equation*}
						\bigg\|\sum_{k=0}^{\infty}\vec{f}_{k}\bigg\|_{F_{p,t,r}^{s,q}(A_{Q})}\le c\|\{2^{ks}\vec{f}_{k}\}_{k  \in \mathbb N_0}\|_{M_{p}^{t,r}(A_{Q})(\ell^{q})}
					\end{equation*}
					for any sequence of functions $\{\vec{f}_{k}\}_{k\in\mathbb{N}_{0}}$
					such that ${\rm supp}$ $\mathcal{F}\vec{f}_{k}\subset\{\xi:|\xi|\le C2^{k+1}\}$.
				\end{lem}
				
				\begin{proof}
					Let $\{\varphi_{j}\}_{j\in\mathbb{N}_{0}}$ be a resolution of unity
					as in Definition \ref{def:partition of unity inho}. Using the support
					properties of $\{\vec{f}_{k}\}$, we have
					\[
					\sum_{k=0}^{\infty}\varphi_{j}(D)\vec{f}_{k}=\sum_{k=j+\sigma}^{\infty}\varphi_{j}(D)\vec{f}_{k}=\sum_{i=\sigma}^{\infty}\varphi_{j}(D)\vec{f}_{j+i},
					\]
					for some $\sigma\in\mathbb{Z}$. Let $J:= \min(1,p,q)$. Then
					\begin{align*}
						\bigg\|\sum_{k=0}^{\infty}\vec{f}_{k}\bigg\|_{F_{p,t,r}^{s,q}(A_{Q})} & =\bigg\|\bigg\{2^{js}\sum_{i=\sigma}^{\infty}\sum_{Q\in\mathcal{D}_{j}}|A_{Q}\varphi_{j}(D)\vec{f}_{j+i}|\chi_{Q}\bigg\}_{j\in\mathbb{N}_{0}}\bigg\|_{M_{p}^{t,r}(\ell^{q})}\\
						& \le\bigg\{\sum_{i=\sigma}^{\infty}\bigg\|\bigg\{2^{js}\sum_{Q\in\mathcal{D}_{j}}|A_{Q}\varphi_{j}(D)\vec{f}_{j+i}|\chi_{Q}\bigg\}_{j\in\mathbb{N}_{0}}\bigg\|_{M_{p}^{t,r}(\ell^{q})}^{J}\bigg\}^{\frac{1}{J}}
					\\
						& \le\bigg\{\sum_{i=\sigma}^{-1}\bigg\|\bigg\{2^{js}\sum_{Q\in\mathcal{D}_{j}}|A_{Q}\varphi_{j}(D)\vec{f}_{j+i}|\chi_{Q}\bigg\}_{j\in\mathbb{N}_{0}}\bigg\|_{M_{p}^{t,r}(\ell^{q})}^{J}\\
						& \quad +\sum_{i=0}^{\infty}\bigg\|\bigg\{2^{js}\sum_{Q\in\mathcal{D}_{j}}|A_{Q}
						\varphi_{j}(D)\vec{f}_{j+i}|\chi_{Q}\bigg\}_{j\in\mathbb{N}_{0}}\bigg\|_{M_{p}^{t,r}(\ell^{q})}^{J}\bigg\}^{\frac{1}{J}}.
					\end{align*}
					We just consider the second sum in the last inequality since the first
					sum has finite items. From Lemmas \ref{pseudo poi} and \ref{AQ AR improved}, we obtain
					\begin{align*}
						& \sum_{Q\in\mathcal{D}_{j}}|A_{Q}\varphi_{j}*\vec{f}_{j+i}(x)|\chi_{Q}(x)\nonumber \\
						& \lesssim\|\mathcal{F}\varphi_{1}\|_{{B}_{1,A}^{R}}\bigg\{\eta_{j,R-A\Delta }*\bigg|\sum_{Q\in\mathcal{D}_{j}}A_{Q}\vec{f}_{j+i}\bigg|^{A}(x)\bigg\}^{1/A}\nonumber \\
						& \lesssim \bigg\{\sum_{Q\in\mathcal{D}_{j}}\int_{Q}2^{jn}\frac{1}{(1+2^{j}|x-y|)^{R-A\Delta }}|A_{Q}\vec{f}_{j+i}(y)|^{A}\mathrm{d}y\bigg\}^{1/A}\nonumber \\
						& \lesssim \bigg\{\sum_{Q\in\mathcal{D}_{j}} \sum_{P\in\mathcal{D}_{j+i} , P \subset Q }\int_{P}2^{jn}\frac{1}{(1+2^{j}|x-y|)^{R-A\Delta }} 2^{i A \tilde d / p'} |A_{P}\vec{f}_{j+i}(y)|^{A}\mathrm{d}y\bigg\}^{1/A}\nonumber \\
						& \lesssim 2^{i  \tilde d / p'}  \mathcal{M}\bigg(\bigg|\sum_{Q\in\mathcal{D}_{j+i}}A_{Q}\vec{f}_{j+i}\bigg|^{A}\bigg)^{1/A}(x). 
					\end{align*}
					Let $0<A<J$. By Lemma \ref{lem:hardy}, and $s> \tilde d / p'$,
					\begin{align*}
						& \sum_{i=0}^{\infty}\bigg\|\bigg\{2^{js}\sum_{Q\in\mathcal{D}_{j}}|A_{Q}\varphi_{j}(D)\vec{f}_{j+i}|\chi_{Q}\bigg\}_{j\in\mathbb{N}_{0}}\bigg\|_{M_{p}^{t,r}(\ell^{q})}^{J}\\
						& \lesssim\sum_{i=0}^{\infty}\bigg\|\bigg\{2^{js}2^{i \tilde d / p'}\mathcal{M}\bigg(\bigg|\sum_{Q\in\mathcal{D}_{j+i}}A_{Q}\vec{f}_{j+i}\bigg|^{A}\bigg)^{1/A}\bigg\}_{j\in\mathbb{N}_{0}}\bigg\|_{M_{p}^{t,r}(\ell^{q})}^{J}\\
						& \lesssim\sum_{i=0}^{\infty}2^{-i(s- \tilde d / p' )J}\bigg\|\bigg\{2^{(j+i)s}\sum_{Q\in\mathcal{D}_{j+i}}A_{Q}\vec{f}_{j+i}\bigg\}_{j\in\mathbb{N}_{0}}\bigg\|_{M_{p}^{t,r}(\ell^{q})}^{J}\\
						& \lesssim\Big\|\big\{2^{js}\vec{f}_{j}\big\}_{j\in\mathbb{N}_{0}}\Big\|_{M_{p}^{t,r}(A_{Q})(\ell^{q})}^{J}.
					\end{align*}
					Hence, the proof is finished.
				\end{proof}

				We follow the definition of \cite{B88} to define $C_{*}^{\ell}S_{1,\delta}^{m}$
				where $C_{*}^{\ell}$ is the H\"{o}lder-Zygmund spaces which equals the
				homogeneous Besov space $\dot{B}_{\infty}^{\ell,\infty}$.
				\begin{defn}
					Let $0\le\delta\le1$, $\ell>0$, $m\in\mathbb{R}$. We define
					\[
					C_{*}^{\ell}S_{1,\delta}^{m}:=\big\{\sigma:\mathbb{R}^{2n}\to\mathbb{C}:\sigma(x,\cdot)\in C^{\infty},x\in\mathbb{R}^{n}\;\mathrm{and}\;\|\sigma\|_{C_{*}^{\ell}S_{1,\delta}^{m}}<\infty\big\},
					\]
					where
					\begin{equation*}
						\|\sigma\|_{C_{*}^{\ell}S_{1,\delta}^{m}}:=\sup_{\xi\in\mathbb{R}^{n}}\langle\xi\rangle^{-m+|\alpha|-l\delta}\|\partial_{\xi}^{\alpha}\sigma(\cdot,\xi)\|_{C_{*}^{l}}+\sup_{\xi\in\mathbb{R}^{n}}\langle\xi\rangle^{-m+|\alpha|}\|\partial_{\xi}^{\alpha}\sigma(\cdot,\xi)\|_{L^{\infty}},
					\end{equation*}
				and $\langle\xi\rangle : = (1+|\xi|^{2})^{1/2}$.
					
					\begin{rem}
						Let $\ell>0$, $0<\delta\le1$, $m>0$ and $N \gg 1$. Then 	from \cite[Exercise 5.25]{S18}, we have $S^N_{1,0} \subset C_{*}^{\ell}S_{1,\delta}^{m}$.
					\end{rem}

					A pseudo-differential operator on $\mathscr{S}(\mathbb{R}^n)$ with symbol $\sigma\in C_{*}^{\ell}S_{1,\delta}^{m}$ is defined by
					\[
					\sigma(x,D)f(x)=\int_{\mathbb{R}^{n}}\sigma(x,\xi)\mathcal{F}f(\xi)e^{2\pi ix\cdot\xi}\mathrm{d}\xi
					\]
					for $f\in \mathscr{S}(\mathbb{R}^n)$.
					
					For boundedness of non-regular pseudo-differential operators on	Lebesgue spaces, Besov spaces, Triebel-Lizorkin spaces and Sobolev spaces, we refer to \cite{B88}, \cite{CM78}, \cite{Mar88}, \cite{Mar96}.
					In \cite{S09}, Sawano obtained the boundedness of pseudo-differential operators   with Symbol belonging to $S^m_{1,\delta}$ and $C_{*}^{\ell}S_{1,\delta}^{m}$,  $\delta\in[0,1]$ on the Besov-Morrey spaces $\mathcal{N}_{pqr}^s$ and the Triebel-Lizorkin-Morrey spaces $\mathcal{E}_{pqr}^s$. In \cite{CO22}, Congo and Ouedraogo proved the boundedness of pseudo-differential operators   with Symbol belonging to  $C_{*}^{\ell}S_{1,\delta}^{m}$,  $\delta\in[0,1]$ on variable exponent  Triebel-Lizorkin-Morrey spaces $\mathcal{E}_{p(\cdot),u(\cdot), q(\cdot)}^{s(\cdot)}$.

					We call an elementary symbol $\sigma$ in the class $C_{*}^{\ell}S_{1,\delta}^{m}$
					if the symbol $\sigma$ has an expression of the form
					\[
					\sigma(x,\xi)=\sum_{j\ge1}^{\infty}\sigma_{j}(x)\psi_{j}(\xi)
					\]
					where $\psi_{j}(\xi)=\psi_{1}(2^{-j+1}\xi)$ and $\psi_{1}\in C^{\infty}$
					is supported on the ring $\{\xi \in \rn :1\le|\xi|\le2^{2}\}$, while $\sigma_{j}$
					is uniformly bounded sequence such that
					\[
					\sup_{j\in\mathbb{N}_{0}}2^{-jm}\|\sigma_{j}\|_{L^{\infty}}<\infty\quad\text{and}\quad\sup_{j\in\mathbb{N}_{0}}2^{-j(m+l\delta)}\|\sigma_{j}\|_{C_{*}^{\ell}}<\infty.
					\]
				\end{defn}
				
				\begin{lem}
					\cite{M03} Suppose that $f=\sum_{j\in\mathbb{N}_{0}}f_{j}$ in $\mathscr{S}'(\mathbb{R}^{n}),$
					with $\mathrm{supp}(\mathcal{F}f_{j})\subset\{\xi:|\xi|\le C2^{j}\}$ for
					some $C>0$. Then there is a constant $c$ such that for $l>0$,
					\[
					\|f\|_{C_{*}^{l}}\le c\sup_{j\in\mathbb{N}_{0}}\big\{2^{jl}\|f_{j}\|_{L^{\infty}}\big\}.
					\]
				\end{lem}
				
				\begin{thm}\label{non-reg}
					Let $0<p<t<r<\infty$ or $0<p\le t<r=\infty$. 
					Let  $W\in  \mathcal  A_{p}$ with the $\mathcal A_p$-dimension $d\in [0,n)$. 
					Let $\tilde d $  be the same as in Lemma \ref{AQ AR improved}.
					Let $\delta_W >0$  be the same as in Lemma \ref{lem:W AQ}. Let $ q \in (0, p + \delta_W)  $ and
					$s> \tilde d / p'$.
					Let $\sigma\in C_{*}^{\ell}S_{1,\delta}^{m}$ with
					$m\in\mathbb{R},$ $\delta\in[0,1]$ and $l-d/p>s$. Then
					$\sigma(\cdot,D)$ is a bounded operator from $F_{p,t,r}^{s+m,q}(W)$
					to $F_{p,t,r}^{s,q}(W)$.

				\end{thm}
				
				\begin{proof}
					Since any symbol in $C_{*}^{\ell}S_{1,\delta}^{m}$
					can be approximated by elementary symbols, it is enough to examine
					the case when $\sigma$ is an elementary symbol while $\sigma(x,\xi)=\sum_{j\ge1}^{\infty}\sigma_{j}(x)\varphi_{j}(\xi)$.
					Define $\sigma_{j,k}=\varphi_{k}(D)\sigma_{j}$ on $\rn $. Then
					\begin{align*}
						\|\sigma_{j,k}\|_{C_{*}^{l}} & =\sup_{x\in\mathbb{R}^{n}}\sup_{v\in\mathbb{N}_{0}}2^{vl}\varphi_{v}(D)\varphi_{k}(D)\sigma_{j}(x)\\
						& \approx\sup_{x\in\mathbb{R}^{n}}2^{kl}\varphi_{k}(D)\sigma_{j}(x)
					\end{align*}
					and we have
					\begin{equation}
						\|\sigma_{j,k}\|_{L^{\infty}}\lesssim2^{-kl}\|\sigma_{j,k}\|_{C_{*}^{l}}\lesssim2^{-kl}2^{j(m+l\delta)}\lesssim2^{-kl+j(m+l)}.\label{eq:sigma jk infty}
					\end{equation}
					Now we decompose the symbol into three parts:
					\begin{align*}
						\sigma(x,\xi) & =\sum_{j\ge1}^{\infty}\sum_{k=0}^{\infty}\sigma_{j,k}(x)\varphi_{j}(\xi) 
					 =\sum^3_{i=1}\sigma^{(i)}(x,\xi),
					\end{align*}
					where
					\[
					\sigma^{(1)}(x,\xi):=\sum_{j=4}^{\infty}\sum_{k=0}^{j-4}\sigma_{j,k}(x)\varphi_{j}(\xi),
					\]
					\[
					\sigma^{(2)}(x,\xi):=\sum_{j=0}^{\infty}\sum_{k=\max(j-3,0)}^{j+3}\sigma_{j,k}(x)\varphi_{j}(\xi),
					\]
					and
					\[
					\sigma^{(3)}(x,\xi):=\sum_{j=0}^{\infty}\sum_{k=j+4}^{j-4}\sigma_{j,k}(x)\varphi_{j}(\xi).
					\]
					It is understanding that $\sigma_{0,k}=0$. Let $\vec{f_{j}} =\varphi_{j}(D)\vec{f}$.					
					Let $\{A_{Q}\}_{Q \in \mathcal{D}}$ be a sequence of reducing operators of order $p$
					for $W$. Since $F_{p,t,r}^{s,q}(W)=F_{p,t,r}^{s,q}(A_{Q})$, it is
					enough to prove that
					\[
					\|\sigma(\cdot,D)\vec{f}\|_{F_{p,t,r}^{s,q}(A_{Q})}\lesssim\|\vec{f}\|_{F_{p,t,r}^{s+m,q}(A_{Q})}.
					\]
					
					Step 1. We will prove that
					\[
					\|\sigma^{(1)}(\cdot,D)\vec{f}\|_{F_{p,t,r}^{s,q}(A_{Q})}\lesssim\|\vec{f}\|_{F_{p,t,r}^{s+m,q}(A_{Q})}
					\]
					for $\vec{f}\in\mathscr{S}'(\mathbb{R}^{n})$. Indeed, $\sum_{k=0}^{j-4}\sigma_{j,k}(\cdot,D)\vec{f}_{j}$
					has its spectrum in $\{\xi:c_{1}2^{j}\le|\xi|\le c_{2}2^{j}\}$ where
					$c_{1},c_{2}>0$ are independent of $j.$ Then by Lemma \ref{lem:ring}
					we have
					\begin{align*}
						& \|\sigma^{(1)}(\cdot,D)\vec{f}\|_{F_{p,t,r}^{s,q}(A_{Q})} \\
						 & =\bigg\|\sum_{j=4}^{\infty}\sum_{k=0}^{j-4}\sigma_{j,k}\varphi_{j}(D)\vec{f}\bigg\|_{F_{p,t,r}^{s,q}(A_{Q})}						 
						\end{align*}
					\begin{align*}
						& \lesssim\bigg\|\bigg\{2^{js}\sum_{k=0}^{j-4}\sigma_{j,k}\varphi_{j}(D)\vec{f}\bigg\}_{j}\bigg\|_{M_{p}^{t,r}(A_{Q})(\ell^{q})}\\
						& \le\bigg\|\bigg\{\sum_{j=0}^{\infty}2^{jsq}\sum_{Q\in\mathcal{D}_{j}}\bigg|A_{Q}\sum_{k=0}^{j-4}\sigma_{j,k}\varphi_{j}(D)\vec{f}\bigg|^{q}\chi_{Q}\bigg\}^{1/q}\bigg\|_{M_{p}^{t,r}}\\
						& \lesssim\bigg\|\bigg\{\sum_{j=0}^{\infty}2^{j(s+m)q}\sum_{Q\in\mathcal{D}_{j}}|A_{Q}\varphi_{j}(D)\vec{f}|^{q}\chi_{Q}\bigg\}^{1/q}\bigg\|_{M_{p}^{t,r}}\\
						& =\|\vec{f}\|_{F_{p,t,r}^{s+m,q}(A_{Q})}.
					\end{align*}
					
					Step 2. We will show that
					\[
					\|\sigma^{(2)}(\cdot,D)\vec{f}\|_{F_{p,t,r}^{s,q}(A_{Q})}\lesssim\|\vec{f}\|_{F_{p,t,r}^{s+m,q}(A_{Q})}
					\]
					for $\vec{f}\in\mathscr{S}'(\mathbb{R}^{n})$. Since $\sum_{k=j-3}^{j+3}\sigma_{j,k}\vec{f}_{j}$
					has its spectrum in $\{\xi:|\xi|\le c_{2}2^{j}\}$ where $c_{2}>0$
					is independent of $j$, then by Lemma \ref{lem:sum k F le } and (\ref{eq:sigma jk infty}),
					(note that $\sigma_{j,k}$ is a function on $\mathbb{R}^{n}$, not
					on $\mathbb{R}^{2n}$),
					\begin{align*}
						\|\sigma^{(2)}(\cdot,D)\vec{f}\|_{F_{p,t,r}^{s,q}(A_{Q})} & \lesssim\bigg\|\bigg\{2^{js}\sum_{k=j-3}^{j+3}\sigma_{jk}\vec{f}_{j}\bigg\}_{j}\bigg\|_{M_{p}^{t,r}(A_{Q})(\ell^{q})}\\
						& \lesssim\bigg\|\bigg\{2^{js}\sum_{k=j-3}^{j+3}\|\sigma_{jk}\|_{L^{\infty}}\vec{f}_{j}\bigg\}_{j}\bigg\|_{M_{p}^{t,r}(A_{Q})(\ell^{q})} \\
						& \lesssim\bigg\|\bigg\{2^{js}\sum_{k=j-3}^{j+3}2^{-kl+j(m+l)}\vec{f}_{j}\bigg\}_{j}\bigg\|_{M_{p}^{t,r}(A_{Q})(\ell^{q})}\\
						& \approx\big\|\big\{2^{j(s+m)}\vec{f}_{j}\big\}_{j}\big\|_{M_{p}^{t,r}(A_{Q})(\ell^{q})}.
					\end{align*}
					
					Step 3. In $\mathscr{S}'(\mathbb{R}^{n})$, we can write
					\[
					\sum_{j=0}^{\infty}\sum_{k=j+4}^{j-4}\sigma_{j,k}(x)\vec{f}_{j}(x)=\sum_{k=4}^{\infty}\sum_{j=0}^{k-4}\sigma_{j,k}(x)\vec{f}_{j}(x).
					\]
					Thus by Lemmas \ref{lem:ring} and \ref{AQ AR improved},
					we have
					\begin{align*}
						\| & \sigma^{(3)}(\cdot,D)\vec{f}\|_{F_{p,t,r}^{s,q}(A_{Q})}\\
						& =\bigg\|\sum_{k=4}^{\infty}\bigg(\sum_{j=0}^{k-4}\sigma_{j,k}\vec{f}_{j}\bigg)\bigg\|_{F_{p,t,r}^{s,q}(A_{Q})}\\
						& \lesssim\bigg\|\bigg(\sum_{k=0}^{\infty}2^{ksq}\sum_{Q\in\mathcal{D}_{k}}\bigg|A_{Q}\sum_{j=0}^{k-4}\sigma_{j,k}\vec{f}_{j}\bigg|^{q}\chi_{Q}\bigg)^{1/q}\bigg\|_{M_{p}^{t,r}}\\
						& \lesssim\bigg\|\bigg(\sum_{k=0}^{\infty}2^{ksq} \bigg|\sum_{j=0}^{k-4}\sum_{Q\in\mathcal{D}_{k}}A_{Q}\sigma_{j,k}\vec{f}_{j} \bigg|^{q}\chi_{Q}\bigg)^{1/q}\bigg\|_{M_{p}^{t,r}}
							\end{align*}
					\begin{align*}
						& \lesssim\bigg\|\bigg(\sum_{k=0}^{\infty}2^{ksq} \bigg|\sum_{j=0}^{k-4}\sum_{Q\in\mathcal{D}_{k}}A_{Q}2^{-kl+j(m+l)}\vec{f}_{j} \bigg|^{q}\chi_{Q}\bigg)^{1/q}\bigg\|_{M_{p}^{t,r}}\\
						& =\bigg\|\bigg(\sum_{k=0}^{\infty} \bigg|\sum_{j=0}^{k-4}2^{(k-j)(s-l)}2^{j(s+m)}\sum_{Q\in\mathcal{D}_{k}}A_{Q}\vec{f}_{j}\bigg|^{q}\chi_{Q}\bigg)^{1/q}\bigg\|_{M_{p}^{t,r}}.
					\end{align*}
					Note that $j \le k-4$. By Lemma \ref{AQ AR improved}, for a fixed $x \in \rn$, there is only one cube $Q' \in \D_k $ such that $x\in Q$.  For this $Q'$, let $P' \in \D_j$ and $Q' \subset P' $. Then by Lemma \ref{AQ AR improved}, we get 
					\begin{align*}
						\sum_{Q\in\mathcal{D}_{k}} | A_{Q}\vec{f}_{j} (x)|^{q}\chi_{Q} (x)  & =| A_{Q'}\vec{f}_{j} (x)|^{q}\chi_{Q ' } (x)    \\
						& \le \|A_{Q'} A_{P'}^{-1} \|^{q}  | A_{P'}\vec{f}_{j} (x)|^{q}\chi_{Q ' } (x)  \chi_{P'} (x) \\
						& \lesssim 2^{(k-j) q d/p}| A_{P'}\vec{f}_{j} (x)|^{q}\chi_{Q ' } (x)  \chi_{P'} (x) \\
						& = 2^{(k-j) q d/p}   \sum _{P \in \D_j  } | A_{P}\vec{f}_{j} (x)|^{q}\chi_{Q ' } (x)  \chi_{P} (x) \\
						& \le  2^{(k-j) q d/p}   \sum _{P \in \D_j  } | A_{P}\vec{f}_{j} (x)|^{q}  \chi_{P} (x),
					\end{align*}
					where $d \in [0,n)$ is the $\mathcal A_p$-dimension of $W$. Using this estimate, we obtain
					\begin{align*}
						& 	\bigg\|\bigg(\sum_{k=0}^{\infty} \Big |\sum_{j=0}^{k-4}2^{(k-j)(s-l)}2^{j(s+m)}\sum_{Q\in\mathcal{D}_{k}}A_{Q}\vec{f}_{j}  \Big |^{q}\chi_{Q}\bigg)^{1/q}\bigg\|_{M_{p}^{t,r}} \\
						& 
						\lesssim\bigg\|\bigg(\sum_{k=0}^{\infty} \Big|\sum_{j=0}^{k-4}2^{(k-j)(s-l)}2^{j(s+m)}2^{(k-j)  d/p}  \sum_{Q\in\mathcal{D}_{j}}A_{Q}\vec{f}_{j}  \Big |^{q}\chi_{Q}\bigg)^{1/q}\bigg\|_{M_{p}^{t,r}}\\
						& \lesssim\bigg\|\bigg(\sum_{k=0}^{\infty}  \Big |\sum_{j=0}^{k-4}2^{(k-j)(s-l+d/p)}2^{j(s+m)}\sum_{Q\in\mathcal{D}_{j}}A_{Q}\vec{f}_{j}  \Big |^{q}\chi_{Q}\bigg)^{1/q}\bigg\|_{M_{p}^{t,r}}\\
						& \le\bigg\|\bigg(\sum_{k=0}^{\infty} \Big|\sum_{j=0}^{\infty}2^{|k-j|(s-l+d/p)}2^{j(s+m)}\sum_{Q\in\mathcal{D}_{j}}A_{Q}\vec{f}_{j}  \Big |^{q}\chi_{Q}\bigg)^{1/q}\bigg\|_{M_{p}^{t,r}}.
					\end{align*}
					Since $0<s<l-d/p<\infty$, by Lemma \ref{lem: hardy type inequality}, we have
					\begin{align*}
						\bigg\|\bigg( & \sum_{k=0}^{\infty}\bigg|\sum_{j=0}^{k-4}2^{|k-j|(s-l+n/p)}2^{j(s+m)}\sum_{Q\in\mathcal{D}_{j}}A_{Q}\vec{f}_{j}\bigg|^{q}\chi_{Q}\bigg)^{1/q}\bigg\|_{M_{p}^{t,r}}\\
						& \lesssim\bigg\|\bigg(\sum_{j=0}^{\infty}2^{j(s+m)q}\bigg|\sum_{Q\in\mathcal{D}_{j}}A_{Q}\vec{f}_{j}\bigg|^{q}\chi_{Q}\bigg)^{1/q}\bigg\|_{M_{p}^{t,r}}.
					\end{align*}
					Hence, we finish the proof.
				\end{proof}
				\begin{rem}
					Let $W \equiv 1 $. Let  $1\le p\le t <r=\infty$. ($M_p^{t,r} (W) = L^{p,p/t}$ classical Morrey space.)  Since $ W \equiv 1 $, $ d=\tilde d = 0 $. Let $s>\tilde d / p'  = 0$,
					$1\le q\le\infty$. Let $\sigma\in C_{*}^{\ell}S_{1,\delta}^{m}$ with
					$m\in\mathbb{R},$ $\delta\in[0,1]$ and $l>s$. Then Theorem \ref{non-reg} coincides with \cite[Theorem 3.7]{S09}.
				\end{rem}




\begin{thebibliography}{1}

	
\bibitem{BGX252}
T. Bai, P. Guo, and J. Xu,
{\it Bourgain-Morrey-Lorentz spaces and operators on them},
arXiv preprint
{arXiv}:2505.19130,  2025.

\bibitem{BGX25}
T. Bai, P. Guo, and J. Xu,
{\it The preduals of Banach space valued Bourgain-Morrey spaces},
Ann. Funct. Anal.
{\bf 16} (2025), no. 4, https://doi.org/10.1007/s43034-025-00458-w.


\bibitem{BGX253}
T. Bai, P. Guo, and J. Xu,
{\it Weighted {Bourgain}-{Morrey}-{Besov} type and {Triebel}-{Lizorkin}
	type spaces associated with operators}, 
Dissertationes Math., to appear.


\bibitem{BX242}
T. Bai, and J. Xu,
{\it Non-regular pseudo-differential operators on matrix weighted
	{Besov}-{Triebel}-{Lizorkin} spaces},
{J. Math. Study}, 
{\bf 57} (2024),
no. 1, 84--100.


\bibitem{BX243}
T. Bai, and J. Xu,
{\it Precompactness in matrix weighted {Bourgain}-{Morrey} spaces},
Filomat, 
{\bf 39} (2025), no. 18, 6261--6280. 


\bibitem{BX24}
T. Bai, and J. Xu,
{\it Pseudo-differential operators on matrix weighted
	{Besov}-{Triebel}-{Lizorkin} spaces},
Bull. Iran. Math. Soc.
{\bf 50} (2024),
no. 3, 26 pp.


\bibitem{BX25}
T. Bai, and J. Xu,
{\it Weighted {Bourgain}-{Morrey}-{Besov}-{Triebel}-{Lizorkin} spaces
	associated with operators},
Math. Nachr.
{\bf 298} (2025),
no. 3, 886--924.


\bibitem{B88}
G. Bourdaud,
{\it Une alg{\`e}bre maximale d'op{\'e}rateurs pseudo-diff{\'e}rentiels.
	({A} maximal algebra of pseudo differential operators)},
Commun. Partial Differ. Equations
{\bf 13} (1988),
no. 9, 1059--1083.



\bibitem{Bou91}
J. Bourgain,
{\it  On the restriction and multiplier problems in {{\(R^ 3\)}}. In {Geometric aspects of functional analysis. Proceedings of the
		Israel seminar (GAFA) 1989-90}}, pages 179--191,
Springer-Verlag, Berlin etc., 1991



\bibitem{BCYY25}
F. Bu, Y. Chen, D. Yang, and W. Yuan,
{\it Maximal {function} and {atomic} {characterizations} of
	{matrix}-{weighted} {Hardy} {spaces} with {their} {applications} to
	{boundedness} of {Calder{\'o}n}--{Zygmund} {operators}},
arXiv preprint {arXiv}:2501.18800, 2025. 


\bibitem{BHYY23}
F. Bu, T. Hyt{\"o}nen, D. Yang, and W. Yuan,
{\it New {characterizations} and {properties} of {matrix} ${A_\infty}$
	{weights}},
Acta. Math. Sin.-English Ser. (2025),
https://doi.org/10.1007/s10114-025-5143-9.



\bibitem{BHYY3}
F. Bu, T. Hyt{\"o}nen, D. Yang, and W. Yuan,
{\it Matrix-weighted {Besov}-type and {Triebel}-{Lizorkin}-type spaces.
	{III}: {Characterizations} of molecules and wavelets, trace theorems, and
	boundedness of pseudo-differential operators and {Calder{\'o}n}-{Zygmund}
	operators},
Math. Z.
{\bf 308} (2024),
no. 2, Paper No. 32, 67.



\bibitem{BHYY25}
F. Bu, T. Hyt{\"o}nen, D. Yang, and W. Yuan,
{\it Besov--Triebel--Lizorkin-type {spaces} with {matrix} ${A_\infty}$
	{weights}},
Sci. China Math.
(2025),
https://doi.org/10.1007/s11425-024-2385-x.



\bibitem{BHYY1}
F. Bu, T. Hyt{\"o}nen, D. Yang, and W. Yuan,
{\it Matrix-weighted {Besov}-type and {Triebel}-{Lizorkin}-type spaces.
	{I}: {{\(A_p\)}}-dimensions of matrix weights and {{\(\varphi\)}}-transform
	characterizations},
Math. Ann.
{\bf 391} (2025),
no. 4, 6105--6185.



\bibitem{BHYY2}
F. Bu, T. Hyt{\"o}nen, D. Yang, and W. Yuan,
{\it Matrix-weighted {Besov}-type and {Triebel}-{Lizorkin}-type spaces.
	{II}: {Sharp} boundedness of almost diagonal operators},
J. Lond. Math. Soc., II. Ser.
{\bf 111} (2025),
no. 3, 59 pp.



\bibitem{BYY23}
F. Bu,  D. Yang, and W. Yuan,
{\it  Real-variable characterizations and their applications of
	matrix-weighted {Besov} spaces on spaces of homogeneous type},
Math. Z.
{\bf 305} (2023),
no. 1, Paper No. 16, 81.



\bibitem{BYYZ25}
F. Bu,  D. Yang,  W. Yuan, and M. Zhang,
{\it Matrix-weighted {Besov}-{Triebel}-{Lizorkin} {spaces} of {optimal}
	{scale}: {Real}-{variable} {characterizations}, {invariance} on {integrable}
	{index}, and {Sobolev}-{type} {embedding}},
arXiv preprint {arXiv}:2505.02136, 2025.


\bibitem{C64}
A.~P. Calder{\'o}n,
{\it Intermediate spaces and interpolation, the complex method},
Stud. Math.
{\bf 24} (1964),
113--190.


\bibitem{CM78}
R.~R. Coifman, and Y. Meyer,
{\em Au del{\`a} des op{\'e}rateurs pseudo-diff{\'e}rentiels},
volume~57 of {\em Ast{\'e}risque}
Soci{\'e}t{\'e} Math{\'e}matique de France (SMF), Paris, 1978.




\bibitem{CO22}
M. Congo, and M.~F. Ouedraogo,
{\it Boundedness of nonregular pseudo-differential operators on variable
	exponent Triebel-Lizorkin-Morrey spaces},
Eur. J. Pure Appl. Math.
{\bf 15} (2022),
47--63,


\bibitem{CMR16}
D. Cruz-Uribe, K. Moen, and S., Rodney,
{\it Matrix {$\mathcal A_p$} weights, degenerate {Sobolev} spaces, and
	mappings of finite distortion},
J. Geom. Anal.
{\bf 26} (2016),
no. 4, 2797--2830.


\bibitem{Db88}
I. Daubechies,
{\it Orthonormal bases of compactly supported wavelets},
Commun. Pure Appl. Math.
{\bf 41} (1988),
no. 7, 909--996.

\bibitem{Db92}
I. Daubechies,
{\it Ten Lectures on Wavelets},
CBMS-NSF Regional Conference Series in Applied Mathematics,
vol. 61,
SIAM,
Philadelphia, PA,
1992.



\bibitem{D97}
P. Dintelmann,
{\it On the boundedness of pseudo-differential operators on weighted
	{Besov}-{Triebel} spaces},
Math. Nachr.
{\bf 183} (1997),
43--53.


\bibitem{DH19}
D. Drihem, and W.  Hebbache,
{\it Continuity of non-regular pseudodifferential operators on variable
	{Triebel}-{Lizorkin} spaces},
Ann. Pol. Math.
{\bf 122} (2019),
no. 3, 233--248.



\bibitem{FJ90}
M. Frazier, and B. Jawerth,
{\it  A discrete transform and decompositions of distribution spaces},
J. Funct. Anal.
{\bf 93} (1990),
no. 1, 34--170.


\bibitem{FJW91}
M. Frazier,  B. Jawerth, and G. Weiss,
{\it Littlewood-{Paley} theory and the study of function spaces},
volume~79 of {\em Reg. Conf. Ser. Math.}
American Mathematical Society,  Providence, RI, 1991.


\bibitem{FrazierR04}
M. Frazier and S. Roudenko,
{\it  Matrix-weighted {Besov} spaces and conditions of {{\(A_p\)}} type for
	{{\(0 < p \leq 1\)}}},
Indiana Univ. Math. J.
{\bf 53} (2004),
no. 5, 1225--1254.


\bibitem{FraRou19}
M. Frazier, and S. Roudenko,
{\it Littlewood-{Paley} theory for matrix-weighted function spaces}.
Math. Ann.
{\bf 380} (2021),
no. 1-2, 487--537.


\bibitem{Goldberg03}
M. Goldberg,
{\it Matrix {{\(A_p\)}} weights via maximal functions},
Pac. J. Math.
{\bf 211} (2003),
no. 2, 201--220.


\bibitem{G14}
L. Grafakos,
{\em Classical {Fourier} analysis}, volume 249 of {\em Grad. Texts
	Math.},
Springer, New York,  2014.



\bibitem{HNS23}
N. Hatano, T. Nogayama, Y. Sawano, and D.~I. Hakim,
{\it Bourgain-{Morrey} spaces and their applications to boundedness of
	operators},
J. Funct. Anal.
{\bf 284} (2023),
no. 1, 52 pp.


\bibitem{HLY23}
P. Hu, Y. Li, and D. Yang,
{\it Bourgain-Morrey spaces meet structure of {Triebel}-{Lizorkin} spaces},
Math. Z.
{\bf 304} (2023),
no. 1, 49 pp.


\bibitem{IST15}
T. Izumi, Y, Sawano, and H, Tanaka,
{\it Littlewood-{Paley} theory for {Morrey} spaces and their preduals},
Rev. Mat. Complut.
{\bf 28} (2015),
no. 2, 411--447.


\bibitem{joh}
J. Johnsen,
{\it Domains of pseudo-differential operators: a case for the
	{Triebel}-{Lizorkin} spaces},
J. Funct. Spaces Appl.
{\bf 3} (2005),
no. 3, 263--286.


\bibitem{LM87}
P.~G. Lemari{\'e}, and Y.~Meyer,
{\it Ondelettes et bases hilbertiennes. ({Wavelets} and {Hilbert} bases)},
Rev. Mat. Iberoam.
{\bf 2} (1987),
no. 1-2, 1--18.


\bibitem{LYY24}
Z. Li, D. Yang, and W. Yuan,
{\it Matrix-weighted {Besov}-{Triebel}-{Lizorkin} spaces with logarithmic
	smoothness},
Bull. Sci. Math.
{\bf 193} (2024),
Paper No. 103445,
54 pp.


\bibitem{Mar96}
J.~Marschall.
{\it Nonregular pseudo-differential operators},
Z. Anal. Anwend.
{\bf 15} (1996),
no. 1, 109--148.


\bibitem{Mar88}
J. Marschall,
{\it Pseudo-differential operators with coefficients in {Sobolev} spaces},
Trans. Am. Math. Soc.
{\bf 307} (1988),
no. 1,  335--361.


\bibitem{Mar91}
J. Marschall,
{\it  Weighted parabolic {Triebel} spaces of product type. {Fourier}
	multipliers and pseudo-differential operators},
Forum Math.
{\bf 3} (1991),
no. 5, 479--511.


\bibitem{M23}
S. Masaki,
{\it Two minimization problems on non-scattering solutions to
	mass-subcritical nonlinear Schr\"odinger equation},
arXiv preprint  arXiv:1605.09234,  2016.



\bibitem{MS18}
S.  Masaki, and J. Segata,
{\it Existence of a minimal non-scattering solution to the
	mass-subcritical generalized {Korteweg}-de {Vries} equation},
Ann. Inst. Henri Poincar{\'e}, Anal. Non Lin{\'e}aire,
{\bf 35} (2018),
no. 2, 283--326.


\bibitem{M03}
A.~L. Mazzucato,
{\it Besov-{Morrey} spaces: {Function} space theory and applications to
	nonlinear {PDE}},
Trans. Am. Math. Soc.
{\bf 355} (2003),
no. 4,  1297--1364.


\bibitem{MV98}
F.~Merle, and L.~Vega,
{\it  Compactness at blow-up time for {{\(L^2\)}} solutions of the critical
	nonlinear {Schr{\"o}dinger} equation in 2D},
Int. Math. Res. Not.
{\bf 1998} (1998),
no. 8, 399--425.


\bibitem{Mer87}
Y. Meyer,
{\it The uncertainty principle, {Hilbert} base and operator algebras},
S{\'e}min. {Bourbaki}, 38{\`e}me ann{\'e}e, {Vol}. 1985/86,
{Ast{\'e}risque} 145/146, 209-223, {Exp}. {No}. 662 (1987)., 1987.


\bibitem{md}
M. Moussai, and A. Djeriou,
{\it Boundedness of some pseudo-differential operators on generalized
	{Triebel}-{Lizorkin} spaces},
Analysis, M{\"u}nchen,
{\bf 31} (2011),
no. 1, 13--29.


\bibitem{MVV99}
A.~Moyua, A.~Vargas, and L.~Vega,
{\it Restriction theorems and maximal operators related to oscillatory
	integrals in {{\(\mathbf R^3\)}}},
Duke Math. J.
{\bf 96}  (1999),
no. 3, 547--574.


\bibitem{N25}
M. Nielsen,
{\it Matrix weighted {{\({{\alpha}} \)}}-modulation spaces},
Monatsh. Math.
{\bf 206} (2025),
no. 2, 419--448.


\bibitem{P16}
B.~J. Park,
{\it On the boundedness of pseudo-differential operators on
	{Triebel}-{Lizorkin} and {Besov} spaces},
J. Math. Anal. Appl.
{\bf 461} (2018),
no. 1, 544--576.


\bibitem{P19}
B.~J. Park,
{\it  Boundedness of pseudo-differential operators of type (0,0) on
	{Triebel}-{Lizorkin} and {Besov} spaces},
Bull. Lond. Math. Soc.
{\bf 51}  (2019),
no. 6, 1039--1060.


\bibitem{P20}
B.~J. Park,
{\it Sharp estimates for pseudo-differential operators of type
	{{\((1,1)\)}} on {Triebel}-{Lizorkin} and {Besov} spaces},
Stud. Math.
{\bf 250} (2020),
no. 2, 129--162.


\bibitem{R03}
S.  Roudenko,
{\it Matrix-weighted {Besov} spaces},
Trans. Am. Math. Soc.
{\bf 355} (2003),
no. 1, 273--314.


\bibitem{Roudenko04}
S.  Roudenko,
{\it Duality of matrix-weighted {Besov} spaces},
Stud. Math.
{\bf 160} (2004),
no. 2, 129--156.

\bibitem{ru}
T. Runst,
{\it Pseudo-differential operators of the `exotic' class {{\(L^
			o_{1,1}\)}} in spaces of {Besov} and {Triebel}-{Lizorkin} type},
Ann. Global Anal. Geom.
{\bf 3} (1985),
no. 1, 		13--28.


\bibitem{sa}
S. Sato,
{\it Non-regular pseudo-differential operators on the weighted
	{Triebel}-{Lizorkin} spaces},
T{\^o}hoku Math. J. (2)
{\bf 59} (2007),
no. 3, 323--339.


\bibitem{S09}
Y. Sawano,
{\it A note on {Besov}-{Morrey} spaces and {Triebel}-{Lizorkin}-{Morrey}
	spaces},
Acta Math. Sin., Engl. Ser.
{\bf 25} (2009),
no. 8, 1223--1242.


\bibitem{S18}
Y. Sawano,
{\em Theory of {Besov} spaces}, volume~56 of {\em Dev. Math.}
Springer, Singapore, 2018.


\bibitem{SFH20}
Y. Sawano, G. Di~Fazio, and D.~I. Hakim,
{\em Morrey spaces. {Introduction} and applications to integral
	operators and {PDE}'s. {Volume} {I}},
CRC Press, Monogr. Res. Notes Math. Boca Raton, FL, 2020.


\bibitem{ST09}
Y. Sawano, and H. Tanaka,
{\it Predual spaces of {Morrey} spaces with non-doubling measures},
Tokyo J. Math.
{\bf 32} (2009),
no.2, 471--486. 


\bibitem{tor}
R.~H. Torres,
{\em Boundedness results for operators with singular kernels on
	distribution space}, volume 442 of {\em Mem. Am. Math. Soc.}
American Mathematical Society, Providence, RI,  1991.


\bibitem{TV97}
S.~Treil, and A.~Volberg,
{\it Wavelets and the angle between past and future},
J. Funct. Anal.
{\bf 143} (1997),
no. 2, 269--308.


\bibitem{T83}
H. Triebel,
{\em Theory of function spaces}, volume~78 of {\em Monogr. Math.,
	Basel},
Cham, Birkh{\"a}user,  1983.


\bibitem{V97}
A.~Volberg,
{\it Matrix {{\(A_{p}\)}} weights via {{\(S\)}}-functions},
J. Am. Math. Soc.
{\bf 10} (1997)
no. 2,  445--466.


\bibitem{WYZ22}
Q.~Wang, D. Yang, and Y. Zhang,
{\it Real-variable characterizations and their applications of
	matrix-weighted {Triebel}-{Lizorkin} spaces},
J. Math. Anal. Appl.
{\bf 529} (2024),
no. 1, 37 pp.


\bibitem{WGX24}
S. Wang, P. Guo, and J. Xu,
{\it Precompact {sets} in {matrix} {weighted} {Lebesgue} {spaces} with
	{variable} {exponent}},
Georgian Math. J.
{\bf 32} (2025),
no. 6, 1071--1083.



\bibitem{WGX25}
S. Wang, P. Guo, and J. Xu,
{\it Embedding and duality of matrix-weighted modulation spaces},
Taiwanese J. Math.
{\bf 29} (2025),
no. 1, 171--187.


\bibitem{w2}
M.~V. Wong,
{\em An introduction to pseudo-differential operators.}, volume~6 of
{\em Ser. Anal. Appl. Comput.}
World Scientific, Singapore, 2014.

\bibitem{w1}
M.~W. Wong,
{\it Fredholm pseudo-differential operators on weighted {Sobolev} spaces},
Ark. Mat.
{\bf 21} (1983),
271--282.


\bibitem{YYZ25}
D. Yang, W. Yuan, and M. Zhang,
{\it Matrix-weighted {Besov}--{Triebel}--{Lizorkin} {spaces} of {optimal}
	{scale}: {Boundedness} of {pseudo}-{differential}, {trace}, and
	{Calder{\'o}n}--{Zygmund} {operators}},
arXiv preprint {arXiv}:2504.19060, 2025.

\bibitem{YYZ252}
D. Yang, W. Yuan, and Z. Zeng,
{\it Variable matrix-weighted {Besov} spaces},
arXiv preprint {arXiv}:2509.07786, 2025.

\bibitem{ZYZ24}
Y. Zhang, D. Yang, and Y. Zhao,
{\it Grand {Besov}-{Bourgain}-{Morrey} spaces and their applications to
	boundedness of operators},
Anal. Math. Phys.
{\bf 14} (2024),
no. 4,  Paper No. 79, 58 pp.


\bibitem{ZSTYY23}
Y. Zhao, Y. Sawano, J. Tao, D. Yang, and W. Yuan,
{\it Bourgain-Morrey spaces mixed with structure of {Besov} spaces},
Proc. Steklov Inst. Math.
{\bf 323} (2023),
244--295.


\end{thebibliography}
\end{document}